\documentclass{article}

\usepackage{a4wide}
\usepackage{cite}
\usepackage[utf8]{inputenc}
\usepackage[T1]{fontenc}

\usepackage[colorlinks=true, pdfstartview=FitV, linkcolor=blue,citecolor=blue, urlcolor=blue]{hyperref}

\usepackage[french,english]{babel}
\usepackage{amsmath}
\usepackage{amscd}
\usepackage{amssymb}
\usepackage{amsrefs}
\usepackage{amsthm}
\usepackage{cases}
\usepackage{framed}
\usepackage{graphicx}
\usepackage{latexsym}
\usepackage{color}
\usepackage{array,multirow,makecell}
\usepackage[capitalise]{cleveref}

\usepackage{enumerate}

\usepackage{bbm}

\definecolor{purple}{RGB}{100, 0, 200}

\setcellgapes{1pt}
\makegapedcells

\theoremstyle{plain} 
\newtheorem{thm}{Theorem}[section]
\newtheorem{prop}[thm]{Proposition}
\newtheorem{lem}[thm]{Lemma} 
\newtheorem{cor}[thm]{Corollary} 
\numberwithin{equation}{section} 

\theoremstyle{definition}

\theoremstyle{remark} 
\newtheorem{rmk}[thm]{Remark}

\newcommand{\R}{\mathbb{R}}

\newcommand{\ep}{\varepsilon}

\newcommand{\p}{\partial}
\newcommand{\be}{\begin{equation}}
\newcommand{\ee}{\end{equation}}
\newcommand{\ba}{\begin{aligned}}
	\newcommand{\ea}{\end{aligned}}

\newcommand{\bv}{\bar{\mathpzc{v}}}
\newcommand{\bu}{\bar{\mathpzc{u}}}
\newcommand{\bw}{\bar{\mathpzc{w}}}
\newcommand{\bp}{\bar{\mathpzc{p}}}
\newcommand{\cA}{\mathcal A}
\newcommand{\cE}{\mathcal E}
\newcommand{\tv}{\mathpzc{v}}
\newcommand{\tu}{\mathpzc{u}}
\newcommand{\tw}{\mathpzc{w}}
\newcommand{\tp}{\mathpzc{p}}
\newcommand{\tx}{\widetilde{x}}
\newcommand{\ty}{\widetilde{y}}
\newcommand{\tz}{\widetilde{z}}

\newcommand{\bbv}{\overline{\overline{v}}}

\newcommand{\uuv}{\underline{\underline{v}}}
\newcommand{\ug}{\underline{\underline{g}}}
\newcommand{\cT}{\mathcal T}
\newcommand{\dv}{D_v}
\newcommand{\du}{D_u}

\DeclareFontFamily{OT1}{pzc}{}
\DeclareFontShape{OT1}{pzc}{m}{it}{<-> s * [1.10] pzcmi7t}{}
\DeclareMathAlphabet{\mathpzc}{OT1}{pzc}{m}{it}

\title{Local and global well-posedness of one-dimensional free-congested equations}
\author{Anne-Laure Dalibard\footnote{Sorbonne Université, Université Paris-Diderot SPC, CNRS,  Laboratoire Jacques-Louis Lions, LJLL, F-75005 Paris \& ~École Normale Supérieure, Université PSL, Département de Mathématiques et applications,
F-75005, Paris, France; dalibard@ann.jussieu.fr} \ and Charlotte Perrin\footnote{Aix Marseille Univ, CNRS, Centrale Marseille, I2M, Marseille, France; charlotte.perrin@univ-amu.fr}}

\begin{document}

\maketitle

\begin{small}

\begin{abstract}
    This paper is dedicated to the study of a one-dimensional congestion model, consisting of two different phases. In the congested phase, the pressure is free and the dynamics is incompressible, whereas in the non-congested phase, the fluid obeys a pressureless compressible dynamics.
    
    We investigate the Cauchy problem for initial data which are small perturbations in the non-congested zone of travelling wave profiles. We prove two different results. First, we show that for arbitrarily large perturbations, the Cauchy problem is locally well-posed in weighted Sobolev spaces. The solution we obtain takes the form $(v_s, u_s)(t, x-\tx(t))$, where $x<\tx(t)$ is the congested zone and $x>\tx(t)$ is the non-congested zone. The set $\{x=\tx(t)\}$ is a free surface, whose evolution is coupled with the one of the solution. Second, we prove that if the initial perturbation is  sufficiently small
    , then the solution is global. 
    This stability result relies on coercivity properties of the linearized operator around a travelling wave, and on the introduction of a new unknown which satisfies better estimates than the original one.
    In this case, we also prove that travelling waves are asymptotically stable.

\end{abstract}

	
\end{small}

\section{Introduction}

The purpose of this paper is to construct solutions of the fluid system
\begin{subnumcases}{\label{eq:NS-0}}
\partial_t \tv - \partial_x \tu = 0, \label{eq:NS-0-v}\\
\partial_t \tu + \partial_x \tp - \mu \partial_x\left(\dfrac{1}{\tv}\partial_x \tu \right) = 0, \\
\tv \geq 1, \quad (\tv-1) \tp = 0, \quad \tp \geq 0,\label{eq:NS-0-excl}\\
\tv_{|t=0}= \tv^0,\quad \tu_{|t=0}=\tu^0,
\end{subnumcases}
for a large class of initial data $(\tv^0, \tu^0)$, with 
\[
(\tv,\tu)(t,x) \underset{x \rightarrow \pm \infty}{\longrightarrow} (v_\pm,u_\pm).
\]
The variable $\tv$ represents the specific volume of the fluid, that is the inverse of the density, while $\tu$ denotes its velocity.
Equations~\eqref{eq:NS-0} are actually a reformulation in Lagrangian coordinates of the constrained Navier-Stokes system introduced in~\cite{BPZ2014}.
We further assume that
\begin{equation}\label{lim:vup}
v_- = 1, \quad v_+ > 1, \quad u_- > u_+.
\end{equation}	
We do not impose a limit condition on the pressure variable $\tp$ which is actually linked to $(\tv,\tu)$.
The pressure is indeed seen as a Lagrange multiplier associated to the constraint
$\partial_x\tu \geq 0$ on $\{\tv=1\}$.

\bigskip
Let us recall a few facts about system \eqref{eq:NS-0}, and be more specific about the contents of the present paper. 
It describes a partially congested system, consisting of two different phases. In the phase $\{\tv>1\}$ (non-congested phase where $\rho =1/\tv < 1$), the pressure vanishes and the dynamic is compressible. In the phase $\{\tv=1\}$ (congested phase), the pressure is activated and the dynamic is incompressible. \\
From the modelling point of view, the system~\eqref{eq:NS-0} may apply in various contexts. 
A first example is given by the dynamics of two-phase flows in presence of pure-phase (or saturation) regions as described by Bouchut {\it et al.} in~\cite{bouchut2000}.
In this context, the constrained variable is the volume fraction which has to stay between $0$ and $1$, the extremal values corresponding to the pure-phase states.
Another domain of application of Equations~\eqref{eq:NS-0} is the modeling of collective motion (like crowds or vehicular traffic, see for instance~\cites{berthelin2017,degond2011,maury2011}). 
There, the maximal density limit (or equivalently the minimal specific volume) corresponds to a microscopic \emph{packing constraint}, constraint which is locally achieved when the agents are in contact. 
In this framework, models of type~\eqref{eq:NS-0} are called \emph{hard congestion models} (see~\cite{maury2012}).
Finally, let us also mention the connections between~\eqref{eq:NS-0} and models for wave-structure interactions developed in the very recent years by Lannes~\cite{lannes2017} and Godlewski et al.~\cite{godlewski2018}. 
A similar constraint to~\eqref{eq:NS-0-excl} can be indeed formulated to express the two possible states of the flow: \emph{pressurized} in the ``interior'' domain at the contact with the structure (the height being then constrained by the structure), \emph{free} in the exterior domain.

\bigskip
It can be easily checked that there exist travelling wave solutions $(\bu, \bv)(x-st)$ for \eqref{eq:NS-0}. These were constructed in \cite{DP} and their main features are recalled below in Lemma \ref{lem:TW}. 
They consist of a congested zone for $x-st<0$, and of a non-congested zone for $x-st>0$, in which $\bv$ is the solution of a logistic equation (see Figure~\ref{fig:profile} below). 
The setting of the current paper is the following: we consider initial data $(\tu^0, \tv^0)$ which are perturbations of the travelling wave profiles $(\bu,\bv)$ in the non-congested zone $x>0$ only. In other words, $\tu^0(x)=\bu(x)=u_-$ and $\tv^0(x)=v_-=1$ for $x<0$. Under compatibility conditions on the initial data, we prove that there exists a local strong solution of \eqref{eq:NS-0}. Furthermore, this solution is global provided 
the initial perturbation is sufficiently small.

Originally, the study of density constrained fluid systems begins with the proof of the existence of global weak solutions by Lions and Masmoudi~\cite{lions1999} for the multi-dimensional free-congested Navier-Stokes equations.
The result is achieved via a penalty approach: the equations are approximated by a fully compressible Navier-Stokes system in which the maximal density constraint has been relaxed and the (compressible) pressure plays the role of the penalty function.
Later, the same type result was obtained in~\cite{perrin2015} by means of a \emph{soft congestion approximation} which consists of a fully compressible Navier-Stokes system with a singular pressure law blowing up as the density approaches $1$. 
Contrary to the former study, the maximal density constraint is satisfied even at the approximate level which can be useful from the numerical point of view (see for instance~\cites{degond2011,PS}), and relevant from the physical point of view if one thinks for instance of the influence of repulsive social forces in collective motion.
If the constructed weak solutions can theoretically couple the  free and congested states, the previous existence results do not give any information about the congested domain and about the time evolution of its boundary. 
In other words, it is not clear whether free and congested states actually co-exist within a given weak solution.\\
The existence of more regular solutions to~\eqref{eq:NS-0} for initial mixed compressible-incompressible data is, up to our knowledge, a largely open problem. 
Comparatively to other compressible-incompressible free boundary problems like the ones studied in~\cite{shibata2016} or in~\cite{colombo2016}, we have to handle the fact that the interface between the free (compressible) domain and the congested (incompressible) domain is not closed, {i.e.} matter passes through the boundary and the volume of the congested region evolves with time.
Besides, the identification of appropriate transmission conditions across the interface is a non-trivial issue which is for instance raised in~\cite{bresch2017} by Bresch and Renardy.
In the hyperbolic framework of wave-structure interactions (WSI), the recent study of Iguchi and Lannes~\cite{iguchi2019} provides a one-dimensional existence result in $H^m$, $m \geq 2$ (regularity of the solution in the ``exterior'' domain, which is called the free domain in our framework).
This result has been then extended to the dispersive Boussinesq case in~\cites{bresch2019,beck2021freely}, and to the two-dimensional axisymmetric case in~\cite{bocchi2020}. 
Finally, the study~\cite{maity2019} includes viscosity effects, still in an axisymmetric configuration.
The congestion problem~\eqref{eq:NS-0} is similar to the viscous WSI problem~\cite{maity2019} in the sense that it can be formulated as a mixed initial-boundary value problem with a implicit coupling of the (parabolic) PDEs describing the dynamics in the free/exterior zone with a nonlinear ODE. 
This ODE (see~\eqref{EDO-tx} below) represents the evolution of the free-congested interface in~\eqref{eq:NS-0}, while in~\cite{maity2019} the ODE models the vertical motion of the structure (the contact line between the fluid and the structure is there constant due to the axisymmetric hypothesis).

\medskip
As said before, partially congested propagation fronts $(\bu, \bv)(x-st)$ for the viscous system~\eqref{eq:NS-0} have already been identified in the previous study~\cite{DP}.
These traveling waves are such that $\bv, \bu$ and the (effective) flux $\bp-\mu\partial_x \bu$ are continuous across the free-congested interface.
But the core of~\cite{DP} is devoted to the analysis of approximate traveling waves $(\tv_\ep,\tu_\ep)_{\ep> 0}$ solutions to the \emph{soft congestion approximation} of~\eqref{eq:NS-0}. 
Under some smallness condition (quantified in terms of $\ep$) on the initial perturbation, the profiles $(\tv_\ep,\tu_\ep)_{\ep> 0}$ are shown to be asymptotically stable.
This result is achieved by means of weighted energy estimates, it relies on the use of integrated variables and a reformulation of the system in the variables $(\tv_\ep,\tw_\ep)$ where $\tw_\ep:= \tu_\ep -\mu \partial_x \ln\tv_\ep$ is the so-called \emph{effective velocity}.
Roughly speaking, the use of this new velocity induces regularization effects on the specific volume $\tv_\ep$, effects previously identified (among others) by Shelukhin~\cite{shelukhin1984}, Bresch, Desjardins~\cite{bresch2006}, Vasseur~\cite{vasseur2016}, Haspot~\cite{haspot2018}.
The use of the integrated variables is related to the structure of the dissipation and source terms.
As detailed in Section~\ref{sec:unif-est-g} below, it enables the derivation of uniform-in-time energy estimates on the solution. \\
Unfortunately, as $\ep \to 0$ the smallness condition on the initial perturbations degenerates and no stability can be inferred directly for the limit profiles $(\bv,\bu)$.

\medskip
The present study contains three main results related to initial perturbations of the profile $(\bv,\bu)$ in the free zone.
We demonstrate a local well-posedness result for large data as well as a global result for small initial perturbations.
Finally, we prove the asymptotic stability the profile $(\bv,\bu)$. 
Similarly to the case $\ep > 0$ described above, our analysis is based on energy estimates and the use of the effective velocity to rewrite the equation on the specific volume~\eqref{eq:NS-0-v} as a parabolic equation with a nonlinear diffusion. 
One significant difference between the two studies is that $\tw$ satisfies in the present case a pure transport equation in the free domain due to the absence of pressure in that region.
As a consequence, the two equations in $\tv$ and $\tw$ can be decoupled, which simplifies somehow the dynamics and the derivation of estimates on $\tv$.
One the other hand, and more importantly, the analysis (in particular the global-in-time existence proof) is made here more difficult as a result of the dynamics of the new free boundary, {\it i.e.} the interface between the free and congested domains, which is coupled to the dynamics of $(\tu, \tv)$ itself through a continuity condition imposed at the interface (see~\eqref{eq:cont-v-i}-\eqref{eq:tx-0}).
Similarly to the WSI problem tackled by Iguchi and Lannes~\cite{iguchi2019}, we introduce a new variable which allows us to have a good control of the motion of the interface.

These results are presented in the next section.

\section{Main results and strategy}
\label{sec:results}
As explained in the introduction, we will construct solutions in the vicinity of the travelling wave solution $(\bu,\bv)(x-st).$ Hence we first recall some features of the profile $(\bu, \bv)$:

\begin{lem}[\cite{DP}] Assume that $u_->u_+$, $v_+>1$, and let
\[
s:=\frac{u_--u_+}{v_+-1}.
\]
Then there exists a unique (up to a shift) travelling wave solution of \eqref{eq:NS-0}. This travelling wave propagates at speed $s$ and is of the form $(\bu, \bv)(x-st)$. Furthermore, 
\[
 \bv(x)= \begin{cases}
1&\text{ if }x\leq 0,\\
\displaystyle\frac{v_+}{1 + (v_+-1) \exp(-sv_+x/\mu)}&\text{ if }x> 0,
\end{cases}.
\]
\[
\bu=u_+ + s v_+ - s\bv = u_- + s v_- - s\bv.
\]
In the zone $x<0$, the pressure is constant and equal to $p_-= s^2 (v_+ - 1)$. Eventually, introducing the effective velocity $\bw= \bu - \mu \p_x \ln \bv$, we have
\[
\bw(x)=u_- \mathbf 1_{x<0} + u_+ \mathbf 1_{x>0}.
\]

\label{lem:TW}
\end{lem}

The profile is represented in Figure~\ref{fig:profile}.
\begin{figure}[h]
\begin{center}
\includegraphics[scale=0.35]{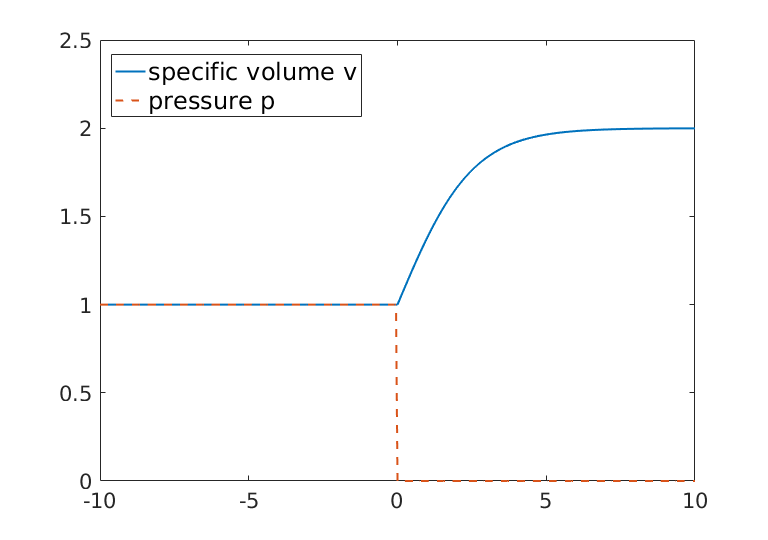}
\includegraphics[scale=0.35]{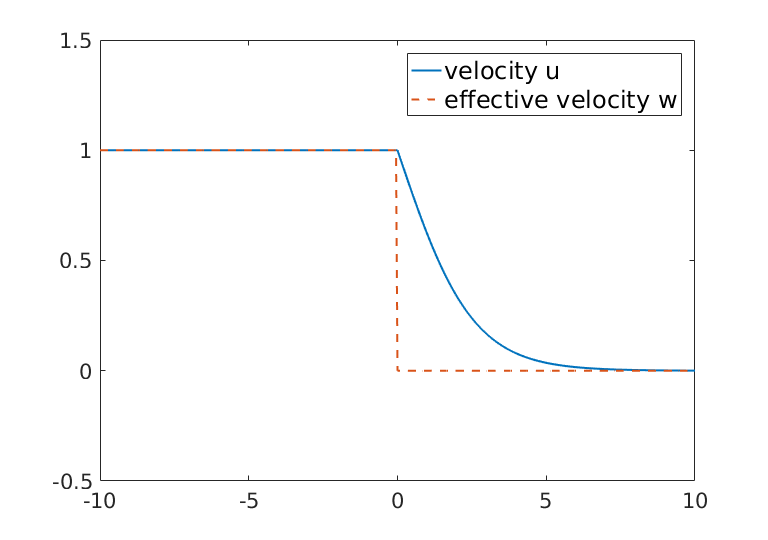}
\end{center}
\caption{ On the left: the profiles $\bv$ and $\bp$, on the right: the profiles $\bu$ and $\bw$. }
\label{fig:profile}
\end{figure}
Let us now explicit our assumptions on the initial data~$(\tu^0, \tv^0)$:
\begin{enumerate}[(H1)]
\item{\it  Partially congested initial data: }$(u^0, v^0)\in (\bu, \bv) + L^1(\R)$, and such that $u^0(x)=\bu(x)=u_-$, $v^0(x)=\bv(x)=1$ if $x<0$;

\item {\it Regularity: }$\mathbf 1_{x>0}(\tu^0-\bu, \tv^0-\bu)\in  H^3(\R_+)$ and $\mathbf 1_{x>0} \sqrt{x}\p_x^k \tw^0\in L^2(\R_+) $ for $k=1,2$, where $\tw^0 = \tu^0-\mu \p_x \ln \tv^0$;

\item {\it Compatibility:} $\tu^0(0^+)=u_-$, $\tv^0(0^+)=1$, and
\begin{equation}
\left[-\frac{(\p_x \tu^0)^2}{\p_x \tv^0} - \mu \p_x \tv^0 \p_x u^0 + \mu \p_x^2 \tu^0\right]_{|x=0^+}=0;
\end{equation}

\item{\label{item:no-deg}} {\it Non-degeneracy:}
$\p_x \tv^0(0^+)>0$, $\p_x \tu^0(0^+)<0$ and $\tv^0(x)>1$ for $x>0$;

\item {\it Decay: }$\mathbf 1_{x>0}V^0\in L^2(\R_+)$, where $V^0(x):=-\int_x^\infty (v^0-\bv)$, and  $\mathbf 1_{x>0} (1+ \sqrt{x}) W^0 \in L^2(\R_+) $, where $W^0:=-\int_x^\infty (\tw^0-u_+)$.

\end{enumerate}

Under these assumptions, the solution of \eqref{eq:NS-0} associated with $(u^0, v^0)$, if it exists, will not be a travelling wave. However, it is reasonable to expect such a solution to be congested in a zone $x<\tx(t)$, and non-congested in a zone $x>\tx(t)$, where the free boundary $x=\tx(t)$ is an unknown of the problem. 
The dynamics of the interface is actually encoded in the continuity condition that we impose on the specific volume, namely:
\begin{equation}\label{eq:cont-v-i}
v(t, \tx(t)) = 1 \qquad \forall\ t.
\end{equation}
By differentiating the relation with respect to time, we then get:
\begin{equation}\label{eq:tx-0}
\tx'(t) = - \dfrac{\partial_t v_{|x=\tx^+}}{\partial_x v_{|x=\tx^+}}.
\end{equation}
The free boundary problem~\eqref{eq:NS-0} differs from ``classical'' free boundary problems associated with a kinematic condition at the interface. 
In that latter case, the regularity of  $\tx'$ is the same as the regularity of the solution at the boundary, while there is here a loss of one derivative for $\tx'$ with respect to the solution $v$. 
The boundary condition~\eqref{eq:tx-0} is \emph{fully nonlinear} (see the study of Iguchi, Lannes~\cite{iguchi2019} ).
Eventually, using the mass equation, we get the dynamics of $\tx$:
\be\label{EDO-2}
\tx'(t) = - \dfrac{\p_x u_{|x=\tx^+}}{\p_x v_{|x=\tx^+}}
.\ee
We actually prove the following result:

\begin{thm}[Local in time result]
Let  $(\tu^0, \tv^0)$ satisfying the assumptions (H1)-(H5).

Then there exists $T>0$ and $\tx \in H^2_\text{loc}([0,T[)$, with $\tx(0) = 0$, $\tx'(0) = -[\frac{\partial_x\tu^0}{\partial_x \tv^0}]_{x=0}$, such that~\eqref{eq:NS-0} has a unique maximal solution $(\tu,\tv)$ of the form $(\tu,\tv)(t,x)= (u_s, v_s)(t, x-\tx(t))$ on the interval $[0, T[$, where $u_s(t,x)= u_-$, $v_s(t,x)= 1$ and $p_s(t,x) = -\mu (\partial_x u_s)_{|x=0^+}$ for $x<0$. Furthermore, 
\begin{equation}
v_s(t,x) > 1 \quad \text{for all}~ t\in [0,T[, ~x > 0,
\end{equation}
and the solution $(u_s,v_s)$ has the following regularity in the free domain:
\begin{align}
v_s - \bv, \ u_s -\bu \in L^\infty([0,T[; H^3(\R_+)), \\
\partial_t(v_s -\bv), ~ \partial_t(u_s-\bu) 
\in L^\infty([0,T[;H^1(\R_+)) \cap L^2(]0,T[; H^2(\R_+)).
\end{align}
Eventually, the pressure in the congested domain satisfies
\begin{equation}
p_s \in H^1(0,T).
\end{equation}
\label{thm:main-loc}
\end{thm}

Our second result shows the global existence of the solution provided
the initial perturbation is small. 

\begin{thm}
Let  $(\tu^0, \tv^0)$ satisfying the assumptions (H1)-(H5), and let
\[\ba
\cE_0
& :=\|\tv^0 - \bv\|_{H^3(\R_+)}^2 + \|\tu^0-\bu\|_{H^3(\R_+)}^2 +\|\tw^0 - u_+\|_{L^2(\R_+)}^2 +\|V^0\|_{L^2(\R_+)}^2  \\
& \quad+ \|(1+\sqrt{x}) W^0\|_{L^2(\R_+)}^2+ \|(1+\sqrt{x}) \p_x\tw^0\|_{L^2(\R_+)}^2+ \|(1+\sqrt{x})\partial_x^2 \tw^0\|_{L^2(\R_+)}^2.\nonumber
\ea
\]
Assume moreover that
\begin{equation}
\mathbf{1}_{\{x > 0\}} \sqrt{x} \partial^3_x \tw^0 \in L^2(\R).
\end{equation}
Then, there exist constants $c_0,\delta_0 >0$ depending only on the parameters of the problem $s, \mu, v_+,u_\pm$, such that for all $\delta \in (0,\delta_0)$, if 
\begin{equation}\label{hyp:small-glob}
\cE_0 \leq c_0 \delta^2, \quad  \|(1+\sqrt{x})\partial^k_x \tw^0\|_{L^2(\R_+)} \leq c_0 \delta^{3/2}, \quad \forall\ \delta \in (0, \delta_0), ~ k=1,2,3,
\end{equation}
then
the solution $(\tx,u_s,v_s)$ is global, and 
\[
\|v_s-\bv\|_{L^\infty(\R_+, H^3(\R_+))} + \|u_s-\bu\|_{L^\infty(\R_+, H^3(\R_+))} + \|\tx'-s\|_{H^1(\R_+)}\leq C \delta.
\]
Furthermore
\begin{equation}
|\tx'(t) -s| + \sup_{x\in \R} |(u_s,v_s,p_s)(t,x) - (\bu, \bv,\bp)(x)| \longrightarrow 0 \quad \text{as}~ t\to +\infty,
\end{equation}
where $\bp(x) = p_- \mathbf{1}_{\{x <0\}}$.
\label{thm:main-glob}
\end{thm}

The strategy of proof is the following. We work in the shifted variable $x-\tx(t)$. Since $(u,v)$ is expected to be constant in $x-\tx(t)<0$, we only consider the system satisfied by $(u_s, v_s)$ in the positive half-line, which reads
\begin{subnumcases}{\label{eq:NS-shifted-1}}
\partial_t v_s - \tx'(t) \p_x v_s- \partial_x u_s = 0,\quad  t>0,\ x>0\label{eq:NS-shifted-1-v}\\
\partial_t u_s - \tx'(t) \p_x u_s - \mu \partial_x\left(\dfrac{1}{v_s}\partial_x u_s \right) = 0,\quad  t>0,\ x>0\label{eq:NS-shifted-1-u}\\
(v_s,u_s)_{|x=0}= ( 1 ,u_-), \quad \lim_{x\to \infty} (v_s(t,x),u_s(t,x)) = (v_+,u_+) \quad \forall \ t > 0.
\end{subnumcases}
We have already seen that the dynamics of $\tx$ is coupled with the dynamics of $(v_s, u_s)$ through~\eqref{EDO-2}. 
In order to construct a solution of \eqref{eq:NS-shifted-1}, it will be more convenient to modify the equation on $v_s$ in order to make the regularizing effects of the diffusion more explicit. Indeed, setting $w_s= u_s - \mu \p_x \ln v_s$, we find that equation \eqref{eq:NS-shifted-1-v} can be written as
\[
\partial_t v_s - \tx'(t) \p_x v_s- \mu\partial^2_x \ln v_s  = \partial_x w_s,\quad t>0,\ x>0.
\] 
Moreover,
\begin{equation}\label{eq:ws}
\p_t w_s - \tx'(t) \p_x w_s = 0,\quad t>0,\ x>0,
\end{equation}
therefore $w_s(t,x)= \tw^0 (x + \tx(t))$ for all $t>0, x>0$ provided $\tx'(t)>0$ for all $t>0$. In particular, letting $x\to 0^+$, we obtain
\be\label{EDO-1}
u_- - \mu \p_x v_{s|x=0^+}= \tw^0(\tx(t)).
\ee
Gathering \eqref{EDO-1} and \eqref{EDO-2} leads to
\be\label{EDO-tx}
\tx'(t) = -\mu  \frac{ \p_x u_{s|x=0^+}}{u_- - \tw^0(\tx(t))}.
\ee

Since $w_s(t,x)=\tw^0(\cdot + \tx(t))$, the equation on $v_s$ rewrites 
\be\label{eq:vs}\ba
\p_t v_s - \tx'(t) \p_x v_s - \mu \p_{x x} \ln v_s= \p_x \tw^0(x+ \tx(t)),\quad t>0, \ x>0,\\
v_{s|x=0}=1,\quad \lim_{x\to \infty} v_s(t,x)= v_+,\\
v_{s|t=0}=v^0.
\ea\ee
Thus we will build a solution $(\tx, v_s, u_s)$ of \eqref{EDO-tx}-\eqref{eq:vs}-\eqref{eq:NS-shifted-1-u} thanks to the following fixed point argument:
\begin{enumerate}
    \item For any given $\ty\in H^2_\text{loc}(\R_+)$, such that $\ty(0) = 0$ and $\ty'(0)= - \frac{\p_x \tu^0_{|x=0^+}}{\partial_x \tv^0_{|x=0^+}}$, we consider the solution $v$ of the equation
    \be\label{eq:vs-y}\ba
\p_t v - \ty'(t) \p_x v - \mu \p_{x x} \ln v = \p_x \tw^0(x+ \ty(t)),\quad t>0, \ x>0,\\
v_{|x=0}=1,\quad \lim_{x\to \infty} v(t,x)= v_+,\\
v_{|t=0}=v^0.
\ea\ee
    We prove that under suitable conditions on the initial data, there exists a unique solution $v \in \bv + L^\infty_\text{loc}(\R_+,H^1(\R_+))$, and we derive higher regularity estimates.
    \item We then consider the unique solution $u\in \bu + L^\infty_\text{loc}(\R_+, H^1(\R_+))$ of
    \be\label{eq:us-y}
    \ba
    \partial_t u - \ty'(t) \p_x u - \mu \partial_x\left(\dfrac{1}{v}\partial_x u\right) = 0\quad t>0, \ x>0,\\
    u_{|x=0}= u_-, \quad \lim_{x\to \infty} u(t,x)= u_+,\\
    u_{|t=0}=u^0,
\ea\ee
    where $v$ is the solution of \eqref{eq:vs-y}. Once again, we derive regularity estimates on $u$.
    
    \item Eventually, we define
    \[
    \tilde z(t):=-\mu \int_0^t  \  \frac{ \p_x u(\tau, 0)}{u_- - \tw^0(\tilde y(\tau))}d\tau,
    \]
    and we consider the application $\cT:\tilde y \in H^2(0,T) \mapsto \tilde z \in H^2(0,T)$.
    
    We prove that  for $T>0$ small enough the application $\cT$ is a contraction, and therefore has a unique fixed point. 
    
\end{enumerate}

We then need to prove that the solution $(\tx, v_s, u_s)$ provided by the fixed point of $\cT$ is global when the initial energy is small (see Hypothesis~\eqref{hyp:small-glob}).
First, we will show that the passage to the integrated variables allows us to remove the exponential dependency with respect to time in the energy estimates.
Next, we prove in Section~\ref{sec:global}  that if the initial data is sufficiently small, then $\|\tx'-s\|_{H^1}$ remains bounded (and small) on the existence time of the solution.
The key ingredient to get this  property is the introduction of the new unknown 
$g_1 = -s (v-\bv) - \mu \partial_x \left(\frac{v-\bv}{\bv}\right) = \mathcal{A}(v-\bv)$
where $\partial_x \mathcal{A}$ is the linearized operator around $\bv$.
We exhibit coercivity properties for the linearized operator, and we prove that $g_1$ satisfies an equation of the type
\[
\p_t g_1 + \cA\p_x g_1=\text{quadratic terms}.
\]
Whence we deduce good estimates on both $g_1$ and $\tx'-s$.

We finally need to check that the solution $(\tx, v_s, u_s)$ of \eqref{EDO-tx}-\eqref{eq:vs}-\eqref{eq:NS-shifted-1-u} 
is indeed a solution of the original problem. 
Since system \eqref{eq:NS-shifted-1} has been modified, this is not completely obvious. In fact, we need to check that the function $w_s=u_s-\mu \ln v_s$ is indeed equal to $\tw^0 (x+\tx(t))$. To that end, let us compute the equation satisfied by $w_s$ if $v_s$ is the solution of \eqref{eq:vs} and if $u_s$ is the solution of \eqref{eq:NS-shifted-1-u}.
Combining \eqref{eq:vs} and \eqref{eq:NS-shifted-1-u}, we have
\be\label{eq:ws-para}
\p_t w_s - \tilde x'(t) \p_x w_s -\mu \p_x\left(\frac{1}{v_s}\p_x w_s\right)=-\mu \p_x\left(\frac{1}{v_s}\p_x \tw^0(x+\tx(t))\right).
\ee
Furthermore, the condition $u_{s|x=0^+}=u_-$ ensure that
\[
w_{s|x=0^+}= u_-- \mu \p_x v_{s |x=0^+},
\]
and using the equation \eqref{eq:vs} together with~\eqref{EDO-tx}
\begin{align*}
\p_x w_{s|x=0^+}
& = \p_x u_{s|x=0^+} - \mu \p_{xx} \ln v_{s|x=0^+} \\
& = \frac{\tw^0(\tx(t))-u_-}{\mu}\tx'(t) + \tx'(t) \p_x v_{s|x=0^+} + \p_x \tw^0(\tx(t)).
\end{align*}
Taking a linear combination of these two equations leads to
\be\label{BC-w}
\mu \p_x w_{s|x=0^+} + \tx'(t) w_{s|x=0^+}
= \tx'(t)\tw^0(\tx(t)) + \mu \p_x \tw^0(\tx(t)).
\ee
It can be easily proved that the unique solution of \eqref{eq:ws-para}-\eqref{BC-w} endowed with the initial data $\tw^0$ is the  function $\tw^0(x+\tx(t))$.
Thus the function $(v_s,u_s)$ constructed as the solution of \eqref{eq:vs}-\eqref{eq:NS-shifted-1-u}, where $\tx$ is the solution of \eqref{EDO-tx}, is in fact a solution of \eqref{eq:NS-shifted-1}. We extend this solution in $x<0$ by setting $v_s(t,x)=1$, $u_s(t,x)=u_-$, and we set
\[
p_s(t,x)=-\mu\p_x u_{s|x=0^+}= \tx'(t)(u_--\tw^0(\tx(t)))\quad \forall x<0.
\]
Eventually, we come back to the original variables and set 
$(v,u,p)(t,x)=(v_s,u_s,p_s)(t,x-\tx(t)).$
Then it is easily checked that $(v,u,p)$ is a solution of the original system \eqref{eq:NS-0}.

\medskip
\begin{rmk}[About the regularity of $\tx$]
In the above discussion, we have claimed that we will prove the existence of a fixed point $\tx$ in $H^2_{loc}(\R_+)$. 
Let us discuss why this regularity is required on $\tx$.
First, we need a control of $\tx'$ in $L^\infty(\R_+)$ in order to control the transport equation~\eqref{eq:ws} satisfied by $w_s$.
Next, we see in~\eqref{EDO-tx} that the control in $L^\infty$ of $\tx$ requires a bound on $\partial_x u_s$ in $L^\infty(\R_+ \times \R_+)$, while this latter bound would a priori rely on a control of $\tx''$ in $L^2_{loc}(\R_+)$ (see Proposition~\ref{prop:est-gal} in the Appendix).
Therefore the regularity $\tx \in H^2_{loc}(\R_+)$ is the minimal regularity which allows us to formally close the fixed point argument.
\end{rmk}

\medskip
\begin{rmk}[More general perturbations] 
This study is concerned with perturbations only affecting the ``free part'' of the travelling wave profile $(\bv,\bu,\bp)$.
This allows us to concentrate the analysis on the free domain $\{ x>\tx(t)\}$. 
Nevertheless, it would be interesting to study more general perturbations of $(\bv,\bu,\bp)$, including for instance a free zone $\{\tv^0 > 1\}$ in $\{x < 0\}$. 
It would be then necessary to analyze the coupled dynamics between the two disconnected free domains through the interior congested phase. 
This is actually out of the scope of the present paper and it is left for future work.
\end{rmk}

\bigskip

The organization of the paper is the following.
Given the discussion above, the goal is to construct a fixed point of the application $\mathcal T$.
Most of the proof is devoted to the derivation of suitable energy estimates on $v_s$. 
We first prove the well-posedness of equation \eqref{eq:vs-y} in Section \ref{sec:v-WP}.
Higher regularity estimates on $v_s$ and estimates for the solution $u_s$ of \eqref{eq:us-y} are then derived in Section~\ref{sec:estimates}. 
We prove in Section~\ref{sec:local} that $\mathcal T$ is a contraction and get therefore the local-wellposedness result stated in Theorem~\ref{thm:main-loc}.
Eventually, for small initial perturbations, we extend in Section~\ref{sec:global} the local solutions and show the asymptotic stability of the profiles $(\bv,\bu)$ which completes the proof of Theorem~\ref{thm:main-glob}.
We have postponed to the Appendix several technical results.

\subsubsection*{Notation.} Throughout the paper, we denote by $C$ a constant depending only on the parameters of the problem, i.e. $u_+$, $u_-$, $v_+$, $s$, and $\mu$. 

\section{Well-posedness of equation \texorpdfstring{\eqref{eq:vs-y}}{eqv}}
\label{sec:v-WP}
This section is devoted to the proof of well-posedness of equation \eqref{eq:vs-y}, and to the derivation of some preliminary bounds. In order to alleviate the notation, we drop the indices $s$. Throughout this section, we assume that $\ty$ is a given function in $H^2(0,T)$, such that $\ty(0)=0$, $\ty'(t)>0$ for all $t\in [0,T]$,  and satisfying the compatibility condition
\be\label{compatibility-ty}
\ty'(0)=-\frac{\p_x \tu^0_{|x=0^+}}{\p_x \tv^0_{|x=0^+}}.
\ee

We approximate equation \eqref{eq:vs-y} by considering an equation in a truncated, bounded domain. We derive $L^\infty$, $L^1$  and energy bounds on the sequence of approximated solutions, which are all uniform in $R$. Passing to the limit as the domain fills out the whole half-line, we recover the well-posedness of \eqref{eq:vs-y}.

Therefore, for $R \geq 2$, we consider the following equation in  $(0,T) \times (0,R)$
\begin{subnumcases}{\label{eq:v-trunc}}
\partial_t v^R(t,x) -\ty'(t) \partial_x v^R(t,x) - \mu \partial_{xx} a(v^R(t,x)) =  (\chi_R\partial_x \tw^{0})(x+\ty(t)), \\
v^R_{|t=0}=\tv^{0,R},\\
v^R_{|x=0}=1,\quad v^R_{|x=R}=\bv_{|x=R},
\end{subnumcases}
where the nonlinearity in the diffusion term has been changed in order to avoid degeneracy. The function $a\in \mathcal C^\infty(\R)$ is such that $a(x)=\ln (x)$ for $x \in [1/2,\bar C]$ where  $\bar C:= 2 \sup \tv^0$. 
Note that with this choice, 
$\bar C \geq v_+$.
We also assume  that there exists $\nu>0$ such that $\nu\leq a'(x)\leq \nu^{-1}$ for all $x\in \R$. The initial data $\tv^{0,R}$ is defined by $\tv^{0,R}=\tv^0 \chi_R + \bv_{|x=R} (1-\chi_R)$, where $\chi_R\in \mathcal C^\infty(\R)$ satisfies $\chi_R(x)=1$ if $x\leq R-2$ and $\chi_R(x)=0$ if $x\geq R-1$.  This change in the initial condition and in the source term ensures that the compatibility conditions at first and second order are satisfied at $t=0$, $x=R$, namely
\begin{equation}\label{eq:compatib_v}
v^R_{|t=0,x=R} = \bv_{|x=R}, \qquad
\ty'(0) \partial_x v^R_{|t=0,x=R} + \mu (\partial^2_x \ln v^R)_{|t=0,x=R} = 0. 
\end{equation}
The following result is a straightforward consequence of Theorem 6.2 in Chapter 5 of \cite{lady-para}:

\begin{lem}[Theorem 6.2 in \cite{lady-para}]{\label{thm:lady-1}}
There exists $\alpha>0$ such that system \eqref{eq:v-trunc} has  a unique solution $v^R\in \mathcal C([0,T]\times [0,R])$  such that $\p_t v^R \in \mathcal C^\alpha((0,T)\times (0,R))$ and $\p_{xx} v_R \in \mathcal C^\alpha((0,T)\times (0,R))$.
\end{lem}
\begin{proof}
We merely check the assumptions of \cite{lady-para}*{Theorem 6.2, Chapter 5}.
Due to the assumptions on the function $a$, we only need to verify that the initial data satisfies the compatibility conditions.
Note that the  function $(t,x)\in (0,T)\times (0,R)\mapsto \tv^{0,R}(x)$ coincides with $v^R$ on the sets $\{t=0\}$, $\{x=0\}$ and $\{x=R\}$. Furthermore, thanks to \eqref{compatibility-ty} and to \eqref{eq:compatib_v} we have the compatibility condition
	\begin{equation}\label{eq:compa}
	\left(\ty'(0)\p_x \tv^{0,R} + \mu \p_{xx} \ln \tv^{0,R}\right)_{|x=0}	=\partial_x \tw^0_{|x=0}, \qquad\left(\ty'(0)\p_x \tv^{0,R} + \mu \p_{xx} \ln \tv^{0,R}\right)_{|x=R}	=0.
	\end{equation}
	We infer that the assumptions of Theorem 6.1 in \cite{lady-para}*{Chapter 5} are satisfied, and $\p_t v^R, \p_{xx} v^R \in \mathcal C^\alpha ([0,T]\times [0,R])$ for some $\alpha>0$. 

\end{proof}

\subsection{Local-in-time \texorpdfstring{$L^\infty$}{Linfty} estimate}

\begin{lem}\label{lem:infty-est}
Let $T>0$ be arbitrary.
Assume  that there exists  $M\geq 1$ such that
\begin{equation} \label{hyp:M}
\frac{1}{M} \leq \ty'(t)\leq M \quad \forall \ t \in [0,T],\quad \inf_{x\in [0,1]}\p_x \tv^0(x) \geq \frac{1}{M}.
\end{equation}
Then there exists $T_0\in (0,T]$, depending only on $\|\p_x \tw^0\|_{L^\infty(\R_+)}$, $\sup \tv^0$, $M$ and on the parameters of the system, such that for all $t\in [0,T_0]$,
\begin{equation}
    1 < v^R(t,x) \leq \bar C\quad \forall x \in (0,R),
\end{equation}
where $\bar C = 2 \sup \tv^0 $ is the constant involved in the definition of $a$.
As a consequence, $a(v^R)= \ln (v^R)$ almost everywhere.

\end{lem}

\begin{proof}
The proof relies on the maximum principle. We construct by hand a super and a sub-solution, which give the desired estimates. 


	\begin{itemize}
	\item \textit{Sub-solution.}
	We set
		\[
		\uuv(t,x) = 1 + \ug(x) \exp(-A t),
		\]
	where $A>0$ is a parameter that will be chosen later on and where the function $\ug$ will be chosen so that $0\leq \ug \leq (v_+-1)/2$. 
	It follows that $a(\uuv)=\ln (\uuv)$.
	We will also take $T_0$ so that $\exp(-A T_0)=1/2$.
	With this choice, we have, on the one hand
\begin{eqnarray}
		&&\partial_t \uuv - \ty'(t)\partial_x \uuv - \mu \partial^2_x a (\uuv)\nonumber \\
		&=&- A e^{-At} \ug(x) - \tilde y'(t) \ug'(x) e^{-At} - \mu \p_x^2 \ln (1+ \ug(x) e^{-At})\nonumber\\
		&=&- A e^{-At} \ug(x) - \tilde y'(t) \ug'(x) e^{-At} -\mu\frac{\ug''(x) e^{-At}}{1+ \ug(x)e^{-At} } 
		+\mu \frac{(\ug'(x))^2 e^{-2At}}{(1+ \ug(x)e^{-At})^2}. \label{g-sous-sol}
		\end{eqnarray}
	On the other hand,
	\[
	(\chi_R \p_x \tw^0)(x+ \ty(t)) \geq -\|\p_x \tw^0\|_\infty.
	\]
	We choose the function $\ug$ as follows:
	\begin{itemize}
    \item For $x\leq x_0\leq 1$, we take $\ug(x)= \alpha_1 x + \alpha_2 x^2$, with $\alpha_1>0$ small enough and $\alpha_2>0$. Then, if $x_0$ is small enough, $\ug(x_0) = \alpha_1 x_0 + \alpha_2 x_0^2 \leq 1$, and therefore
    \[
    - A e^{-At} \ug(x) - \tilde y'(t) \ug'(x) e^{-At} - \mu\frac{\ug''(x) e^{-At}}{1 + \ug(x)e^{-At} } 
    \leq -{\mu \alpha_2}.
    \]
	Choose $\alpha_2$ so that $\mu \alpha_2\geq 2 \|\p_x \tw^0\|_\infty$, and choose $0<\alpha_1\leq  1/(2M)$. Then for $x_0$ small enough, $\uuv(0,x)\leq \tv^0(x)$ for all $x\in [0,x_0]$. Furthermore, the last term in \eqref{g-sous-sol} can be treated as a perturbation provided $\alpha_1$ and $x_0 $ are small enough.
	It can be checked that it is sufficient to take $\alpha_1^2\leq \alpha_2/8$ and $x_0\leq \alpha_1/(2\alpha_2)$.

	\item For $x\geq x_0$, we choose $\ug$ to be monotone increasing and such that $1 + \ug \leq \tv^0$ a.e.  and $\ug \leq (v_+-1)/2$. We then choose $A$ so that
	\[\ba
	\frac{A}{4} \ug(x_0) \geq \|\p_x \tw^0\|_\infty,\\
	A \ug(x_0)\geq 4\mu (\|\ug''\|_\infty + \|\ug'\|_\infty^2).
	\ea
	\]
	It follows that the right-hand side of \eqref{g-sous-sol} is lower that $-\|\p_x \tw^0\|_\infty$. 
	\end{itemize}
	
	Note that by construction, we have $\uuv\leq v^R$ on the sets $\{t=0\}\cup \{x=0\}\cup \{x=R\}$ provided $R$ is large enough.
	
	Thus, by the maximum principle,  $\uuv\leq v^R$. 
	
	\item {\it Super-solution:} We take
	\[
	\bbv(t,x):= C_1\left(2 - e^{-At}\right),
	\]
	where $C_1\geq  \sup \tv^0 $. 
	We keep choosing $T_0$ so that $\exp(-AT_0)=1/2$. Then
	\[
	\partial_t \bbv - \ty'(t)\partial_x \bbv - \mu \partial^2_x a (\bbv)= C_1 A e^{-At}.
	\]
	Take $A$ so that $C_1 A/2\geq \|\p_x \tw^0\|_\infty$. Then $\bbv$ is a super-solution of \eqref{eq:v-trunc}, and by the maximum principle, $v_R\leq \bbv$.
	
	\end{itemize}

\end{proof}

\bigskip

\begin{rmk}
The use of the maximum principle is actually not necessary for the local existence results that follow. 
Indeed, since we work in high regularity spaces, we could adapt our fixed point argument by linearizing our different systems around the profile $(\bv,\bu)$, treat the non-linear term perturbatively and use the inequality $\|v-\bv\|_{L^\infty(\R_+)} \leq \|v-\bv\|_{L^2(\R_+)}^{1/2} \|\partial_x(v -\bv)\|_{L^2(\R_+)}^{1/2}$. However, since this would lead to unnecessary technicalities, we have decided to use the maximum principle in this paper.
\end{rmk}

\subsection{\texorpdfstring{$L^1$}{L1} estimate}
Anticipating Section~\ref{sec:unif-est-g} and the derivation of uniform-in-time estimates on $g = v-\bv$, we aim at proving that $g(t,\cdot) \in L^1(\R_+)$ for all times $t \geq 0$.
This property will be used in Section~\ref{sec:unif-est-g} to justify the passage to the integrated quantity $V$.

The goal of this Subsection is to prove the next Lemma
\begin{lem}\label{lem:L1-vR}
	Assume that $\tv^0 -\bv \in L^1(\R_+)$, and that the assumptions of Lemma \ref{lem:infty-est} are satisfied. 
	Then there exists a constant $C$, depending only on $T_0$,  $\sup \tv^0$, $M$ (see~\eqref{hyp:M}) and the parameters of the system, such that the solution $v^R$ of \eqref{eq:v-trunc} provided by Lemma~\ref{thm:lady-1} satisfies the following estimate:
	\begin{equation}\label{eq:L1-vR}
	\sup_{t \in [0,T_0]} \|v^R - \bv\|_{L^1(0,R)} 
	\leq  C  \Big[\|\tv^{0,R} - \bv\|_{L^1(0,R)} + \|\partial_x \tw^0\|_{L^1(\R_+)} +  \|\partial_x \bv\|_{L^1(\R_+)}\Big].
	\end{equation}
\end{lem}

\medskip
\begin{proof}
Using the equation satisfied by $\bv$, we have
\begin{equation}\label{eq:vR-bv}
\partial_t(v^R - \bv) - \ty'(t) \partial_x(v^R - \bv) - \mu \partial^2_x \big(\ln v^R - \ln \bv) = (\chi_R\partial_x \tw^0)(x+\ty(t)) + (\ty'(t) -s) \partial_x \bv.
\end{equation}
For $n > 0$, let us introduce $j_n \in \mathcal{C}^2(\R)$ defined by
\[
j_n(r) = \sqrt{r^2 + \frac{1}{n}} - \sqrt{\frac{1}{n}} \quad \forall r \in \R
\]
which is a smooth, convex, approximation of the function $r \mapsto |r|$ as $n \to + \infty$.
Hence, $j'_n(r) = r/\sqrt{r^2+ 1/n}$ is an approximation of the sign function, and $j_n(0)=j_n'(0)=0$.
Multiplying~\eqref{eq:vR-bv} by $j'_n\left(\frac{v^R -\bv}{\bv}\right)$, integrating the result between $0$ and $R$ and noting that $(v^R-\bv)_{|x=0,R}=0$, we get
\begin{align*}
&\int_0^R \bv \ \partial_t j_n\left(\frac{v^R -\bv}{\bv}\right) \ dx 
- \ty'(t) \int_0^R \partial_x(v^R -\bv)\ j_n'\left(\frac{v^R -\bv}{\bv}\right) \ dx \\
& \quad + \mu \int_0^R j_n''\left(\frac{v^R-\bv}{\bv}\right)\partial_x \ln\left(\frac{v^R}{\bv}\right) \partial_x\left(\frac{v^R -\bv}{\bv}\right)dx \\
& = \int_0^R (\chi_R\partial_x \tw^0)(x+\ty(t)) \ j'_n\left(\frac{v^R -\bv}{\bv}\right) dx
+ (\ty'(t) -s) \int_0^R \partial_x \bv \ j'_n\left(\frac{v^R -\bv}{\bv}\right) dx. 
\end{align*}
Since $\frac{v^R-\bv}{\bv} \geq \frac{1}{v_+} - 1 > -1$, we have
\begin{align*}
& \mu \int_0^R j_n''\left(\frac{v^R-\bv}{\bv}\right)\partial_x \ln\left(\frac{v^R}{\bv}\right) \partial_x\left(\frac{v^R -\bv}{\bv}\right)dx \\
& = \mu \int_0^R \dfrac{1}{1+\frac{v^R-\bv}{\bv}} j_n''\left(\frac{v^R-\bv}{\bv}\right)\left(\partial_x \left(\frac{v^R-\bv}{\bv}\right)\right)^2 dx
\geq 0,
\end{align*}
and, thanks to the boundary conditions $(v^R-\bv)_{|x=0,R}=0$,
\begin{align*}
& \ty'(t) \int_0^R \partial_x(v^R -\bv)\ j_n'\left(\frac{v^R -\bv}{\bv}\right) \ dx \\
& =  \ty'(t) \int_0^R \bv \ \partial_x\left(\frac{v^R -\bv}{\bv}\right)\ j_n'\left(\frac{v^R -\bv}{\bv}\right) \ dx 
 + \ty'(t) \int_0^R \bv \frac{\partial_x \bv}{\bv^2} (v^R -\bv)\ j_n'\left(\frac{v^R -\bv}{\bv}\right) \ dx \\
& =  \ty'(t) \int_0^R \bv \ \partial_x j_n\left(\frac{v^R -\bv}{\bv}\right) \ dx 
+ \ty'(t) \int_0^R \dfrac{v^R -\bv}{\bv} \partial_x \bv \ j'_n\left(\frac{v^R -\bv}{\bv}\right) \ dx \\
& = - \ty'(t) \int_0^R \partial_x \bv \ j_n\left(\frac{v^R-\bv}{\bv}\right) \ dx
+ \ty'(t) \int_0^R \dfrac{v^R -\bv}{\bv} \partial_x \bv \ j'_n\left(\frac{v^R -\bv}{\bv}\right) \ dx.
\end{align*}
Hence, using once again the $L^\infty$ control of $v^R$ from Lemma \ref{lem:infty-est}, we get
\begin{align*}
& \int_0^R  \bv \ j_n\left(\frac{v^R -\bv}{\bv}\right)(t) \ dx -  \int_0^R  |v^{0,R} -\bv| \ dx \\
& \leq C \|\ty'\|_{L^1(0,t)} \|\partial_x \bv\|_{L^1(\R_+)}
+ t \|\partial_x \tw^0\|_{L^1(\R_+)}
+ \|\ty'(\cdot) -s\|_{L^1(0,t)} \|\partial_x \bv\|_{L^1(\R_+)}. 
\end{align*}
where we have used the fact that $j_n(r) \leq |r|$. 
Passing to the limit $n\to +\infty$ and using Fatou’s lemma, we finally obtain~\eqref{eq:L1-vR}.
\end{proof}

\subsection{Energy estimates and existence of \texorpdfstring{$v$}{v}}{\label{sec:existence-v}}

The goal of this subsection is to prove the existence and uniqueness of $v \in 	L^\infty([0,T], H^1_{loc}(\R_+))\cap L^2([0,T], H^2(\R_+))$ solution to \eqref{eq:vs-y}.
We proceed in two steps. 
First we derive classical a priori estimates on $v^R$, exponentially growing with time but independent of $R$.
These estimates ensure the existence of $v$ by passing to the limit in \eqref{eq:vR-bv} as $R\to +\infty$.

\begin{lem}\label{lem:vR_L2}
Assume that $1\leq v^R(t,x)\leq \bar C$ for all $t \leq T_0, \ x \in (0,R)$, where we recall that $\bar C=2 \sup \tv^0$, and that there exists $M\geq 1$ such that $M^{-1}\leq \ty'(t)\leq M$ for all $t \in [0,T_0]$.

There exists a constant $C$, depending only on $s,\mu, v_+, M$, and $\bar C$,
such that the solution $v^R$ of~\eqref{eq:v-trunc} satisfies
\begin{align}\label{eq:E1-v0}
&  \dfrac{d}{dt} \|(v^R-\bv)(t)\|_{H^1(0,R)}^2
+ \|\partial_x (v^R-\bv)(t) \|_{L^2(0,R)}^2 + \|\partial_t v^R(t) \|_{L^2(0,R)}^2 \nonumber \\
& \leq 
C \left(\|(v^R-\bv)(t) \|_{L^2(0,R)}^2 + \|\partial_x\tw^0(\cdot +\ty(t))\|_{L^2(\R_+)}^2 + |\ty'(t) -s|^2\right)  \qquad \forall \ t \in (0,T_0),
\end{align}	
so that, by a Gronwall inequality,
\begin{align}\label{eq:est-H1-vR}
&  \sup_{t \in [0,T]} \|(v^R-\bv)(t)\|_{H^1(0,R)}^2
+ \|\partial_x (v^R-\bv) \|_{L^2((0,T)\times (0,R))}^2 + \|\partial_t v^R \|_{L^2((0,T)\times (0,R))}^2 \nonumber \\
& \leq 
 C \left( \|v^{R,0} -\bv\|_{H^1(0,R)}^2 +  \|\sqrt{x} \partial_x\tw^0\|_{L^2(\R_+)}^2 +  \|\ty'-s\|_{L^2(0,T)}^2 \right)\exp(C T)
 \nonumber 
\end{align}	
for all $T \leq T_0$.
\end{lem}

\begin{rmk}
Note that throughout this section and the next ones (i.e. in all energy estimates), the constant $C$ might depend on $\bar C$, hence on $\sup \tv^0$. 
Hence, all constants will depend on the initial size of the perturbation.

\end{rmk}

\medskip
\begin{proof}
This is an immediate consequence of Lemma \ref{lem:generique-g} in the Appendix. We set $\bar g =\bar \tv$, $g=v^R-\bar \tv$, $G=(\chi_R \p_x\tw^0)(x+\ty(t)) + (\ty'(t)-s)\p_x\bar \tv$. It can be easily checked that the assumptions of Lemma \ref{lem:generique-g} are satisfied. Furthermore, $\bar \tv$ does not depend on time, and its $W^{2,\infty}$ norm only depends on $\mu, s$ and $v_+$. 

There only remains to evaluate $\|G\|_{L^2}$. We use Lemma \ref{lem:x+y} in the Appendix and we find
\[
\|G\|_{L^2((0,T)\times (0,R))}\leq C\left(\|\partial_x \bv\|_{L^2(\R_+)}\|\ty' -s\|_{L^2(0,T)} + \|\sqrt{x} \partial_x\tw^0\|_{L^2(\R_+)}\right).
\]
This concludes the proof of the Lemma.
\end{proof}

\bigskip

\bigskip

We are now ready to prove the existence and uniqueness of $v$.

\bigskip

\begin{lem}
	Let $\ty\in H^2(0,T)$ such that $\ty(0)=0$ and  $\ty$ satisfies \eqref{compatibility-ty}. Assume that that there exists a constant $M>1$ such that $\ty, \tv^0$ satisfy \eqref{hyp:M}.
	There exists $T_0\leq T$, depending on $M$ and on the parameters of the problem, such that there exists a unique solution $v \in 
	L^\infty([0,T_0], H^1_{loc}(\R_+))\cap L^2([0,T_0], H^2(\R_+))$
	to
	\begin{equation}{\label{eq:v}}
	\begin{cases}
	\partial_t v(t,x) -\ty'(t) \partial_x v(t,x) - \mu \partial_{xx} \ln v(t,x) =  \partial_x \tw^0(x+\ty(t)), \\
	v_{|t=0}=\tv^0,\\
	v_{|x=0}=1,\quad \lim_{x \to +\infty} v(t,x) = v_+ \qquad \forall \ t \geq 0,
	\end{cases}
	\end{equation}	
	and such that $\p_t v \in L^2((0,T_0)\times \R_+)$.
	Furthermore, there exists a constant $C>0$ depending $M, v_+, u_\pm, s,\mu$ such that $v$ satisfies the estimate
	\begin{align}
	 E_1(T_0)
	 &:= \|v-\bv\|_{L^\infty((0,T_0), H^1}^2 + \|\p_x (v-\bv)\|_{L^2((0,T_0)\times \R_+)}^2 + \|\p_t v\|_{L^2 ((0,T_0)\times \R_+)}^2\nonumber\\
	 &\leq C\left( \|\tv^0 -\bv\|_{H^1(\R_+)}^2 +  \|\sqrt{x} \partial_x\tw^0\|_{L^2(\R_+)}^2 +  \|\ty'-s\|_{L^2(0,T_0)}^2\right) e^{CT_0},  \label{est:H1-v}\\
	 \|\p_x^2 (v-\bv)\|_{L^2((0,T_0)\times \R_+)}&\leq C (E_1(T_0) + \inf(1, T_0)E_1(T_0)^2).\nonumber
	    \end{align}

	\label{lem:WP-vn}
\end{lem}

\bigskip
\begin{proof}
We consider the function $v^R$, which satisfies the estimates of Lemmas \ref{lem:infty-est} and \ref{lem:vR_L2}.  

Extending $v^R$ to $[0,T_0] \times \R_+$ by taking $v^R(t,x)=\bv(x=R)$ for $x\geq R$, the following limits hold: 
	\begin{itemize}
		\item Up to a subsequence, $v^R\to v$ in $L^2([0,T_0]\times K)$ for any compact set $K\subset \R_+$;
		\item $v^R\stackrel{*}{\rightharpoonup} v\quad w^*-L^\infty([0,T_0]\times \R_+)$ 
		\item $	v^R-\bv \rightharpoonup v-\bv \quad w-H^1([0,T_0]\times \R_+)$;
	\end{itemize}
	Thanks to the strong compactness property, we can pass to the limit in \eqref{eq:v-trunc} and we infer that $v$ satisfies 
	\begin{equation*}
	\partial_t v(t,x) -\ty'(t) \partial_x v(t,x) - \mu \partial_{xx} \ln v(t,x) =  \partial_x \tw^0(x+\ty(t)).
	\end{equation*}
	 Passing to the limit in the estimates on $v^R$, we also know that $v$ satisfies the estimates of Lemmas \ref{lem:infty-est} and \ref{lem:vR_L2}. This completes the proof of existence.
	 Additionally, using the last item of Lemma \ref{lem:generique-g}, we obtain the estimate on $\p_x^2(v-\bv)$.
	
	Uniqueness is quite classical. For instance, if $v_1,v_2$ are two solutions, then setting $\hat{v} =v_1-v_2$, we find that $\hat{v}$ satisfies
	\[
	\p_t \hat{v} - \ty'(t)\p_x \hat{v} - \mu \p_{xx} \ln \left(1+ \frac{\hat{v}}{v_2}\right)=0.
	\]
	Multiplying the above equation by $\hat{v}$ and integrating by parts, we find that
	\[
	\frac{d}{dt}\int_0^\infty |\hat{v} |^2 \ dx + \mu \int_0^\infty \frac{(\p_x \hat{v})^2}{v_1} \ dx
	\leq C \|\p_x v_2 \|_{L^\infty(\R_+)}^2 \int_0^\infty |\hat{v}|^2 dx .
	\]
	Writing 
	\begin{eqnarray*}
		\|\p_x v_2(t)\|_{L^\infty(\R_+)}&\leq & \|\p_x (v_2(t)-\bv)\|_{L^\infty(\R_+)} + \|\p_x \bv\|_{L^\infty(\R_+)}\\
		&\leq & \|\p_x (v_2-\bv)\|^{1/2}_{L^\infty([0,T], L^2(\R_+))}\|\p_{xx} (v_2(t)-\bv)\|^{1/2}_{ L^2(\R_+)} + \|\p_x \bv\|_{L^\infty(\R_+)}
	\end{eqnarray*}
	and using a Gronwall type argument, we infer that $\hat{v} \equiv 0$.
\end{proof}

\section{Energy estimates on \texorpdfstring{$v$}{v} and \texorpdfstring{$u$}{u}}
\label{sec:estimates}

This section is devoted to the derivation of high regularity estimates on $v -\bv$ and $u - \bu$, where $v, u$ are respectively solutions of \eqref{eq:vs-y} and \eqref{eq:us-y}.
The existence and uniqueness of $v$ follows from Lemma \ref{lem:WP-vn}.
Preliminary energy estimates on $v$ were derived in the previous section (see \eqref{eq:est-H1-v}). 
However, these estimates are not completely satisfactory, because of the exponential growth in the right-hand side. 
This exponential growth is not really a problem for the local well-posedness result we will prove in the next section (see Theorem \ref{thm:main-loc}), but it could prevent us from proving the global stability in section \ref{sec:global}.
Fortunately, we are able to prove that energy bounds hold {\it without} the exponential loss.

The outline of this section is the following:
\begin{enumerate}
    \item We first revisit the energy estimates from Lemma \ref{lem:vR_L2} in order to prove that there is no exponential growth of the energy.
    To that end, we consider the equation satisfied by the integrated variable $V:=-\int_x^{+\infty} (v - \bv)$ and derive energy estimates for $V$.
    
    \item We then prove higher order regularity estimates for $v$. More precisely, we prove that $v-\bv\in L^\infty_t(H^3)\cap W^{1,\infty}_t(H^1)\cap H^1_t(H^2)\cap H^2_t(L^2)$.
    
    \item The existence and uniqueness of $u$ then follows from classical results on linear parabolic equations with smooth coefficients. 
    We also derive energy estimates for $u-\bu$ in $W^{1,\infty}_t(H^1)\cap H^1_t(H^2)\cap H^2_t(L^2)$.

\end{enumerate}

As in the previous section, we consider a given function $\ty\in H^2(0,T)$ such that $\ty(0)=0$ and satisfying \eqref{compatibility-ty}. 
We assume throughout this section that the existence time $T$ of the solution $v$ of \eqref{eq:vs-y} is such that $1< v \leq \bar C=2\sup \tv^0$ a.e. (see Lemma \ref{lem:infty-est}), and that there exists a constant $M\geq 1$ such that
\be\label{hyp:ty-M}
\ba
M^{-1}\leq \ty'(t)\leq M \quad \forall t\in [0,T],\\
M^{-1}\leq \p_x \tv^0_{|x=0}\leq M.\ea
\ee

We shall see that the total initial energy is
\begin{align}\label{def:E0}
\mathcal E_0&:=\|\tv^0 - \bv\|_{H^3(\R_+)}^2 + \|\tu^0-\bu\|_{H^3(\R_+)}^2  +\|V^0\|_{L^2(\R_+)}^2 \\
&\quad + \|(1+\sqrt{x}) W^0\|_{L^2(\R_+)}^2 + \|(1+\sqrt{x}) \p_x\tw^0\|_{L^2(\R_+)}^2 + \|(1+\sqrt{x})\partial_x^2\tw^0\|_{L^2(\R_+)}^2,\nonumber
\end{align}
where $W^0(x):=-\int_x^\infty (\tw^0-u_+)$.

The energy at time $T$ is defined by
\be\label{def:ET}
\cE_T:=\cE_0 + \|\ty'-s\|_{H^1(0,T)}^2.
\ee

Let us now state the results we shall prove for $u$ and $v$:

\begin{prop}[Energy estimates for $v$]
Assume that $\ty'(0)=-\frac{\p_x \tu^0_{|x=0}}{\p_x \tv^0_{|x=0}}$ and that $\ty$ and $\tv^0$ satisfy \eqref{hyp:ty-M}.
There exists a universal constant $p>1$, a constant $C$ depending on $\mu, v_+, s, \bar C$ and $M$,
and a time  $T^0>0$ depending only on $\cE_0$ and on $\|\ty'-s\|_{L^2([0,T])}$, such that if $T\leq T^0$, then
    \begin{align*}
    &\|v-\bv\|_{L^\infty([0,T], H^2(\R_+))}^2 + \|\p_t v\|_{L^\infty([0,T], L^2(\R_+))}^2 + \|\p_t v\|_{L^2([0,T], H^1(\R_+))}^2 
     \leq C \cE_T,
   \end{align*}
     and
     \begin{align*}
     &\|\p_t \p_x v\|_{L^\infty((0,T), L^2(\R_+))}^2 + \| \p_x^3(v-\bv)\|_{L^\infty([0,T], L^2(\R_+))}^2 + \| \p_t^2 v\|_{L^2([0,T]\times \R_+)}^2 + \|\p_x^2 \p_t v\|_{L^2([0,T]\times \R_+)}^2\\
     &\leq  C \cE_T (1+ \cE_T)^p. 
     \end{align*}

\label{cor:recap-est-v}
\end{prop}

A similar result holds for $u$:
\begin{prop}[Energy estimates for $u$]
Assume that $\ty'(0)=-\frac{\p_x \tu^0_{|x=0}}{\p_x \tv^0_{|x=0}}$ and that $\ty$ satisfies \eqref{hyp:ty-M}.
Assume furthermore that
\begin{equation}\label{hyp:compa-2-u}
\partial_x \tu^0_{|x=0}\Big(\ty'(0) -\mu (\partial_x \tv^0)_{|x=0} \Big)
+ \mu (\partial^2_x \tu^0)_{|x=0} = 0,
\end{equation}
Then there exists a time $T^0>0$ depending only on $\cE_0$ and on $\|\ty'-s\|_{L^2([0,T])}$, such that for all $0<T<T^0$, with the same notation as in Proposition \ref{cor:recap-est-v},
\[
 \|u-\bu\|_{H^1([0,T], H^2(\R_+))}^2 + \|u-\bu\|_{W^{1,\infty}([0,T], H^1(\R_+))}^2 + \|u-\bu\|_{H^2([0,T], L^2(\R_+))}^2\leq C \cE_T (1+ \cE_T)^p.
\]

\label{prop:est-u}
\end{prop}

\medskip
\begin{rmk}
Recalling the compatibility condition $\ty'(0)=-(\p_x \tu^0/\p_x \tv^0)_{|x=0}$, the compatibility condition \eqref{hyp:compa-2-u} actually amounts to
\[\left[-\frac{(\p_x \tu^0)^2}{\p_x \tv^0} - \mu \p_x \tv^0 \p_x \tu^0 + \mu \p_x^2 \tu^0\right]_{|x=0}=0,\]
which is exactly (H3).
Notice that this second compatibility condition only involves the initial data $\tu^0, \tv^0$. In other words, no further condition on $\ty$ is required.

\end{rmk}

We now turn towards the proof of these two propositions, following the outline above.
In order to alleviate the notation, we introduce the following quantities:
\[
g:=v-\bv,\quad h:=u-\bu.
\]

\subsection{Uniform in time estimates for \texorpdfstring{$g=v-\bv$}{g}}{\label{sec:unif-est-g}}

As announced previously, the first step consists in removing the exponential dependency with respect to time in~\eqref{eq:est-H1-v}. 
For that purpose, we use the integrated variable 
\begin{equation}\label{df:V}
V(t,x) := -\int_x^{+\infty} (v - \bv)(t,z) \ dz.
\end{equation}
Using Lemma \ref{lem:L1-vR} and passing to the limit as $R\to \infty$, we recall that $V$ is well-defined and uniformly bounded in $L^\infty([0,T]\times \R_+)$. 
Furthermore, using the identity $s\p_x \bv + \mu \p_{xx} \ln \bv=0 $, we find that $g$ is a solution of
\be\label{eq:g}
\p_t g - \ty'\p_x g -\mu \p_{xx}\ln \left(1 + \frac{g}{\bv}\right)=\p_x \tw^0(x+\ty(t)) - \frac{\mu}{s}(\ty'-s)\p_{xx}\ln \bv.
\ee
Integrating the above equation with respect to $x$, we find that $V$ is solution to
\begin{align}\label{eq:V}
& \partial_t V - \ty'(t) \partial_x V - \mu \partial_x \ln \left(1 + \dfrac{\partial_x V}{\bv} \right) \nonumber \\
& = \big(\tw^0(x + \ty(t)) - u_+ \big)
- \dfrac{\mu}{s}(\ty'-s)\partial_x \ln \dfrac{\bv}{v_+}\nonumber\\
&=\p_x W^0(x+\ty(t))- \dfrac{\mu}{s} (\ty'-s)\partial_x \ln \dfrac{\bv}{v_+},
\end{align}
where $W^0(x)=-\int_x^\infty (\tw^0 - u_+).$
As detailed below, we expect to control $\|\partial_x V\|_{L^2([0,T]\times\R_+)}=\|v-\bv\|_{L^2([0,T]\times\R_+)}$ uniformly with respect to the time $T$. 
Replacing this estimate in~\eqref{eq:E1-v0}, we get the following result.

\medskip
\begin{lem}\label{lem:L2-V}
The solution $v$ constructed in Lemma \ref{lem:WP-vn} satisfies the following estimate
	\begin{align}\label{eq:est-H1-v}
E_1(T)	
	& \leq 
	C \Bigg[ \|V^{0} \|_{L^2(\R_+)}^2 +\|(1+\sqrt{x}) W^0\|_{L^2(\R_+)}^2+ \|v^0-\bv\|_{H^1(\R_+)}^2
	\\
	& \qquad + \| \sqrt{x} \p_x \tw^0\|_{L^2(\R_+)}^2  + \|\ty'-s\|_{L^2(0,T)}^2 \Bigg]
	\nonumber \\
	& \leq C \cE_T , \nonumber 
	\end{align}
	where the energy $E_1(T)$ was defined in Lemma \ref{lem:WP-vn}.

\end{lem}
Note that the upper bound on $ E_1(T)$ only depends on $\cE_0$ and on $\|\ty'-s\|_{L^2(0,T)}$. Hence, assuming that $T$ is such that $\inf(1, T) \cE_T \leq 1$, we obtain the first set of inequalities announced in Proposition \ref{cor:recap-est-v}.

\medskip
\begin{proof}
	We multiply Eq.~\eqref{eq:V} by $V$ and integrate over $\R_+$:
	\begin{align*}
	& \dfrac{d}{dt} \int_0^{+\infty} \dfrac{|V|^2}{2} dx
	+ \dfrac{\ty'(t)}{2} (V_{|x=0})^2  + \mu \int_{\R_+}  \ln \left(1 + \dfrac{\partial_x V}{\bv} \right) \ \partial_x V \ dx \nonumber \\
	& = \int_{\R_+} \p_x W^0(x+\ty(t)) \ V dx
	+ \dfrac{\mu}{s} \big(\ty'(t)-s\big) \int_{\R_+} \partial_x \ln \dfrac{\bv}{v_+}\ V \ dx \\
	& = - \int_{\R_+}  W^{0}(x + \ty(t)) \ \partial_x V dx
	- \dfrac{\mu}{s} \big(\ty'(t)-s\big) \int_{\R_+} \ln \dfrac{\bv}{v_+}\ \partial_x V \ dx \\
	& \quad - \left[W^{0}_{|x=\ty(t)} +  \dfrac{\mu}{s} \big(\ty'(t)-s\big) \ln \dfrac{1}{v_+}\right] \ V_{|x=0}.
	\end{align*}
	Using Lemma \ref{lem:infty-est}, we recall that
	\[
	\frac{\p_x V}{\bv}=\frac{v-\bv}{\bv}\in \left[\frac{1}{v_+}-1, \frac{\bar C}{v_+}-1\right].
	\]
Observing that 
	\[
	\mu \int_{\R_+}  \ln \left(1 + \dfrac{\partial_x V}{\bv} \right) \ \partial_x V
	\geq c_0 \int_{\R_+}  |\partial_x V|^2
	\]
	for some $c_0 = c_0(\mu,\bar C, v_+) > 0$, we have by the Cauchy-Schwarz inequality
	\begin{align*}
	& \dfrac{d}{dt} \int_{\R_+} \dfrac{|V|^2}{2} dx
	+ \dfrac{\ty'(t)}{2} (V_{|x=0})^2 + c_0 \int_{\R_+}  |\partial_x V|^2 \\
	& \leq \dfrac{c_0}{2} \int_{\R_+}  |\partial_x V|^2 + \dfrac{\ty'(t)}{4} (V_{|x=0})^2 \\
	& \quad + C \left[\| W^{0}(\cdot + \ty(t))\|_{L^2(\R_+)}^2 + (W^{0}_{|x=\ty(t)})^2
	+ |\ty'(t)-s|^2 \big(\|\bv- v_+\|_{L^2(\R_+)}^2 + |v_+-1| \big)
	\right]
	\end{align*}
	for some positive constant $C = C(c_0, \mu,s, M)$.
	Using Lemma \ref{lem:x+y} in the Appendix, we have
	\begin{align}\label{eq:W0-L2}
	&\| W^{0}(\cdot + \ty)\|_{L^2((0,T)\times\R_+)} \leq C\|\sqrt{x} W^0\|_{L^2(\R_+)},\nonumber\\
	&\| W^{0}(\ty(\cdot))\|_{L^2(0,T)}\leq  C\| W^0\|_{L^2(\R_+)}.
	\end{align}
	Hence
	\begin{align}\label{eq:est_E0}
E_0(T)	& := \sup_{t \in [0,T]} \Big[ \|V(t)\|_{L^2(\R_+)}^2 + \ty'(t) (V_{|x=0})^2 \Big] 
	+ \|v - \bv \|_{L^2([0,T] \times \R_+)}^2 \nonumber \\
	& \leq C \Bigg[ \|V^{0} \|_{L^2(\R_+)}^2 + 
	 \|(1+\sqrt{x}) W^0\|_{L^2(\R_+)}^2  \nonumber \\
	& \qquad 
	+ \|\ty'(t)-s\|_{L^2(0,T)}^2 \big(\|\bv- v_+\|_{L^2(\R_+)}^2 + |v_+-1| \big)
	\Bigg] \nonumber\\
	& \leq C \Bigg[ \|V^{0} \|_{L^2(\R_+)}^2 + 
	 \|(1+\sqrt{x}) W^0\|_{L^2(\R_+)}^2 + \|\ty'(t)-s\|_{L^2(0,T)}^2 \Bigg]
\\
&\leq C \cE_T.
	\end{align}
	By Lemma~\ref{lem:WP-vn}, we have
	\begin{align*}
	&  \dfrac{d}{dt} \|(v-\bv)(t)\|_{H^1(\R_+)}^2
	+ \|\partial_x (v-\bv) \|_{L^2(\R_+)}^2 + \|\partial_t v \|_{L^2(\R_+)}^2 \nonumber \\
	& \leq 
	C \left(\|(v-\bv)(t) \|_{L^2(\R_+)}^2 + \|\partial_x \bv\|_{L^2(\R_+)}^2 |\ty' -s|^2 + \| \partial_x\tw^0(\cdot +\ty(t))\|_{L^2(\R_+)}^2\right)
	\end{align*}	
	so that
	\begin{align*}
E_1(T)	&=  \sup_{t \in [0,T]} \|(v-\bv)(t)\|_{H^1(\R_+)}^2
	+ \|\partial_x (v-\bv) \|_{L^2([0,T] \times \R_+)}^2 + \|\partial_t v \|_{L^2([0,T] \times \R_+)}^2 \nonumber \\
	& \leq 
	\|\tv^0-\bv\|_{H^1}^2 + C\left(E_0(T) + \|\ty'-s\|_{L^2(0,T)}^2+ \| \sqrt{x}\p_x \tw^0\|_{L^2(\R_+)}^2\right)\\
	&\leq C \cE_T.
	\end{align*}		
	This concludes the proof of the lemma.
\end{proof}


\subsection{Higher order regularity estimates for \texorpdfstring{$g=v-\bv$}{g}} 
\label{sec:regularity-v}

This section is devoted to the derivation of energy estimates in $L^\infty([0,T], H^3(\R_+))$ for $v-\bv$, in $L^\infty([0,T], H^1(\R_+))\cap L^2([0,T], H^2(\R_+))$ for $\p_t v$, and in $L^2([0,T]\times \R_+)$ for $\p_t^2 v$.

We split the rest of the proof of  Proposition \ref{cor:recap-est-v} into two steps corresponding to Lemmas~\ref{lem:Hreg-v-1} and~\ref{lem:Hreg-v-2}.

\begin{lem}[Higher regularity - I]{\label{lem:Hreg-v-1}}
	
	Assume that $\ty'(0)= - \frac{\p_x \tu^0_{|x=0}}{\partial_x \tv^0_{|x=0}}$ and that $\inf (1,T^{1/2}) E_1(T)\leq 1$.

	The solution $v$ of~\eqref{eq:v} satisfies
	\[
	\ba
	v -\bv \in L^\infty([0,T], H^2(\R_+)),\\
	\partial_t(v-\bv) \in L^\infty([0,T];L^2(\R_+)) \cap L^2([0,T]; H^1(\R_+)).
	\ea
	\]
	Furthermore, the following estimate holds:

	\begin{align}\label{est:high-reg-v-I-bis}
	E_2(T):=\sup_{t \in[0,T]} \Big[\left\|\partial_t  (v-\bv)(t)\right\|_{L^2(\R_+)}^2 + \|\partial^{2}_x(v-\bv)(t)\|_{L^2(\R_+)}^2  \Big] \nonumber\\
	+ \| \partial_t \partial_x (v-\bv)\|_{L^2([0,T]\times\R_+)}^2   \leq \cE_T. 
	\end{align}

\end{lem}

\bigskip
\begin{proof}
	Let us first derive formally these regularity estimates, and then explain how they can be proved rigorously.
	We recall that we have set $g:=v-\bv$.

	Differentiating \eqref{eq:v} with respect to time leads to
	\begin{equation}\label{eq:dt2-v}
	\partial_t \partial_t g - \mu \partial_{x} \left(\dfrac{1}{v} \partial_x \partial_t  g \right) \\
	= \partial_t F 
	- \mu \partial_x \left( \dfrac{\partial_t v}{v^2} \partial_xg \right)
	\end{equation}
	where we have denoted 
	\begin{equation}\label{eq:F}
	F 
	= \partial_x \tw^0(x+\ty(t))
	+ \ty'(t) \partial_x g  + (\ty'-s)\partial_x \bv 
	+ \mu \partial_{x} \left( \dfrac{\partial_x \bv }{v \bv} g\right). 
	\end{equation}
	Multiplying by $\partial_t g(=\p_t v)$ and integrating we get
	\begin{align*}
	& \dfrac{d}{dt} \int_{\R_+} \dfrac{|\partial_tg|^2}{2} 
	+ \mu \int_{\R_+} \dfrac{|\partial_x \partial_t g|^2}{v} 
	+ \mu \underset{=0}{\underbrace{\left[\dfrac{1}{v}\partial_x \partial_t g \ \partial_t g \right]_{|x=0}}} \\
	& = \int_{\R_+} \partial_t F \partial_tg
	+ \mu \int_{\R_+}\dfrac{\partial_t v}{v^2}\partial_xg \partial_x \partial_tg
	+ \mu \underset{=0}{\underbrace{\left[\dfrac{\partial_t v}{v^2}\partial_x g \ \partial_t g \right]_{|x=0}}} 
	\end{align*}
	We estimate the two terms of the RHS as follows
	\begin{align*}
	\left|\int_{\R_+} \partial_t F \partial_tg\right|
	& \leq \|\partial_t F\|_{L^2(\R_+)} \| \partial_tg\|_{L^2(\R_+)} \\
	& \leq \dfrac{\mu}{4} \int_{\R_+} \dfrac{|\partial_x \partial_t g|^2}{v}
		 + |\ty'(t)|^2 \|\partial^2_x \tw^0(\cdot + \ty(t))\|_{L^2(\R_+)}^2 \\
	& \quad + C|\ty''(t)|(\|\p_x g\|_{L^2(\R_+)} + \|\p_x \bv\|_{L^2(\R_+)})\|\p_t g\|_{L^2(\R_+)}\\
	& \quad + C \|\partial_t g\|_{L^2(\R_+)}^2 \big(\|\partial_x g\|_{L^\infty(\R_+)} +|\ty'(t)|^2 + 1\big),
	\end{align*}
	and
	\begin{align*}
	 \left|\mu \int_{\R_+}\dfrac{\partial_t v}{v^2}\partial_xg \partial_x \partial_tg\right|
	& \leq  \dfrac{\mu}{4} \int_{\R_+} \dfrac{|\partial_x \partial_t g|^2}{v}
	+ \mu \int_{\R_+} \dfrac{|\partial_t g|^2}{v^3} |\partial_x g|^2 \\
	& \leq \dfrac{\mu}{4} \int_{\R_+} \dfrac{|\partial_x \partial_t g|^2}{v}
	+ C \|\partial_xg\|_{L^\infty(\R_+)}^2 \|\partial_tg\|_{L^2(\R_+)}^2.
	\end{align*}	
	Hence, using Young's inequality and recalling that $\ty'\leq M$,
	\begin{align*}
	& \dfrac{d}{dt} \|\partial_tg(t)\|_{L^2(\R_+)}^2 +\mu \|\partial_t \partial_xg\|_{L^2(\R_+)}^2 \\
	& \leq C \Big[|\ty'(t)|^2 \|\partial^2_x \tw^0(\cdot + \ty(t))\|_{L^2(\R_+)}^2 +|\ty''(t)|(\|\p_x g\|_{L^2(\R_+)} + \|\p_x \bv\|_{L^2(\R_+)})\|\p_t g\|_{L^2(\R_+)}\Big]\\
&	\quad + C(1+ \|\partial_xg\|_{L^2(\R_+)}\|\partial_x^2g\|_{L^2(\R_+)}) \|\partial_t g\|_{L^2(\R_+)}^2\\
&\leq C\Big[ \|\partial^2_x \tw^0(\cdot + \ty(t))\|_{L^2(\R_+)}^2 +|\ty''(t)|^2 + \|\p_t g\|_{L^2(\R_+)}^2\Big]\\
&\quad + C \left(\|\partial_xg\|_{L^2(\R_+)}\|\partial_x^2g\|_{L^2(\R_+)} + \|\p_x g\|_{L^2(\R_+)}^2\right) \|\p_t g\|_{L^2}^2.
	\end{align*}


	Applying Gronwall's inequality, we deduce that 
	 	\begin{align*}
	& \sup_{t \in [0,T]} \|\partial_tg(t)\|_{L^2(\R_+)}^2 + \|\partial_t \partial_xg\|_{L^2([0,T],L^2(\R_+)}^2 \\
	& \leq C\Big[  \|(\partial_tg)_{|t=0}\|_{L^2(\R_+)}^2 + \|\sqrt{x} \partial^2_x \tw^0\|_{L^2(\R_+)}^2 + \|\ty''\|_{L^2(0,T)}^2 + \|\p_t g\|_{L^2([0,T]\times \R_+)}^2\Big] \\
	& \qquad \times e^{C[ \|\partial_xg\|_{L^2L^2}\|\partial^2_xg\|_{L^2L^2} + \|\p_x g\|_{L^2L^2}^2]}
	.	 	\end{align*}

	 We now use the control of $\|\p_xg\|_{L^2(\R_+)}$ and $\|\p_x^2g\|_{L^2(\R_+)}$ provided by Lemma \ref{lem:L2-V}. Note that thanks to the assumption $\inf (1,T) \cE_T\leq 1$, we have
	 $ \|\p_x^2g\|_{L^2((0,T)\times \R_+)}\leq C { E_1}^{1/2}$. Moreover, we have
	 $\|\p_xg\|_{L^2((0,T)\times \R_+)}\leq C\inf(1, T^{1/2})  E_1^{1/2}$ (for small times, we use the $L^\infty(L^2)$ estimate).
	 It follows that
	 \[
	 \exp\left(C[ \|\partial_xg\|_{L^2L^2}\|\partial^2_xg\|_{L^2L^2} + \|\p_x g\|_{L^2L^2}^2]\right) \leq \exp\left( C\inf(1, T^{1/2})  E_1\right) \leq C
	 \]
	 (for some different constant $C$).

	 As a consequence, we obtain
	\begin{align}
	& \sup_{t \in [0,T]} \|\partial_tg\|_{L^2(\R_+)}^2 + \|\partial_t \partial_xg\|_{L^2([0,T],L^2(\R_+)}^2 \label{est:dtv-2}\\
	& \leq C\Big[  \|(\partial_tg)_{|t=0}\|_{L^2(\R_+)}^2 +\| \sqrt{x}\partial^2_x \tw^0\|_{L^2(\R_+)}^2 + \|\ty''\|_{L^2(0,T)}^2 +  E_1\Big].\nonumber
	\end{align}
	
	Next, using \eqref{eq:F}, we have
	\begin{align}
	 \|(\partial_tg)_{|t=0}\|_{L^2(\R_+)}^2 
	& \leq  \|F_{|t=0}\|_{L^2(\R_+)}^2 + \|g_{|t=0}\|_{H^2(\R_+)}^2 \nonumber\\
	&\leq \|\p_x \tw^0\|_{L^2(\R_+)}^2 + C \|\tv^0 - \bv\|_{H^2(\R_+)}^2
	+ C |\ty'(0) -s|^2 .\label{CI-dtv}
	\end{align}
	Gathering 
	\eqref{est:dtv-2} and \eqref{CI-dtv}, we obtain the inequality announced in the statement of the Lemma.

	To complete the proof, it remains to estimate $\partial^2_x g$ in $L^\infty L^2$. 
	For that purpose, we write
	\begin{align*}
	\mu^2 \int_{\R_+} \left|\partial_x\left(\dfrac{\partial_xg}{v} \right)\right|^2(t)
	& = \int_{\R_+} |\partial_t g - F|^2(t) \\
	& \leq 2 \|\partial_tg(t)\|_{L^2(\R_+)}^2 + 2 \|F(t)\|_{L^2(\R_+)}^2
	\end{align*}
	for any $t \in [0,T]$.
	Hence, considering the sup in time and combining with
	\eqref{est:dtv-2}, we get,
	\begin{align*} 
	& \sup_{t \in [0,T]} \dfrac{\mu^2}{4}\left\|\partial_x\left(\dfrac{\partial_xg}{v} \right)\right\|_{L^2(\R_+)}^2   \nonumber \\
	& \leq C\Big[  \|\tv^0 - \bv\|_{H^2}^2  + \|(1+\sqrt{x}) \partial_x \tw^0\|_{L^2 }^2 + \|\sqrt{x} \partial^2_x \tw^0\|_{L^2}^2 \|\ty''\|_{L^2(0,T)}^2+ \|\ty'-s\|_{L^\infty(0,T)}^2\Big].
	\end{align*} 
	Using the compatibility condition $\ty'(0)=-(\p_x \tu^0/\p_x \tv^0)_{|x=0}$, we obtain
	  \begin{align}
    \nonumber\|\ty'-s\|_{L^\infty([0,T])}&\leq |\ty'(0)-s| + C\|\ty'-s\|_{L^2(0,T)}^{1/2} \|\ty''\|_{L^2(0,T)}^{1/2}\\&\leq C\left(\|\tu^0-\bu\|_{H^2} +\|\tv^0-\bv\|_{H^2}+ \|\ty'-s\|_{L^2(0,T)}^{1/2} \|\ty''\|_{L^2(0,T)}^{1/2}\right).\label{est:b-infty}
    \end{align}
	
	Using the estimate on $ E_1$ leads to the desired inequality.

	\medskip

	Let us now say a few words about a more rigorous derivation of these estimates. Once again, we consider the approximation $v^R$ solution to \eqref{eq:v-trunc}. 
	We can perform the previous estimates for $v^R$, and we obtain uniform bounds in $R$. Passing to the limit $R\to +\infty$, we get  inequality~\eqref{est:high-reg-v-I-bis}.
\end{proof}

\bigskip
\begin{lem}[Higher regularity - II]{\label{lem:Hreg-v-2}}
Assume that $\ty'(0)= - \frac{\p_x \tu^0_{|x=0}}{\partial_x \tv^0_{|x=0}}$.
The solution $v$ of~\eqref{eq:v} satisfies
$v -\bv \in L^\infty([0,T], H^3(\R_+))$, and $\partial_t(v-\bv) \in L^\infty([0,T];H^1(\R_+)) \cap L^2([0,T]; H^2(\R_+))$.

Furthermore, the following estimate holds:
assume that $\inf(1, T^{1/2}) \cE_T\leq 1$.
 Then, for some $p> 1$,
\begin{align}\label{est:high-reg-v-II-bis}
     & E_3 := \sup_{t \in [0,T]} \Big[\|\partial_x \partial_t (v-\bv)(t)\|_{L^2(\R_+)}^2 
     +  \|\partial^3_x (v-\bv)\|_{L^2(\R_+)}^2 \Big] \nonumber \\
     & \qquad     + \|\partial^2_t(v-\bv)\|_{L^2([0,T] \times \R_+)}^2 + \|\partial^2_x \partial_t (v-\bv)\|_{L^2([0,T] \times \R_+)}^2 \nonumber \\
    & \quad \leq \cE_T(1+\cE_T)^p.
    \end{align}

\end{lem}

\begin{proof}
    Let us now consider \eqref{eq:dt2-v} as a linear parabolic equation on $\p_t(v-\bv)=\p_t g$, with homogeneous boundary conditions, endowed with the initial data $\ty'(0)\p_x \tv^0 + \mu \p_{xx}\ln \tv^0 + \partial_x \tw^0= \ty'(0)\p_x \tv^0  + \p_x \tu^0$. 
    Thanks to the compatibility condition \eqref{compatibility-ty}, the initial condition belongs to $H^1_0(\R)$, and we can apply estimate \eqref{est:energy-gal-u_x} in Proposition \ref{prop:est-gal} in the Appendix. 
    Indeed, setting $a = 1/v$, we have $\partial_t a = -\frac{\partial_t v}{v^2}$, $\partial_x a = -\frac{\partial_x v}{v^2}$. Using Lemma~\ref{lem:Hreg-v-1}, we infer that $\|\p_t a\|_{L^\infty(L^2)}\leq E_2^{1/2}$ and
    \[
    \ba
    \|\p_x a\|_{L^\infty([0,T];L^4(\R_+))}\leq C\left(1+ E_2^{1/2}\right).
    \ea
    \]
    	
    Moreover, setting 
    \[
    f = \partial_t F- \mu \p_x \left(\frac{\p_t v}{v^2}\p_x (v-\bv)\right),
    \]
    we have,  as shown in the course of the proof of Lemma~\ref{lem:Hreg-v-1},
    \begin{eqnarray*}
    \|f\|_{L^2([0,T] \times \R_+)}
    &\leq& C \|\p_t F \|_{L^2((0,T)\times \R_+)}+ C \|\p_x (v-\bv)\|_{L^\infty} \|\p_x \p_t v\|_{L^2 L^2} \\
    && +C \|\p_t v\|_{L^2 L^\infty}\|\p_x^2(v-\bv))\|_{L^\infty L^2} 
    + C \|\p_x v\|_{L^\infty}  \|\p_t v \|_{L^2 L^2}\|\p_x (v-\bv)\|_{L^\infty }\\
    & \leq & C \cE_T^{1/2} (1+ \cE_T).
    \end{eqnarray*}
    Hence, according to Proposition \ref{prop:est-gal}
    \begin{align*}
    & \sup_{t \in [0,T]} \|\partial_x \partial_t g\|_{L^2(\R_+)}^2
    + \|\partial^2_tg\|_{L^2([0,T] \times \R_+)}^2 + \|\partial^2_x \partial_t g\|_{L^2([0,T] \times \R_+)}^2 \\
    & \leq C 
    \|\partial_x \partial_t g_{|t=0}\|_{L^2(\R_+)}^2
    + C \cE_T (1+ \cE_T)^2.
    \end{align*}
    We have
    \begin{equation*}
    [\partial_t \partial_xg]_{|t=0} - \mu \left[\partial^2_x\left(\dfrac{1}{v}\partial_xg\right)\right]_{t=0}
    = [\partial_x F]_{|t=0}
    \end{equation*}
    where
    \begin{align*}
    \partial_x F 
    & = \partial^2_x \tw^0(\cdot + \ty(t)) + \ty'(t) \partial^2_x g
    + (\ty'-s)\partial^2_x \bv 
    + \mu \partial^2_x\left(\dfrac{\partial_x \bv}{v \bv}g \right),
    \end{align*}
    so that
    \begin{align*}
    \left\|(\partial_{t,x}g)_{|t=0}\right\|_{L^2(\R_+)} 
    &\leq C \Big( \|g_{|t=0}\|_{H^3(\R_+)} (1  + \| g_{|t=0}\|_{H^3(\R_+)}^2)\\&\qquad +  \|\partial^2_x \tw^0\|_{L^2(\R_+)}   +    |\ty'(0)-s| \|\partial^2_x \bv\|_{L^2(\R_+)}
    \Big). 
    \end{align*}
    We then estimate $\ty'(0)-s$ as in \eqref{est:b-infty}, so that $\left\|(\partial_{t,x}g)_{|t=0}\right\|_{L^2(\R_+)}\leq C \cE_0^{1/2} (1+ \cE_0).$
    To conclude, coming back to
    \begin{equation*}
    \mu \partial^2_x\left(\dfrac{1}{v} \partial_{x}g \right)
    = \partial_t (\partial_{x}g) - \partial_{x} F 
    \end{equation*}
    we have
    \begin{align*}
    & \sup_{t \in [0,T]} \mu^2 \int_{\R_+}\left|\partial^2_x\left(\dfrac{1}{v} \partial_{x}g\right)(t)\right|^2 \\
    & \leq 2 \|\partial_t \partial_{x}g \|_{L^\infty([0,T];L^2(\R_+))}^2 
    + 2 \| \partial_{x} F\|_{L^\infty([0,T];L^2(\R_+))}^2,
    \end{align*}
    with
    \begin{align*}
    \|\partial_x F\|_{L^\infty([0,T];L^2(\R_+))}
    & \leq C \Bigg[
    \|\partial^2_x \tw^0\|_{L^2(\R_+)} 
    +\|g\|_{L^\infty([0,T], H^2(\R_+))} \\
    &\qquad+ \|b\|_{L^\infty(0,T)} + \|g\|_{L^\infty((0,T), H^2)}(1 +\|g\|_{L^\infty((0,T), H^2})^2 
        \Bigg].
    \end{align*}
    Finally, using the previous estimates as well as the inequality \eqref{est:b-infty},
    we obtain, for some $p\geq 1$,
   \[\sup_{t \in [0,T]} \|\partial^3_x g\|_{L^2(\R_+)}^2 
    \leq C \cE_T (1+ \cE_T)^p.\]

\end{proof}

Gathering the estimates of Lemmas \ref{lem:L2-V}, \ref{lem:Hreg-v-1} and \ref{lem:Hreg-v-2}, we obtain Proposition \ref{cor:recap-est-v}.

\subsection{Estimates on \texorpdfstring{$h=u-\bu$}{h}}
\label{sec:us}

This subsection is devoted to the study of~\eqref{eq:us-y} satisfied by $u$, which we rewrite in terms of $h:=u-\bu$ 
\be\label{eq:u}
\ba
\partial_t h - \ty'(t) \p_x h - \mu \partial_x\left(\dfrac{1}{v}\partial_x h \right) = (\ty'-s)\p_x \bu + \mu \p_x\left( \left(\frac{1}{v}- \frac{1}{\bv}\right) \p_x \bu\right),\quad t>0, \ x>0,\\
h_{|x=0}=0, \quad \lim_{x\to \infty} h(t,x)=0,\\
h_{|t=0}=\tu^0-\bu,
\ea\ee
where $v$ is the solution of \eqref{eq:vs-y} constructed in the previous subsections. The goal is to prove Proposition \ref{prop:est-u}.

Throughout this subsection, we use the notations introduced in \eqref{def:E0}, \eqref{def:ET}, and we assume that the time $T$ is such that $\inf(1,T) \bar E_1\leq 1$, where $\bar E_1$ was introduced in Lemma \ref{lem:L2-V}. We also assume that assumption \eqref{hyp:ty-M} is satisfied.

\bigskip

Our first result is the classical energy estimate for \eqref{eq:u}:

\begin{lem}[Existence and energy estimates]
\label{lem:u-1}
There exists a unique $u \in \bu + L^\infty([0,T];H^1(\R_+))$ solution to~\eqref{eq:u} which is such that:
\begin{align}
E_4(T)&:= \sup_{t\in [0,T]}\|(u-\bu)(t)\|_{H^1(\R_+)}^2 
+ \|\partial_x (u-\bu)\|_{L^2([0,T];H^1(\R_+))}^2
+ \|\partial_t (u-\bu)\|_{L^2([0,T]\times\R_+)}^2 \nonumber \\
&\leq \cE_T(1+ \cE_T)^{p}
\end{align}
for some $p\geq 1$.
\end{lem}
\begin{proof}
The existence and uniqueness of $u$ in $\bu+L^\infty([0,T], L^2(\R_+))\cap L^2([0,T], H^1(\R_+))$ follows from classical variational arguments.
The energy estimates in $L^\infty([0,T], H^1)\cap L^2([0,T], H^2)$ are a consequence of Proposition \ref{prop:est-gal} in the Appendix, setting $b(t,x)=-\tilde y'(t)$, $c\equiv 0$, $a=1/v$, and 
\[
f= (\ty'-s)\p_x \bu + \mu \p_x\left( \left(\frac{1}{v}-\frac{1}{\bv}\right) \p_x \bu\right).
\]
 We obtain in particular, using the lower bound $a\geq \bar C^{-1}$ and setting $h=u-\bu$,
\begin{align}\label{eq:energy-u}
& \|h\|_{L^\infty([0,T], L^2(\R_+))} 
+ \| \p_xh\|_{L^2((0,T)\times\R_+) } \nonumber \\
& \leq C\left(\|\tu^0-\bu\|_{L^2(\R_+)} + \|\ty'-s\|_{L^2(0,T)}\|\bu-u_+\|_{L^2(\R_+)} + \|v-\bv\|_{L^2([0,T], L^2(\R_+))}\right),
\end{align}
and the right-hand side is bounded by $(\cE_T)^{1/2}$. 

Furthermore, note that with the notations of Lemma \ref{lem:L2-V}
\[\ba
\|\p_t a\|_{L^\infty([0,T], L^2(\R_+))}\leq \|\p_t v\|_{L^\infty([0,T], L^2(\R_+))}\leq  E_1^{1/2},\\
\|\p_x a\|_{L^\infty([0,T], L^4(\R_+))} \leq 
C\big(1 + \|\p_x (v -\bv) \|_{L^\infty([0,T], H^1(\R_+))}\big)\leq  C\big(1+ E_1^{1/2}\big).
\ea
\]
The estimate announced in the Proposition follows easily from Proposition \ref{prop:est-gal}.
\end{proof}

\bigskip
\begin{rmk}
Note that, compared to Proposition~\ref{prop:est-gal} and Lemma~\ref{lem:vR_L2} concerning the $L^\infty L^2$ estimate on $v-\bv$, we do not have the exponential dependency with respect to $T$ in~\eqref{eq:energy-u}.
This results from the divergence structure of the diffusion term ($\partial_x \left(\frac{\partial_xh}{v}\right)$ here, compared to $\partial^2_x \ln \left(\frac{v}{\bv}\right)$ in Lemma~\ref{lem:vR_L2}), but also from the structure of the source term $f$.
Indeed $f= \partial_x \mathcal  F$ with $\mathcal F = (\ty'(t)-s)(\bu - u_+) + \left(\frac{1}{v}-\frac{1}{\bv}\right) \p_x \bu$, therefore
\[
\left|\int_0^T \int_{\R_+} f \ h \ dx dt \right|
\leq C_\eta \|\mathcal F\|_{L^2([0,T]\times \R_+)}^2 +  \eta \|\partial_xh\|_{L^2([0,T]\times \R_+)}^2,
\]
where the second term is absorbed in the left-hand side of the energy inequality.
Hence, compared to Proposition~\ref{prop:est-gal}, we can close the energy estimate without need of $\|u-\bu\|_{L^2([0,T] \times \R_+)}$.
\end{rmk}

\bigskip

\begin{lem}[Improved regularity]
\label{lem:u-2}
Under the compatibility condition \eqref{hyp:compa-2-u}, namely
\[
\partial_x \tu^0_{|x=0}\Big(\ty'(0) -\mu (\partial_x \tv^0)_{|x=0} \Big)
+ \mu (\partial^2_x \tu^0)_{|x=0} = 0,
\]
the solution $u$ to~\eqref{eq:u} is such that:
\begin{align}
E_5(T) &:=\|\partial_{t,x}(u-\bu)\|_{L^\infty[0,T]; L^2(\R_+))}^2 + \|\partial^2_t(u-\bu)\|_{L^2([0,T]\times \R_+)}^2 + \|\partial_t \partial^2_x(u-\bu)\|_{L^2([0,T]\times \R_+)}^2 \nonumber \\
&\leq \cE_T(1+ \cE_T)^p
\end{align}
for some integer $p\geq 1$.
\end{lem}

\begin{proof}
We begin with the control of $\p_t h=\partial_t (u-\bu)$ in $L^\infty([0,T];L^2(\R_+)) \cap L^2([0,T];H^1(\R_+))$ using estimate~\eqref{est:energy-gal-u_xx} from Proposition~\ref{prop:est-gal} and Lemma~\ref{lem:u-1}:
\begin{align}\label{eq:est-dtu-1}
& \|\partial_th\|_{L^\infty([0,T];L^2(\R_+))}^2 
+ \|\partial_{t,x}h\|_{L^2((0,T)\times \R_+)}^2 \nonumber \\
& \leq C \left( \|\tu^0-\bu\|_{H^2(\R_+)}^2 + 
(1+ \|\ty''\|_{L^2(0,T)}^2 + E_2) E_4\right)\nonumber\\
&\qquad + C \|\ty''\|_{L^2(0,T)}^2 + C \|\p_t g\|_{L^2([0,T];H^1(\R_+))}^2 (1 + \|\p_x g\|_{L^\infty([0,T];L^2(\R_+))}^2) \nonumber\\
&\leq \cE_T\left(1 + \cE_T\right)^{p}.
\end{align}
Now, we see $\partial_th$ as the solution to the problem
\begin{align}
\partial_t (\partial_th) -\ty'(t) \partial_{t,x}h - \mu \partial_t \left(\frac{1}{v}\partial_{t,x}h\right)
=  H
\end{align}
with 
\begin{align*}
H= \ty''(t)\partial_xh +\ty''(t)\partial_x \bu- \mu \partial_x\left(\frac{\partial_t v}{v^2} \partial_xh \right)
- \mu \partial_x\left(\frac{\partial_t v}{v^2} \partial_x \bu \right),
\end{align*}
endowed with the initial condition
\begin{align*}
\partial_th_{|t=0} 
& = \ty'(0) \partial_x \tu^0  
+ \mu \partial_x\left(\frac{1}{\tv^0} \partial_x\tu^0 \right) .
\end{align*}
Under the compatibility condition~\eqref{hyp:compa-2-u}, we ensure that $\partial_th_{|t=0, x=0} = 0$.
We can now apply Proposition~\ref{prop:est-gal} and we get thanks to~\eqref{est:energy-gal-u_x}
\begin{align*}
&\|\partial_{t,x}h\|_{L^\infty([0,T]; L^2(\R_+))}^2 + \|\partial^2_th\|_{L^2([0,T]\times \R_+)}^2 + \|\partial_t \partial^2_xh\|_{L^2([0,T]\times \R_+)}^2 \\
& \leq C \Bigg[ \|\partial_t \partial_xh_{|t=0}\|_{L^2(\R_+)}^2 
+ \|\ty''\|_{L^2(0,T)}^2\|\p_xh\|_{L^\infty([0,T];L^2(\R_+))}^2 + \|\ty''\|_{L^2(0,T)}^2\\
& \qquad + \|\p_t v\|_{L^\infty([0,T];H^1(\R_+))}^2 \big(1+ \|\p_x (v-\bv)\|_{L^\infty([0,T];H^1(\R_+))}^2\big) \|\p_xh\|_{L^2([0,T];H^1(\R_+))}^2 \\
& \qquad + \|\p_t v\|_{L^2([0,T];H^1(\R_+))}^2\big(1+ \|v-\bv\|_{L^\infty([0,T];H^1(\R_+))}^2\big) \\
& \qquad + \Big(\|\partial_t v\|_{L^\infty([0,T];L^2(\R_+))} + \|\partial_t v\|_{L^\infty([0,T];L^2(\R_+))}^2+ \|\partial_x v\|_{L^\infty ([0,T];L^4(\R_+))}^2\\&\qquad \quad+ \|\partial_x v\|_{L^\infty ([0,T];L^4(\R_+))}^4 + \|\ty'\|_{L^\infty(0,T)}^2\Big) \|\partial_{t,x}h\|_{L^2([0,T] \times \R_+)}^2  \Bigg].
\end{align*}
The right-hand side is controlled as follows. Proposition \ref{cor:recap-est-v} allows us to upper-bound $v$ and its derivatives.
We also use~\eqref{eq:est-dtu-1} to estimate $\|\partial_{t,x}h\|_{L^2 L^2}$ and Lemma \ref{lem:u-1} to estimate $h$ in $L^\infty(H^1)\cap L^2(H^2)$.
Next, using \eqref{eq:u}, we observe that
\begin{align*}
\partial_t \partial_xh_{|t=0}
&=\ty'(0) \p_x^2(\tu^0 - \bu) + (\ty'(0)-s)\p_x^2\bu + \mu   \partial^2_x\left(\frac{1}{\tv^0}\partial_x (\tu^0 -\bu)\right) + \mu \partial^2_x\left(\left(\frac{1}{\tv^0} - \frac{1}{\bv}\right) \partial_x \bu\right)
\end{align*}
and therefore, using the compatibility condition $\ty'(0)=-\p_x \tu^0_{x=0}/\p_x \tv^0_{|x=0}$,
\begin{align*}
& \|\partial_t \partial_xh_{|t=0}\|_{L^2(\R_+)} \\
& \leq C \left(\|\partial_x (\tu^0-\bu)\|_{H^2(\R_+)} + \|\tv^0-\bv\|_{H^2(\R_+)}\right)\left(1 + |\ty'(0)-s| + \|\tv^0-\bv\|_{H^2(\R_+)}\right) \\
& \qquad + C |\ty'(0) - s|\\
&\leq C \cE_0^{1/2} \left(1+ \cE_0^{1/2}\right).
\end{align*}
Gathering all the estimates, we find that
\[\|\partial_{t,x}h\|_{L^\infty L^2}^2 + \|\partial^2_th\|_{L^2 L^2}^2 + \|\partial_t \partial^2_xh\|_{L^2 L^2}^2 \\
\leq \cE_T\left(1 + \cE_T\right)^p\]
for some large and computable constant $p>1$.

\end{proof}

\section{Existence and uniqueness of local solutions}\label{sec:local}

The purpose of this section is to prove Theorem \ref{thm:main-loc}, following the strategy outlined in section \ref{sec:results}.
We shall construct the solution $\tilde x$ of \eqref{EDO-tx} as the fixed point of a nonlinear application $\cT$, whose definition we now recall:
 $\cT: \ty\in H^2(0,T) \mapsto \tz\in  H^2(0,T)$, where
\[
\tilde z'(t) =-\mu  \frac{ \p_x u(t, 0)}{u_- - \tw^0(\tilde y(t))},\quad \tz(0)=0
\]
and $u$ is the unique solution of \eqref{eq:us-y}.
As a consequence, recalling that $s=-\mu \p_x \bu_{|x=0}/(u_--u_+)$,
\begin{align}\label{y'-s}
\tilde z' - s
& = -\mu \frac{(\p_x u - \p_x \bu)_{|x=0}(t)}{u_--\tw^0(\ty(t))}
    - \mu \partial_x \bu_{|x=0} \dfrac{\tw^0(\ty(t)) - u_+}{(u_- - \tw^0(\ty(t))) (u_- -u_+)} \nonumber \\
& = -\mu \frac{(\p_x u - \p_x \bu)_{|x=0}(t)}{u_--\tw^0(\ty(t))} + s \frac{\tw^0(\ty(t)) - u_+}{u_--\tw^0(\ty(t))}. 
\end{align}

We will first prove that for any initial data satisfying (H1)-(H5), 
 the application $\cT$ has an invariant set provided $T$ is chosen sufficiently small. We will then prove that $\cT$ is a contraction on this invariant set.

Let us define
\begin{multline*}
    \mathcal I_M :=\Big\{\ty\in H^2([0,T]),\ \ty(0)=0, \ \ty'(0)=-\frac{\p_x \tu^0_{|x=0}}{\p_x \tv^0_{|x=0}}\\\text{ and }
M^{-1} \leq \ty'(t)\leq M \ \forall t\in [0,T],
\ \|\ty'-s\|_{H^1([0,T])}\leq M\Big\}.
\end{multline*}
The set $\mathcal I_M$ will be our invariant set provided $M$ is chosen sufficiently large, and $T$ sufficiently small.
Without loss of generality, we will assume that $M \geq 1$.\\
Following the previous section (see \eqref{def:E0} and \eqref{def:ET}), we set
\begin{align*}
    \mathcal E_0&:=\|\tv^0 - \bv\|_{H^3(\R_+)}^2 + \|\tu^0-\bu\|_{H^3(\R_+)}^2  +\|V^0\|_{L^2(\R_+)}^2 \\
&\quad + \|(1+\sqrt{x}) W^0\|_{L^2(\R_+)}^2 + \|(1+\sqrt{x}) \p_x\tw^0\|_{L^2(\R_+)}^2 + \|(1+\sqrt{x})\partial_x^2\tw^0\|_{L^2(\R_+)}^2,\\
\cE_T&:=\cE_0 + \|\ty'-s\|_{H^1([0,T])}^2.
    \end{align*}


Theorem \ref{thm:main-loc} is an immediate consequence of the following result:

\begin{prop}\label{prop:fixed-pt}
Let $M > 1$ be such that
\be\label{def:M}
\sqrt{\frac{2}{M}}\leq \p_x \tv^0_{|x=0}\leq \sqrt{\frac{M}{2}} , \quad \cE_0\leq  M^2, \quad -\sqrt{\frac{M}{2}}\leq \p_x \tu^0_{|x=0}\leq - \sqrt{\frac{2}{M}}.
\ee
Then there exists $T > 0$ depending on $M$ and the parameters of the system, such that $\cT$ has a unique fixed point on $\mathcal{I}_M$.

\end{prop}


\subsection{\texorpdfstring{$\cT$}{T} has an invariant set}

First, notice that by definition, $\tz(0)=0$ and
\[
\tz'(0)= -\mu  \frac{ \p_x \tu^0_{|x=0}}{u_- - \tw^0(0)} \cdot
\]
We recall that $\tw^0=\tu^0 - \mu \p_x \ln \tv^0$, so that $u_- - \tw^0(0) = \mu \p_x \tv^0_{|x=0}$.
It follows that
\[
\tz'(0)=-\frac{\p_x \tu^0_{|x=0}}{\p_x \tv^0_{|x=0}} \cdot
\]

We assume that $M\geq 1$ is chosen so that \eqref{def:M} is satisfied.
Notice that these conditions ensure that
\[
\frac{2}{M} 
\leq \ty'(0) = \tz'(0)
= - \frac{\p_x \tu^0_{|x=0}}{\p_x \tv^0_{|x=0}}\leq \frac{M}{2}.
\]
Since 
\[
\tz'(0) - \sqrt{t}\|\tz''\|_{L^2(0,T)}\leq \tz'(t)\leq \tz'(0) + \sqrt{t}\|\tz''\|_{L^2(0,T)},
\]
it suffices to prove that $\|\tz'-s\|_{H^1(0,T)}\leq M$ for $T$ sufficiently small (depending on $M$), and to further choose $T$ so that $\sqrt{T} M \leq (2M)^{-1}$, i.e. $T \leq (4 M^4)^{-1}.$

Now, let us consider the solutions $v,u$ of \eqref{eq:vs-y}, \eqref{eq:us-y} respectively.
Using the notation of Lemma \ref{lem:L2-V}, we have
\[
 E_1\leq C(\cE_0 + \|\ty'-s\|_{L^2(0,T)}^2)\leq C_M.
\]
(We recall that the constant $C$ in the above inequality may depend on $M$). 
Now, let us choose $T_M$ so that $\inf(1,T_M) C_M\leq 1$. Then the solutions $v,u$ of \eqref{eq:vs-y}, \eqref{eq:us-y}  satisfy the estimates of Propositions \ref{cor:recap-est-v} and \ref{prop:est-u}.

Let us now turn towards the bound on $\|\ty'-s\|_{H^1(0,T)}$. Note that $\cE_T\leq 2M^2$ according to our choice of $M$.
We also note that
\[
u_--\tw^0(0) = \frac{1}{\mu }\p_x \tv^0_{|x=0}\geq \frac{1}{\mu}\sqrt{\frac{2}{M}}.
\]
Since $\|\ty'\|_{L^\infty(0,T)}\leq M$ and $\|\p_x \tw^0\|_\infty \leq \cE_0^{1/2} \leq M$, we obtain, for $t\in [0,T]$,
\be\label{est-lower-bound-u_-w0}
u_--\tw^0(\ty(t))
\geq \frac{1}{\mu}\sqrt{\frac{2}{M}} - \|\ty'\|_\infty t \|\p_x\tw^0\|_\infty 
\geq  \frac{1}{\mu\sqrt{M}}
\ee
provided $T\leq C M^{-5/2}$. 
We infer, using \eqref{y'-s},
\[
\|\tz'-s\|_{L^2(0,T)}\leq C_M \left(\|\p_x (u-\bu)\|_{L^2((0,T), L^\infty)} + T^{1/2} \|\tw^0- u_+\|_\infty\right)\leq C_M \cE_T^{1/2}T^{1/2}\leq C_M T^{1/2}.
\]
Thus, for $T$ sufficiently small, the right-hand side is smaller than $M/2$.

We now consider $\tz''$. We recall that we set $h=u-\bu$. Using that
\begin{align}\label{eq:z''}
  \tz''(t) 
 & = - \mu \frac{(\p_t \p_x u)_{|x=0}}{u_--\tw^0(\ty(t))} 
  + \mu \frac{(\p_x u)_{|x=0}}{(u_--\tw^0(\ty(t)))^2}\ty'(t) \p_x \tw^0(\ty(t)) \nonumber \\
 & = - \mu \frac{\p_t\p_xh_{|x=0}}{u_--\tw^0(\ty(t))} + \mu \frac{\p_xh_{|x=0}}{(u_--\tw^0(\ty(t)))^2}\ty'(t)\p_x \tw^0(\ty(t))\\&\quad  + s \ty'(t) \p_x \tw^0(\ty(t)) \frac{u_--u_+}{(u_--\tw^0(\ty(t)))^2},
 \end{align}
it follows that
\begin{align*}
\|\tz''\|_{L^2(0,T)}
& \leq  C \sqrt{M} \|\p_t \p_x h\|_{L^2((0,T), L^\infty(\R_+))}
+ C M^3 \| \p_x h\|_{L^2((0,T), L^\infty(\R_+))} \\
& \qquad + C M^{3} T^{1/2}.
\end{align*}
We now use Proposition \ref{prop:est-u}: there exists a constant $p\geq 1$ such that
\[
\|\p_t \p_xh\|_{L^\infty((0,T), L^2)}^2 + \|\p_t\p_x^2 h\|_{L^2((0,T)\times \R_+)}^2 \leq C_M M^{2p+2}.
\]
Hence
\begin{align*}
\|\p_t \p_x h\|_{L^2((0,T), L^\infty(\R_+))}
& \leq C\|\p_t \p_x h\|_{L^2((0,T) \times \R_+)}^{1/2} \|\p_t \p^2_x h\|_{L^2((0,T) \times \R_+)}^{1/2} \\
& \leq C_M T^{1/4} M^{p+1}.
\end{align*}
Similarly,
\begin{align*}
\| \p_x h\|_{L^2((0,T), L^\infty(\R_+))}
& \leq T^{1/4} \|\p_xh\|_{L^\infty([0,T];L^2(\R_+))}^{1/2} \|\p_x^2h\|_{L^2([0,T]\times \R_+)}^{1/2} \\
& \leq C_M T^{1/4} M^p.
\end{align*}
Hence, choosing $T$ sufficiently small (depending on $M$), we obtain
\[
\|\tz''\|_{L^2(0,T)}\leq \frac{M}{2}.
\]
We deduce that $\mathcal I_M$ is stable by $\cT$.

\subsection{\texorpdfstring{$\cT$}{T} is a contraction on \texorpdfstring{$\mathcal I_M$}{IM}}
Let $\ty_1, \ty_2\in \mathcal I_M$.
We consider the associated solutions $v_{i}, u_{i}$ of \eqref{eq:vs-y}, \eqref{eq:us-y}, and we set $\tz_i=\cT(\ty_i)\in \mathcal I_M$. We define $\dv=v_1-v_2$, $\du=u_1-u_2$.

Let us now evaluate $\tz_1'-\tz_2'$ and $\tz_1''-\tz_2''$. First, we have
\[
\tz_1'(t)-\tz_2'(t)=-\mu \frac{\p_x D_{u|x=0}}{u_--\tw^0(\ty_1)} - \mu \p_x u_{2|x=0}\frac{\tw^0(\ty_1)-\tw^0(\ty_2)}{(u_--\tw^0(\ty_1))(u_--\tw^0(\ty_2)) },
\]
so that, using Lemma \ref{lem:Taylor} and \eqref{est-lower-bound-u_-w0}
\be\label{est:tz'}
\|\tz_1'-\tz_2'\|_{L^2(0,T)}\leq C \| \p_x D_{u|x=0}\|_{L^2(0,T)} + C \|\p_x u_{2|x=0}\|_{L^\infty(0,T)}\|\ty_1'-\ty_2'\|_{L^2(0,T)}.
\ee
The constant $C$ depends on parameters of the problem, on Sobolev norms of $\tw^0$, and on $M$. 
In a similar way, using identity \eqref{eq:z''}, we find that
\begin{align}\label{est:tz''}
      \|\tz_1''-\tz_2''\|_{L^2(0,T)}
      & \leq  C \left(\|\p_t \p_x D_{u|x=0}\|_{L^2(0,T)} + \| \p_x D_{u|x=0}\|_{L^2(0,T)}\right)\\
      & \quad + C \left(  \|\p_x u_{2|x=0}\|_{L^\infty(0,T)} 
      + \|\p_t\p_x u_{2|x=0}\|_{L^2(0,T)} \right)\|\ty_1'-\ty_2'\|_{L^2(0,T)}.\nonumber
\end{align}

There remains to evaluate the traces of $\p_t\p_x \du$ and $\p_x \du$ at $x=0$ in terms of $\ty_1-\ty_2$.
In order to do so, we follow the order of the energy estimates in the previous sections and start by evaluating $\dv$.
Note that $\dv$ is a solution of
\[
\ba
\p_t \dv - \ty_1'\p_x \dv - \mu \p_{xx}\ln \left(1+ \frac{\dv}{v_2}\right)= (\ty_1'-\ty_2')\p_x v_2,\\
D_{v|x=0}=0,\quad \lim_{x\to \infty}\dv(t,x)=0,\\
D_{v|t=0}=0.
\ea
\]

We  apply Lemma \ref{lem:generique-g} in the Appendix to the function $\dv$, with $\bar g=v_2$, $G= (\ty_1'-\ty_2')\p_x v_2$.
Since $T\leq T_M$, $\exp(\|\p_x v_2\|_{L^\infty((0,T)\times \R_+)} T_M)\leq C_M$, and 

It follows that
\begin{align*}
&\|\dv\|_{L^\infty((0,T), H^1(\R_+))} + \|\p_t \dv\|_{L^2((0,T)\times \R_+)}
+ \|\p_x \dv\|_{L^2((0,T)\times \R_+)} \\
&\leq C\|\ty_1'-\ty_2'\|_{L^2(0,T)}\left(\|\p_x (v_2 - \bv)\|_{L^\infty((0,T), L^2(\R_+))} 
+ \|\p_x \bv\|_{L^2(\R_+)} \right)\exp\left(\|\p_x v_2\|_{L^\infty((0,T)\times \R_+)} T\right)\\
&\leq C_M \|\ty_1'-\ty_2'\|_{L^2(0,T)},
\end{align*}
and
\begin{align*}
\left\| \p_x^2 \left(\frac{\dv}{v_2}\right)\right\|_{L^2((0,T)\times \R_+)} 
& \leq C_M \left(\|\ty_1'-\ty_2'\|_{L^2(0,T)}  + T^{1/2}\|\ty_1'-\ty_2'\|_{L^2(0,T)}^2\right)
\\
&\leq   C_M \|\ty_1'-\ty_2'\|_{L^2(0,T)}
\end{align*}
Following the estimates of Section~\ref{sec:regularity-v} (see also Proposition \ref{prop:est-gal} in the Appendix), we infer that 
\begin{align*}
    &\|\p_t \dv \|_{L^\infty((0,T), H^1(\R_+))}^2 + \|\p_x \p_t \dv\|_{L^2((0,T), H^1(\R_+))}^2 + \|\p_x^2 \dv\|_{L^\infty((0,T), H^1(\R_+))}^2\\
    &\leq C_M ( \|\ty_1'-\ty_2'\|_{L^2(0,T)}^2  + \|\ty_1''-\ty_2''\|_{L^2(0,T)}^2 ).
    \end{align*}
    
We now turn towards the estimates on $\du$, which satisfies the equation
\[
\p_t \du - \ty_1'\p_x \du -\mu \p_x \left(\frac{1}{v_1}\p_x \du\right)= S
\]
with
\[
S:=(\ty_1'-\ty_2')\p_x u_2 + \mu \p_x \left(\left(\frac{1}{v_1}- \frac{1}{v_2}\right) \p_x u_2\right).
\]
In order to complete the energy estimates on $\du$, we need to evaluate $S$ and $\p_t S$ in $L^2([0,T]\times \R_+)$. We have
\begin{align*}
\|S\|_{L^2([0,T]\times \R_+)}
&\leq  \|\p_x u_2\|_{L^\infty ([0,T],L^2(\R_+))}\|\ty_1'-\ty_2'\|_{L^2(0,T)}\\
&\quad  + C \Big(\|\dv\|_{L^\infty([0,T]\times \R_+)}\|\p_{x}^2 u_2\|_{L^2([0,T]\times \R_+)} \\
& \qquad + \|\partial_x\dv\|_{L^\infty([0,T],L^2(\R_+))}\|\p_x u_2\|_{L^2([0,T],L^\infty(\R_+))} \\
& \qquad  + \|\dv\|_{L^\infty([0,T]\times \R_+)} \|\partial_x(v_1 + v_2)\|_{L^\infty([0,T]\times \R_+)} \|\p_x u_2\|_{L^2([0,T],L^\infty(\R_+))}\Big) \\
&\leq  C_M \|\ty_1'-\ty_2'\|_{L^2(0,T)}
\end{align*}
and similarly,
\begin{align*}
\|\p_t S\|_{L^2([0,T]\times \R_+)}
&\leq C_M \|\ty_1'-\ty_2'\|_{H^1(0,T)}.
\end{align*}
Using once again Proposition \ref{prop:est-gal} in the Appendix, we obtain
\begin{align*}
&\|\du\|_{L^\infty((0,T), H^1(\R_+))} 
+ \|\p_x \du\|_{L^2((0,T),H^1(\R_+))} + \|\p_t \du\|_{L^2((0,T)\times \R_+)}\\
& \leq C_M \|\ty_1'-\ty_2'\|_{L^2(0,T)}
\end{align*}
and
\[\|\p_t \du\|_{L^\infty((0,T), H^1(\R_+))} + \|\p_x\p_t \du\|_{L^2((0,T), H^1(\R_+))}\\
 \leq C_M \|\ty_1'-\ty_2'\|_{H^1(0,T)}.\]

We are now ready to prove the contraction property. We focus on the estimate of $\tz_1''-\tz_2''$, since
\[
\|\tz_1'-\tz_2'\|_{L^2(0,T)}\leq \frac{T}{\sqrt{2}}\|\tz_1''-\tz_2''\|_{L^2(0,T)}.
\]
Using the estimates on $\du$, we obtain
\begin{align*}
&\|(\p_t \p_x \du)_{|x=0}\|_{L^2(0,T)} + \| \p_x D_{u|x=0}\|_{L^2(0,T)}\\
&\leq  C\left(\|\p_t \p_x \du\|_{L^2((0,T)\times \R_+) }^{1/2} \|\p_t \p_x^2 \du\|_{L^2((0,T)\times \R_+) }^{1/2} + \| \p_x \du\|_{L^2((0,T)\times \R_+) }^{1/2}\| \p_x^2 \du\|_{L^2((0,T)\times \R_+) }^{1/2} \right)\\
&\leq  CT^{1/4} \Big(\|\p_t \p_x \du\|_{L^\infty((0,T), L^2(\R_+)) }^{1/2} \|\p_t \p_x^2 \du\|_{L^2((0,T)\times \R_+) }^{1/2}\\
&\qquad\qquad\qquad + \| \p_x \du\|_{L^\infty((0,T), L^2(\R_+)) }^{1/2}\| \p_x^2 \du\|_{L^2((0,T)\times \R_+)}^{1/2} \Big)\\
&\leq  C_M T^{1/4} \|\ty_1'-\ty_2'\|_{H^1(0,T)}.
\end{align*}
Furthermore,
\[
\|\p_x u_{2|x=0}\|_{L^\infty(0,T)} + \|\p_t \p_x u_{2|x=0}\|_{L^2(0,T)} 
\leq |\p_x \bu_{|x=0^+}| + C\cE_T (1 + \cE_T)^p 
\leq C_M.
\]
It follows that
\[
\|\tz_1''-\tz_2''\|_{L^2(0,T)}
\leq C_M T^{1/4}\|\ty_1'-\ty_2'\|_{H^1(0,T)} 
+ C_M \|\ty_1'-\ty_2'\|_{L^2(0,T)}
\leq C_M   T^{1/4}\|\ty_1'-\ty_2'\|_{H^1(0,T)}.
\]
Hence, for $T$ sufficiently small (depending on $M$), $\cT$ is a contraction on $\mathcal I_M$.

\subsection{Conclusion}

We deduce from the previous subsections that there exists $T_M>0$ such that for all $T\leq T_M$, $\cT$ has a unique fixed point in $\mathcal I_M$. Let us consider the solution $(v_s, u_s)$ associated with this fixed point. 
According to the argument at the end of section \ref{sec:results}, $(\tx, v_s, u_s)$ is a solution of \eqref{eq:vs}.
Furthermore, $v_s$ and $u_s$ satisfy the properties listed in Propositions \ref{cor:recap-est-v} and \ref{prop:est-u} respectively, and $v_s$ also satisfies the $L^\infty$ estimates of Lemma \ref{lem:infty-est}, provided $T_M$ is small enough.

Thus $(\tx, v_s, u_s) $ satisfy all the properties listed in Theorem \ref{thm:main-loc}. Eventually, the pressure in the congested domain is given by
\[
p_s(t,x)=p_s(t)=\tx'(t) (u_- - \tw^0(\tx(t))).
\]
Since $\tx\in H^2(0,T)$, $p_s\in H^1(0,T)$. This completes the proof of Theorem \ref{thm:main-loc}.\qed

\section{Global solutions: existence for small data and stability} \label{sec:global}

In this section, we start from a strong solution $(u_s,v_s)$ provided by Theorem \ref{thm:main-loc} (see the previous section).
We  prove that if $\cE_0$ is small enough, the existence time of the solution is infinite. Furthermore, the travelling wave is asymptotically stable.
Once again, we drop the indices $s$ throughout the section in order to alleviate the notation.

Let us now introduce our setting. 
We assume that $\cE_0\leq c_0\delta^2$, where $\delta$ is a small constant, $\cE_0$ is the initial energy, defined in \eqref{def:E0} and $c_0$ is a constant depending only on the parameters of the problem, to be defined later on. At this stage, we merely choose $c_0$ small enough so that $\|\tv_0-\bv\|_{L^\infty(\R_+)}\leq \delta/2$.
We denote by $T^*$ the maximal existence time of the solution.
For $T\in ]0, T^*[$ small enough, using a continuity argument together with the dominated convergence theorem, $\| \tx'-s\|_{H^1([0,T])} + \|v -\bv\|_{L^\infty([0,T] \times \R_+)} \leq \delta$. We set
\[
\bar T:=\sup\left\{T\in ]0, T^*[, \ \|\tx'-s\|_{H^1([0,T])}\leq \delta\right\}.
\]
The global existence of the solution relies on the following bootstrap result:
\begin{prop}
There exist constants $c_0,\delta_0>0$, depending only on the parameters of the problem $s,\mu, v_+$, such that the following result holds.

For all $\delta\in (0, \delta_0)$, if $\cE_0\leq c_0\delta^2$,  and $\|(1+\sqrt{x}) \p_x^k \tw^0\|_{L^2(\R_+)} \leq c_0 \delta^{3/2}$ for $k=1,2,3$, then 
\[
\forall T\in [0, \bar T],\quad \|\tx'-s\|_{H^1([0,T])}\leq\frac{\delta}{2}.
\]

\label{prop:bootstrap}

\end{prop}
The proof of Proposition \ref{prop:bootstrap} is the main purpose of this section. Before describing the strategy of the proof, let us deduce the global existence result of Theorem \ref{thm:main-glob} from Proposition \ref{prop:bootstrap}.
First, it is clear from Proposition \ref{prop:bootstrap} and from a simple continuity argument that $\bar T=T^*$.
Second, notice that for all      $T< \bar T$, $\cE_T=\cE_0 + \|\tx'-s\|_{H^1(0,T)}^2 \leq 2 \delta^2 \leq 1$ provided $\delta_0\leq 1/\sqrt{2}$ and $c_0\leq 1$. Hence the total energy remains bounded on $(0, T^*)$. Classically, this implies that the solution is global. Let us explain why in the present context.

First, note that  for all $t\in [0, \bar T[$,
\be\label{bound:tx'}
\frac s 2 \leq s - \|\tx'-s\|_{L^\infty}\leq   \tx'(t)\leq s + \|\tx'-s\|_{L^\infty}\leq \frac{3s}{2}
\ee
provided $\delta$ is sufficiently small.

As a consequence, as long as $1\leq v\leq 2(v_++1)$, the estimates of Propositions \ref{cor:recap-est-v} and \ref{prop:est-u} hold, with a constant $C$ depending only on the parameters of the problem $s,\mu, v_+, u_\pm$ (note that we chose here $\bar C=2 (v_++1)$).
Hence, let us introduce
\[
\bar T':=\sup \left\{ T\in (0, \bar T),\quad 1\leq v\leq 2( v_+ + 1) \right\}.
\]
By continuity, if $\delta$ is small enough, we have $\bar T'>0$. Furthermore, there exists a constant $C$ depending only on the parameters of the problem such that
\[
\|\p_x (v-\bv)\|_{L^\infty((0, \bar T')\times \R_+)}\leq C \delta.
\]
Consequently, for all $t\in [0, \bar T']$,
\[
\inf_{x\in [0,1]} \p_x v(t,x)\geq \inf_{x\in [0,1]} \p_x \bv(x) - \|\p_x (v-\bv)\|_\infty \geq \frac{1}{2}\inf_{x\in [0,1]} \p_x \bv(x)>0
\]
provided $\delta$ is small enough.

Now, let us prove that $\bar T'=\bar T$ provided $\delta$ is small enough. We argue by contradiction and assume that $\bar T'<\bar T$. 
We then consider the Cauchy problem at $t=\bar T'$, and we apply Lemma \ref{lem:infty-est}. Then \eqref{hyp:M} is satisfied, with a constant $M$ depending only on the parameters of the problem. Furthermore, 
\[
\sup_x v(\bar T', x) \leq \sup \bv + \|v(\bar T')-\bv\|_{L^\infty(\R_+)}\leq v_+ + C\cE_{\bar T'} \leq v_+ + \frac{1}{2}
\]
provided $\delta$ is small enough.
As a consequence, there exists a time $\tau>0$, depending only on the parameters of the problem, such that
\[
1<v(t,x)\leq 2\left(v_+ + \frac{1}{2}\right)=2 v_+ +1
\]
for all $t\in [\bar T', \inf(\bar T, \bar T' + \tau)]$. This contradicts the definition of $\bar T'$, and we infer that $\bar T'=\bar T$.

 Thus the constants in Propositions \ref{cor:recap-est-v} and \ref{prop:est-u} and in all Lemmas of section \ref{sec:estimates} depend only on $s,\mu, v_+, u_\pm$ for all estimates bearing on the interval $[0, \bar T]$.   
 At last, let us emphasize that since $\cE_T\leq 1$ for all $T\leq \bar T$, the condition $\inf(1,T^{1/2}) E_1(T)\leq 1$ of Lemmas \ref{lem:Hreg-v-1}, \ref{lem:Hreg-v-2} etc. is always satisfied.

It then follows from Proposition \ref{prop:fixed-pt} that the time $T_M$ on which the fixed point argument is valid depends only on the parameters of the problem $s,\mu, v_+,u_\pm$. By a classical induction argument, we may solve the Cauchy problem on $[nT_M, (n+1)T_M]$ for all $n\geq 0$, and we deduce that $T^*=+\infty$.

\medskip

Let us now explain our strategy of proof of Proposition \ref{prop:bootstrap}. It relies on two sets of estimates:
\begin{itemize}
    \item The first set of estimates was obtained in Propositions \ref{cor:recap-est-v} and \ref{prop:est-u}. It ensures that if $\delta$ is small enough, as long as $T\leq \bar T$, setting $g:=v-\bv$, $h := u -\bu$,
    \begin{align}
    \|g\|_{L^\infty([0,T], H^3(\R_+))} 
    + \|\partial_t g\|_{L^\infty([0,T], H^1(\R_+))} + \|\p_t g\|_{L^2((0,T), H^2(\R_+))} \nonumber \\
    \leq  C (\cE_0^{1/2} + \|\tx'-s\|_{H^1(0,T)}),\label{est:g-global}\\
    \|h\|_{L^\infty([0,T], H^1(\R_+))}+ \|\partial_t h \|_{L^2([0,T], H^2(\R_+))} + \|\partial^2_t h\|_{L^2((0,T)\times \R_+)} \nonumber \\
    \leq  C (\cE_0^{1/2} + \|\tx'-s\|_{H^1(0,T)}),\label{est:h-global}
    \end{align}
    where the constant $C$ depends only on the parameters of the problem, namely $s,\mu,v_+$ and $u_+$, as explained above (recall that $\cE_T\leq 1$ for all $T<\bar T$).

    From there, using the equations \eqref{eq:g} and \eqref{eq:u} on $g$ and $h$ respectively, we also deduce that for all $T\leq \bar T$,
    \begin{align}
    \|g\|_{L^2((0,T), H^4(\R_+))} + \|h\|_{L^2((0,T), H^4(\R_+))} \nonumber \\
    \leq  C \left(\cE_0^{1/2} +\|(1+\sqrt{x})\p_x^3 \tw^0\|_{L^2(\R_+)}+ \|\tx'-s\|_{H^1(0,T)}\right).\label{est:g-h-global-bis}
   \end{align}

    
    \item The second set of estimates relies on a coercivity inequality for the linearized operator around $\bv$, which will be the main focus of this section. One crucial observation lies in the fact that $\p_x \bv$ belongs to the kernel of this linearized operator\footnote{Note that this is a classical property of nonlinear equations with constant coefficients, linearized around a given stationary solution. It is related to the (formal) space invariance of the equation.}. This allows us to define a new unknown
   \be\label{def:g1}
    g_1:= - s(\tv - \bv) - \mu \p_x\left(\frac{\tv - \bv}{\bv}\right),
    \ee
    which will satisfy better estimates than $g=\tv-\bv$ (cf Remark~\ref{rmk:g1}).
    
\end{itemize}

In order to keep the presentation as simple as possible, we will start with the case when $\tw^0\equiv u_+$, which contains the main ideas of the proof.
In subsection \ref{ssec:w0-qcq}, we will address the case of non-constant $\tw^0$, and we will point out the main differences with the simplified case.
We conclude this section with a proof of the long time stability in subsection \ref{ssec:stability}

\subsection{Case \texorpdfstring{$\tw^0\equiv u_+$}{w0 cst}}

In this section, we assume that $\tw^0$ is constant and equal to $u_+$. The equation satisfied by $v$ on $[0,T^*[$ is therefore
\[\ba
\p_t v - \tx' \p_x v - \mu \p_x^2 \ln v=0,\\
v_{|x=0}=1, \quad \lim_{x\to \infty}v(t,x)=v_+,\\
v_{|t=0}=\tv^0.
\ea
\]
In the rest of this section, we introduce the linearized operator around $\bv$, namely $\p_x \cA$, where
\[
\cA:=-s \mathrm{Id} - \mu \p_x\left( \frac{\cdot}{\bv}\right).
\]
We also set $\beta(t)=\tx'(t)-s$ and $g=v-\bv$, as in the previous sections. The equation on $g$ can be written as
\be\label{eq:g0}\ba
\p_t g + \p_x \cA g=  \beta \p_x \bv + \beta \p_x g + \mu \p_x^2\left(\ln\left(1+ \frac{g}{\bv}\right) - \frac{g}{\bv}\right),\\
g_{|t=0}=\tv^0-\bv,\\
g_{|x=0}=0.
\ea
\ee
Let us comment a little on the structure of this equation. We will prove that the operators $\p_x \cA$ and $\cA \p_x$ enjoy nice coercivity properties (see Lemma \ref{lem:coerc-A} below). The second and third terms in the right-hand side of \eqref{eq:g0} are quadratic and will be treated perturbatively, using the first set of estimates on $g$, namely \eqref{est:g-global} and \eqref{est:g-h-global-bis}. Eventually, $\p_x \bv\in \ker \cA$. Hence the first term in the right-hand side disappears when $\cA$ is applied to the equation. As a consequence, $g_1=\cA g$ satisfies the equation
\be\label{eq:g1}
\ba
\p_t g_1 + \cA \p_x g_1= \beta \cA \p_x g + \mu \cA \p_x^2\left(\ln\left(1+ \frac{g}{\bv}\right) - \frac{g}{\bv}\right),\\
g_{1|t=0}=\cA\left(\tv^0-\bv\right).\ea
\ee
From definition~\eqref{def:g1} and estimates \eqref{est:g-global}, \eqref{est:g-h-global-bis}, we know that 
\begin{equation}\label{reg:g1-1}
\begin{aligned}
g_1\in L^\infty((0, \bar T), H^2(\R_+)) , \quad g_1\in L^2((0,\bar T), H^3(\R_+))\\ \text{and that} \quad \p_t g_1\in L^\infty((0,\bar T), L^2(\R_+))\cap L^2((0, \bar T), H^1(\R_+)).  
\end{aligned}
\end{equation} 
Using the identity $u-\mu\p_x \ln v=u_+$, we also infer that $h-\mu\p_x \ln (1+g/\bv)=0$. Using \eqref{est:h-global} and \eqref{est:g-h-global-bis}, we deduce that 
\begin{equation}\label{reg:g1-3}
g_1 \in  L^2((0, \bar T), H^4(\R_+)),\quad \p_t g_1\in L^2((0, \bar T), H^2(\R_+)) \quad \text{and} \quad \p^2_t g_1\in L^2((0, \bar T) \times \R_+)
\end{equation}

\begin{rmk}\label{rmk:g1}
In Equation~\eqref{eq:g1}, all the terms in the right-hand side are quadratic (recall that $\|g\| \lesssim \cE_0^{1/2} + \|\beta\|$), so that $g_1$
can be expected to satisfy better estimates than $g$.
\end{rmk}

\bigskip

Before addressing the estimates on $g_1=\cA g$, let us derive some information on the traces of $\p_x^k g_1$ at $x=0$ for $k=0,1, 2$.
Since $\tw^0=u_+$, we know that
\[
u(t,x)- \mu \p_x \ln v(t,x)=u_+\quad \forall t>0, \ x>0.
\]
As a consequence,
\[
\p_x v_{|x=0}= \frac{u_--u_+}{\mu}=\p_x \bv_{|x=0},
\]
and thus $g_{1|x=0}=\p_x g_{|x=0}=0$. Taking the trace of \eqref{eq:g0} at $x=0$ (which is legitimate since all terms belong to $L^2((0, \bar T), H^2(\R_+))$, we find that
\be\label{trace-dx-g1}
\p_x g_{1|x=0}= \beta \p_x \bv_{|x=0}= \beta \frac{s(v_+-1)}{\mu}.
\ee
We now take the trace of \eqref{eq:g1} at $x=0$, noticing that all terms in \eqref{eq:g1} belong to $L^2((0,T), H^1(\R_+))$. Note that
\be\label{comm-A-dx}
\cA\p_x g= \p_x g_1 + [\cA, \p_x] g= \p_x g_1 + \mu \p_x \left(\frac{\p_x \bv}{\bv^2} g\right).
\ee
Furthermore, $\ln (1+X)-X$ is quadratic close to $X=0$, and we recall that $(g)_{|x=0} = (\p_x g)_{|x=0}=0$. 
It follows that 
\[
\left(\cA \p_x^2\left(\ln\left(1+ \frac{g}{\bv}\right) - \frac{g}{\bv}\right)\right)_{|x=0}=0.
\]
We infer from~\eqref{eq:g1} and~\eqref{comm-A-dx} that
\[
(\cA \p_x g_1)_{|x=0}
= \beta (\cA \partial_x g)_{|x=0}
= \beta \p_x g_{1|x=0}= \beta^2 \frac{s(v_+-1)}{\mu},
\]
and thus
\begin{eqnarray}\label{bc:p2xg1-0}
\mu \p_x \left(\frac{\p_x g_1}{\bv}\right)_{|x=0}
&=&- s \p_x g_{1|x=0} - \beta \p_x g_{1|x=0} \label{eq:px_g1_0}\\
&=&- \beta \frac{s^2(v_+-1)}{\mu} - \beta^2\frac{s(v_+-1)}{\mu}.
\end{eqnarray}

Let us now state a Lemma giving some coercivity properties on $\cA$:
\begin{lem}[Coercivity properties of $\p_x \cA$ and $\cA \p_x$]
\label{lem:coerc-A}~
\begin{itemize}
\item Estimate for $\cA \p_x$: let $\varphi \in H^2(\R_+)$. Then
\begin{align*}
\int_0^\infty (\cA \p_x \varphi )\; \varphi 
& = \mu\int_0^\infty \frac{(\p_x \varphi )^2}{\bar v} + \frac{s}{2}(\varphi(0))^2 + \mu \varphi'(0) \varphi(0) \\
& \geq \frac{\mu}{v_+} \int_0^\infty (\p_x \varphi )^2+ \frac{s}{2}(\varphi(0))^2 + \mu \varphi'(0) \varphi(0).
\end{align*}
\item Estimate for $ \p_x\cA$: let $\varphi \in H^2(\R_+)$. Then, for any weight function $\rho\in \mathcal C^2_b(\R_+)$ such that $\rho>0$ a.e.
\begin{eqnarray*}
\int_0^\infty (\p_x \cA \varphi ) \frac{\varphi }{\bv} \rho &\geq& \mu \int_0^\infty \left( \p_x \frac{\varphi }{\bv}\right)^2 \rho -C\|\rho\|_{W^{2,\infty}}\int_0^\infty \varphi ^2 \\&&+\varphi ^2(0)\left( \frac{s}{2} \rho(0)- \frac{\mu}{2}\rho'(0)\right) + \mu\p_x\left(\frac{\varphi }{\bv}\right)_{|x=0} \varphi (0) \rho(0) .
\end{eqnarray*}

\end{itemize}

\end{lem}
The proof of Lemma \ref{lem:coerc-A} is straightforward and provided in Appendix \ref{app:lem:coerc-A} for the reader's convenience.

We are now ready to tackle the estimates on $g_1$. In the following, we denote by $C$ any constant depending only on the parameters of the problem (i.e. $v_+, u_\pm, s, \mu$).

\bigskip
$\rhd${\it $L^\infty(L^2)\cap L^2(H^1)$ estimate:}

We multiply equation \eqref{eq:g1} by $g_1$. Using the first point in Lemma \ref{lem:coerc-A} together with the observation \eqref{comm-A-dx}, we obtain
\begin{equation}\label{eq:energ-g1-0}
\frac{1}{2}\frac{d}{dt}\int_0^\infty g_1^2 + \frac{\mu}{v_+}\int_0^\infty (\p_x g_1)^2\leq \left| \mu \beta \int_0^\infty\p_x \left(\frac{\p_x \bv}{\bv^2} g\right) g_1\right| + \mu \left|  \int_0^\infty\cA \p_x^2 \left(\ln\left(1+ \frac{g}{\bv}\right) - \frac{g}{\bv}\right) g_1\right|.
\end{equation}
We now evaluate separately the two terms in the right-hand side.
\begin{itemize}
    \item Integrating by parts the first term, we obtain
    \[
    \left| \int_0^\infty\p_x \left(\frac{\p_x \bv}{\bv^2} g\right) g_1\right|\leq C \|\p_x g_1\|_{L^2(\R_+)}\|g\|_{L^2(\R_+)}.
    \]
    Recalling that
    \[
    \|g\|_{L^\infty([0,\bar T], L^2(\R_+))}^2\leq \cE_0 + C\|\beta\|_{L^2(0, \bar T)}^2,
    \]
    we obtain
    \[
    \left| \mu \beta \int_0^\infty\p_x \left(\frac{\p_x \bv}{\bv^2} g\right) g_1\right|\leq \frac{\mu}{4v_+}\|\p_x g_1\|_{L^2(\R_+)}^2 + C(\cE_0 + \|\beta\|_{L^2(0, \bar T)}^2)|\beta|^2.
    \]
    
    \item Using the definition of the differential operator $\cA$ and integrating by parts, we have
    \begin{eqnarray*}
    &&\left|  \int_0^\infty\cA \p_x^2 \left(\ln\left(1+ \frac{g}{\bv}\right) - \frac{g}{\bv}\right) g_1\right|\\
    &\leq &C \left|  \int_0^\infty \p_x \left(\ln\left(1+ \frac{g}{\bv}\right) - \frac{g}{\bv}\right) \p_x g_1\right|+ C \left|  \int_0^\infty \p_x^2 \left(\ln\left(1+ \frac{g}{\bv}\right) - \frac{g}{\bv}\right) \p_x g_1\right|.
    \end{eqnarray*}
We now consider the nonlinear term. We have
    \be\label{diff-quadratic}\ba
    \p_x \left(\ln\left(1+ \frac{g}{\bv}\right) - \frac{g}{\bv}\right)= -\p_x \left( \frac{g}{\bv}\right) \frac{\frac{g}{\bv}}{1+ \frac{g}{\bv}},\\
     \p_x^2 \left(\ln\left(1+ \frac{g}{\bv}\right) - \frac{g}{\bv}\right)= 
     -\p_x^2 \left( \frac{g}{\bv}\right) \frac{\frac{g}{\bv}}{1+ \frac{g}{\bv}} - \left( \p_x \left( \frac{g}{\bv}\right)\right)^2 \frac{1}{(1+ \frac{g}{\bv})^2}.\ea
    \ee
Consequently, recalling that $\bv + g= v\geq 1$, we have
    \begin{eqnarray*}
    &&\left|  \int_0^\infty\cA \p_x^2 \left(\ln\left(1+ \frac{g}{\bv}\right) - \frac{g}{\bv}\right) g_1\right|\\
    &\leq & C \|\p_x g_1\|_{L^2(\R_+)}\left(\|g\|_{L^\infty(\R_+)} \left\| \p_x \left( \frac{g}{\bv}\right) \right\|_{H^1(\R_+)} + \left\| \p_x \left( \frac{g}{\bv}\right) \right\|_{L^4(\R_+)}^2\right)\\
    &\leq & C \|\p_x g_1\|_{L^2(\R_+)}\|g\|_{W^{1,\infty}(\R_+)} \left\| \p_x \left( \frac{g}{\bv}\right) \right\|_{H^1(\R_+)}.
     \end{eqnarray*}
    Since
    \[
    \|g\|_{L^\infty([0,\bar T], W^{1,\infty}(\R_+))}\leq \|g\|_{L^\infty([0,\bar T], H^2(\R_+))}\leq C(\cE_0^{1/2} + \|\beta\|_{H^1(0, \bar T)}),
    \]
    we obtain eventually
    \[
    \left| \mu \int_0^\infty\cA \p_x^2 \left(\ln\left(1+ \frac{g}{\bv}\right) - \frac{g}{\bv}\right) g_1\right|
    \leq \frac{\mu}{4 v_+}\|\p_x g_1\|_{L^2(\R_+)}^2 + C (\cE_0 + \|\beta\|_{H^1(0, \bar T)}^2)\left\| \p_x \left( \frac{g}{\bv}\right) \right\|_{H^1(\R_+)}^2.
    \]
\end{itemize}
Gathering all the terms, we are led to
\[
\frac{d}{dt}\int_0^\infty g_1^2 + \frac{\mu}{v_+}\int_0^\infty (\p_x g_1)^2
\leq  C(\cE_0 + \|\beta\|_{H^1(0, \bar T)}^2)\left(|\beta|^2 + \left\| \p_x \left( \frac{g}{\bv}\right) \right\|_{H^1(\R_+)}^2\right).
\]
Integrating in time and using the bound
\[
\|g\|_{L^2([0, \bar T], H^2(\R_+))}\leq C(\cE_0^{1/2} + \|\beta\|_{H^1(0, \bar T)}),
\]
we get
\be\label{borne-g1-L2H1}
\|g_1\|_{L^\infty([0,\bar T], L^2(\R_+))}^2 + \frac{\mu}{v_+}\|\p_x g_1\|_{L^2([0,\bar T], L^2(\R_+))}^2 \leq \cE_0 + C (\cE_0 + \|\beta\|_{H^1(0, \bar T)}^2)^2.
\ee

\bigskip

$\rhd${\it $L^\infty(H^1)\cap L^2(H^2)$ estimate:}
Differentiating \eqref{eq:g1} with respect to $x$, we obtain, using \eqref{comm-A-dx}
\be\label{eq:dxg1}
\p_t \p_x g_1 + \p_x \cA \p_x g_1 = \beta \p_x^2 g_1 + \mu \beta \p_x^2\left(\frac{\p_x \bv}{\bv^2} g\right) + \mu \p_x \cA \p_x^2 \left(\ln\left(1+ \frac{g}{\bv}\right) - \frac{g}{\bv}\right).
\ee
We also recall~\eqref{eq:px_g1_0}:
\[
\p_x g_{1|x=0}= \beta \frac{s(v_+-1)}{\mu},\quad 
\mu \p_x \left(\frac{\p_x g_1}{\bv}\right)_{|x=0} =- s \p_x g_{1|x=0} - \beta \p_x g_{1|x=0}.
\]
We multiply \eqref{eq:dxg1} by $\p_x g_1 \rho/\bv $, where $\rho\in \mathcal C^2_b(\R_+)$ is a weight function that will be chosen momentarily, and such that $\rho\geq 1$.  Using the second part of Lemma \ref{lem:coerc-A} and replacing $\mu \p_x \left(\frac{\p_x g_1}{\bv}\right)_{|x=0}$, we have
\begin{eqnarray*}
\int_0^\infty \left( \p_x \cA \p_x g_1\right) \;\p_x g_1 \; \frac{\rho}{\bv } &\geq & \mu \int_0^\infty \left(\p_x \left(\frac{\p_x g_1}{\bv}\right)\right)^2 \rho - C \|\rho\|_{W^{2,\infty}}\int_0^\infty (\p_x g_1)^2\\
&&+ \left( \p_x g_{1|x=0}\right)^2 \left( -\frac{s}{2} \rho(0)- \frac{\mu}{2}\rho'(0) - \beta \rho(0)\right).
\end{eqnarray*}

We now choose the function $\rho$ so that the last term is positive at main order. More precisely, we take
\[
 \rho(0)=2,\quad \rho'(0)=-\frac{4s}{\mu},\quad 1\leq \rho(x)\leq 2\quad \forall x\in \R_+.
\]
Then
\begin{equation}\label{eq:diffu_pxg1}
\int_0^\infty \left( \p_x \cA \p_x g_1\right) \;\p_x g_1 \; \frac{\rho}{\bv } \geq  \mu \int_0^\infty \left(\p_x \frac{\p_x g_1}{\bv}\right)^2 + \beta^2\frac{s^3(v_+-1)^2}{\mu^2} - C \|\rho\|_{W^{2,\infty}}\int_0^\infty (\p_x g_1)^2  - C|\beta|^3.
\end{equation}
Note that the term $\|\p_x g_1\|_{L^2}$ is controlled thanks to \eqref{borne-g1-L2H1}.

We now address the terms in the right-hand side of \eqref{eq:dxg1}.
\begin{itemize}
    \item The first term is easily bounded thanks to the $L^2 L^2$ bound \eqref{borne-g1-L2H1} on $\p_x g_1$. We have
\[
\int_0^\infty \p_x^2 g_1 \p_x g_1 \frac{\rho}{\bv}
= -\frac{1}{2}\int_0^\infty(\p_x g_1)^2 \p_x\left(  \frac{\rho}{\bv}\right)-(\p_x g_{1|x=0})^2,
\]
    so that
    \[
    \left| \beta \int_0^\infty \p_x^2 g_1 \p_x g_1 \frac{\rho}{\bv}\right| 
    \leq C |\beta| \|\p_x g_1\|_{L^2(\R_+)}^2 + C |\beta|^3.
    \]

    \item For the second term in the right-hand side, we use the bounds on $g$. We have
    \[
    \left|\mu \beta \int_0^\infty \p_x^2\left(\frac{\p_x \bv}{\bv^2} g\right) \p_x g_1 \frac{\rho}{\bv}\right|\leq C |\beta|\|\p_x g_1\|_{L^2(\R_+)}\|g\|_{H^2(\R_+)}.
    \]
    
    \item Let us now address the nonlinear term, which is quadratic in $g$. Integrating by parts and recalling that $\cA\p_x^2 (\ln(1+g/\bv)- g/\bv)$ vanishes at $x=0$, we have
    \begin{eqnarray*}
  &&  \int_0^\infty \p_x \cA \p_x^2 \left(\ln\left(1+ \frac{g}{\bv}\right) - \frac{g}{\bv}\right) \p_x g_1 \frac{\rho }{\bv}\\
    &=& - \int_0^\infty \cA \p_x^2 \left(\ln\left(1+ \frac{g}{\bv}\right) - \frac{g}{\bv}\right)\left(\p_x \frac{\p_x g_1}{\bv}\right) \rho\\
    &&- \int_0^\infty \cA \p_x^2 \left(\ln\left(1+ \frac{g}{\bv}\right) - \frac{g}{\bv}\right)\left(\frac{\p_x g_1}{\bv}\right) \p_x \rho.
     \end{eqnarray*}
    Furthermore, by definition of the operator $\cA$,
   \begin{eqnarray*}
   \cA \p_x^2 \left(\ln\left(1+ \frac{g}{\bv}\right) - \frac{g}{\bv}\right)
   &=&-s \p_x^2 \left(\ln\left(1+ \frac{g}{\bv}\right) - \frac{g}{\bv}\right)\\
   &&-\mu \p_x \left(\frac{1}{\bv}  \p_x^2 \left(\ln\left(1+ \frac{g}{\bv}\right) - \frac{g}{\bv}\right)\right).
   \end{eqnarray*}
    Remembering \eqref{diff-quadratic}, we obtain
    \begin{eqnarray*}
      && \left|\cA \p_x^2 \left(\ln\left(1+ \frac{g}{\bv}\right) - \frac{g}{\bv}\right)\right|\\
      &\leq & C\left( \left|\p_x^3\frac{g}{\bv}\right| \|g\|_{L^\infty}+  \left|\p_x^2\frac{g}{\bv}\right| \|g\|_{W^{1,\infty}} + \left|\p_x\frac{g}{\bv}\right|^2 +   \left|\p_x\frac{g}{\bv}\right|^3\right).
    \end{eqnarray*}
    As a consequence,
    \begin{align} \label{eq:nonlin-1}
        & \left| \mu  \int_0^\infty \p_x \cA \p_x^2 \left(\ln\left(1+ \frac{g}{\bv}\right) - \frac{g}{\bv}\right) \p_x g_1 \frac{\rho }{\bv}\right| \nonumber \\
        & \leq C\left( \left\| \p_x \frac{\p_x g_1}{\bv} \right\|_{L^2(\R_+)} + \|\p_x g_1\|_{L^2(\R_+)}\right)\left( \|g\|_{L^\infty_{t,x}}\|\partial_x g\|_{H^2(\R_+)} + \|g\|_{W^{1,\infty}_{t,x}}\|g\|_{H^2(\R_+)}\right).
    \end{align}
    In the inequality above, we have used the fact that $\|g\|_{L^\infty([0,\bar T], W^{1,\infty})}$ remains small, so that the cubic term $|\p_x (g/\bv)|^3$ remains smaller than $|\p_x (g/\bv)|^2$.

\end{itemize}

Gathering all the terms and using several times the Cauchy-Schwarz inequality, we obtain, with $c=s^3(v_+-1)^2/\mu^2$
\begin{eqnarray*}
&&\frac{1}{2}\frac{d}{dt}\int_0^\infty (\p_x g_1)^2 \frac{\rho}{\bv} + \frac{\mu}{2}\int_0^\infty\left( \p_x \frac{\p_x g_1}{\bv}\right)^2 + c \beta^2\\
&\leq & C \left(1+ \|\beta\|_{L^\infty(0,\bar T)}\right)\int_0^\infty (\p_x g_1)^2 + C |\beta|^3 + C |\beta|^2\|g\|_{H^2(\R_+)}^2\\
&&+ C \|g\|_{L^\infty([0,\bar T]\times \R_+) }^2 \|g\|_{H^3(\R_+)}^2 + C \|g\|_{L^\infty([0,\bar T], W^{1,\infty}(\R_+)) }^2 \|g\|_{H^2(\R_+)}^2. 
\end{eqnarray*}
Recall that $\|\beta\|_{L^\infty(0,\bar T)}\lesssim \|\beta\|_{H^1(0,\bar T)}\leq 1$ for $\delta$ small enough.
Integrating with respect to time and using \eqref{borne-g1-L2H1}, we get
\begin{multline}\label{est:b}
    \|\p_x g_1\|_{L^\infty([0,\bar T], L^2)}^2 + \left\|\p_x \frac{\p_x g_1}{\bv}\right\|^2_{L^2([0,\bar T]\times \R_+)}+ \|\beta\|_{L^2([0,\bar T])}^2\\
    \leq  C\cE_0 + C(\cE_0 + \|\beta\|_{H^1(0,\bar T)}^2)^2 + C(\cE_0^{1/2} + \|\beta\|_{H^1(0,\bar T)})\|\beta\|_{L^2(0,\bar T)}^2.
    \end{multline}
Note additionally that the last term in the right hand side can be absorbed in the left-hand side provided $\cE_0^{1/2} + \|\beta\|_{H^1(0,\bar T)}$ is small enough.

\bigskip

$\rhd${\it $W^{1,\infty}(L^2)\cap H^1(H^1)$ estimate:}
Let us now consider the time derivative of $g_1$ satisfying
\begin{equation}\label{eq:dt-g1}
\begin{cases}
\partial_t(\partial_t g_1) + \cA \partial_x \partial_t g_1
= \beta \cA \partial_x \partial_t g + \mu \cA \partial_x^2 \partial_t \left(\ln \left(1+ \dfrac{g}{\bv}\right) - \dfrac{g}{\bv} \right)
+ \beta'\cA\partial_x g, \\
(\partial_t g_1)_{|x=0} = 0.
\end{cases}
\end{equation}
Recall that we have by~\eqref{comm-A-dx}:
\[
\cA \partial_x \partial_t g 
= \partial_x \partial_t g_1 + \mu \partial_x \partial_t \left(\dfrac{\partial_x \bv}{\bv^2}g\right).
\]
Multiplying Equation~\eqref{eq:dt-g1} by $\partial_t g_1$ and using once again Lemma \ref{lem:coerc-A}, we get
\begin{align*}
& \dfrac{1}{2} \dfrac{d}{dt} \int_{\R_+} |\partial_t g_1|^2 + \dfrac{\mu}{v_+} \int_{\R_+} |\partial_x \partial_t g_1|^2 \\
& \leq \mu |\beta| \left|\int_{\R_+} \partial_x \partial_t \left(\dfrac{\partial_x \bv}{\bv^2} g \right) \partial_t g_1 \right|
+ \mu \left| \int_{\R_+} \cA \partial_x^2 \partial_t \left(\ln \left(1+ \dfrac{g}{\bv}\right) - \dfrac{g}{\bv} \right) \partial_t g_1\right| \\
& \quad + |\beta'| \left| \int_{\R_+}\cA\partial_x g \partial_t g_1 \right|.
\end{align*}
Integrating by parts and using the Cauchy-Schwarz inequality together with \eqref{est:g-global}, we have for the different terms of the right-hand side:
\begin{itemize}
    \item First
\begin{align*}
\mu |\beta| \left|\int_{\R_+} \partial_x \partial_t \left(\dfrac{\partial_x \bv}{\bv^2} g \right) \partial_t g_1 \right|
& \leq \dfrac{\mu}{4 v_+} \|\partial_t \partial_x g_1\|_{L^2(\R_+)}^2 + C \|\partial_t g\|_{L^\infty([0,\bar T],L^2(\R_+))}^2 |\beta|^2;
\end{align*}
\item Next, for the nonlinear term, we differentiate \eqref{diff-quadratic} with respect to time, which yields
\begin{align*}
& \mu \left| \int_{\R_+} \cA \partial_x^2 \partial_t \left(\ln \left(1+ \dfrac{g}{\bv}\right) - \dfrac{g}{\bv} \right) \partial_t g_1\right| \\
& \leq s\mu \left|\int_{\R_+} \partial_t \partial_x \left(\ln \left(1+ \dfrac{g}{\bv}\right) - \dfrac{g}{\bv} \right)\partial_t \partial_x g_1\right|
+ \mu^2 \left|\int_{\R_+}  \dfrac{1}{\bv} \partial_t \partial^2_x \left(\ln \left(1+ \dfrac{g}{\bv}\right) - \dfrac{g}{\bv} \right)\partial_t \partial_x g_1\right| \\
& \leq \dfrac{\mu}{4v_+} \|\partial_t \partial_x g_1\|_{L^2}^2
+ C \Big(
\|g\|_{W^{1,\infty}_{t,x}}^2 \left\|\partial_t \partial_x \left(\frac{g}{\bv}\right)\right\|_{H^1(\R_+)}^2 + \|\partial_t g\|_{L^\infty_{t,x}}^2 \| g\|_{H^2(\R_+)}^2\Big)\\
&\qquad
+ C\|\partial_t g\|_{L^\infty_{t,x}}^2 \left\| \partial_x \left(\dfrac{g}{\bv}\right)\right\|_{L^4(\R_+)}^2
;
\end{align*}

\item Finally, recalling \eqref{comm-A-dx}
\begin{align*}
&\qquad|\beta'| \left| \int_{\R_+}\cA\partial_x g \partial_t g_1 \right|\\
& \leq |\beta'| \left| \int_{\R_+}\Big[g_1 + \mu \dfrac{\partial_x \bv}{\bv^2} g \Big] \partial_x \partial_t g_1 \right| \\
& \leq \dfrac{\mu}{4v_+} \|\partial_t \partial_x g_1\|_{L^2(\R_+)}^2
+ C |\beta'|^2 \Big( \|g_1\|_{L^\infty([0,\bar T],L^2(\R_+))}^2 + \|g\|_{L^\infty([0,\bar T],L^2(\R_+))}^2\Big).
\end{align*}
\end{itemize}
Eventually, we obtain after time integration of all previous terms and the use of the estimates \eqref{borne-g1-L2H1} on $g_1$, \eqref{est:g-global} on $g$
\begin{align}
\|\partial_t g_1\|_{L^\infty([0, \bar T], L^2(\R_+))}^2
+ \|\partial_t \partial_x g_1\|_{L^2((0,\bar T) \times \R_+)}^2
\leq C \Big[\cE_0 + (\cE_0 +\|\beta\|_{H^1(0,\bar T)}^2 )^2 \Big].
\end{align}

\bigskip
$\rhd${\it $W^{1,\infty}(H^1)\cap H^1(H^2)$ estimate:}
We differentiate Eq~\eqref{eq:dxg1} with respect to time:
\begin{align}\label{eq:dtdx-g1}
 \partial_t (\partial_t \partial_x g_1) + \partial_x \cA \partial_t \partial_x g_1  
& = \beta \partial^2_x \partial_t g_1 + \beta'\partial^2_x g_1
+ \mu \beta' \partial^2_x \left(\dfrac{\partial_x \bv}{\bv^2} g\right) \\
& \quad + \mu \beta \partial^2_x \left(\dfrac{\partial_x \bv}{\bv^2} \partial_t g\right) 
+ \mu \partial_x \cA \partial_x^2 \partial_t \left(\ln \left(1+ \dfrac{g}{\bv}\right) - \dfrac{g}{\bv} \right)\nonumber.
\end{align}
This equation is endowed with the following boundary condition, which is to be understood in a weak sense
\begin{align}\label{BC:px2g1-0}
\mu \partial_x \left(\dfrac{\partial_t\partial_x g_1}{\bv}\right)_{|x=0}
& = -\beta'(\partial_x g_1)_{|x=0} - (\beta+s) (\partial_t \partial_x g_1)_{|x=0} \nonumber \\
& = - \beta'(2\beta +s)s  \dfrac{(v_+-1)}{\mu}.
\end{align}
In other words, for any test function $\psi\in L^2((0,T), H^1(\R_+))$, for almost every $t\in [0,T]$ (recall the regularities~\eqref{reg:g1-1}-\eqref{reg:g1-3} on $g_1$),
\begin{align}
\nonumber&\langle \p_t (\p_t\p_x g_1)(t), \psi(t)\rangle_{H^{-1}, H^1} -\int_0^\infty \cA(\p_t\p_x g_1)\; \p_x \psi +\left(s\p_t\p_x g_{1|x=0} - \beta'(2\beta +s)s  \dfrac{(v_+-1)}{\mu}\right)\psi_{|x=0}\\
\label{weak-dtdxg1}&=-\int_0^\infty \mu \cA \p_x^2 \p_t \left(\ln \left(1+ \dfrac{g}{\bv}\right) - \dfrac{g}{\bv} \right) \p_x \psi(t,x) \\
\nonumber&\quad +\int_0^\infty \left[\beta \p_x^2\p_t g_1 + \beta'\partial^2_x g_1
+ \mu \beta' \partial^2_x \left(\dfrac{\partial_x \bv}{\bv^2} g\right)
+ \mu \beta \partial^2_x \left(\dfrac{\partial_x \bv}{\bv^2} \partial_t g\right) \right] \psi.
\end{align}
In particular, we have used the fact that
\[
\mu \p_x^3\p_t  \left(\ln \left(1+ \dfrac{g}{\bv}\right) - \dfrac{g}{\bv} \right)_{|x=0}=0
\]
in a weak sense. 
The weak formulation~\eqref{weak-dtdxg1} can be obtained by writing the weak formulation of~\eqref{eq:dxg1} associated with $\p_x g_1$, taking a test function of the form $\p_t \psi$ with $\psi \in \mathcal C^1([0,\bar T[, H^1(\R_+))$, and integrating by parts in time.

Similarly to the $L^\infty(H^1)$ estimate we take $\psi=\partial_x(\partial_t g_1) \rho / \bv$ in the weak formulation above, with the same weight function $\rho$ satisfying $\rho(0)=2$, $\rho'(0)=-\frac{4s}{\mu}$, $1\leq \rho(x)\leq 2~\forall x\in \R_+$, and we evaluate each term separately. 
At the boundary $x=0$, we have
\begin{equation}\label{BC:pxg1-0}
(\partial_t \partial_x g_1)_{|x=0} = \dfrac{s(v_+-1)}{\mu} \beta', 
\end{equation}
and the equality holds in $L^2(0,T)$.

\medskip
\noindent
$\bullet$ For the transport-diffusion term we have, replacing $\psi$ by $\partial_x\partial_t g_1 \rho / \bv$ in the left-hand side of \eqref{weak-dtdxg1} and using Lemma \ref{lem:coerc-A},
\begin{align*}
& -\int_0^\infty \cA(\p_t\p_x g_1)\; \p_x \left(\partial_x\partial_t g_1 \frac{\rho}{\bv}\right) +\left(s\p_t\p_x g_{1|x=0} - \beta'(2\beta +s)s  \dfrac{(v_+-1)}{\mu}\right)\left(\partial_x\partial_t g_1 \frac{\rho}{\bv}\right)_{|x=0}\\
 & \geq \mu \int_{\R_+} \left( \partial_x \left(\dfrac{\partial_x \partial_t g_1}{\bv} \right) \right)^2 \rho 
- C \|\rho\|_{W^{2,\infty}}  \int_{\R_+} \big(\partial_x \partial_t g_1\big)^2 \\
 & \quad + \left( \partial_t \p_x g_{1|x=0}\right)^2 \left( \frac{s}{2} \rho(0)- \frac{\mu}{2}\rho'(0) \right)\\
 & \quad - \beta'(2\beta +s)s  \dfrac{(v_+-1)}{\mu} \big(\partial_t \partial_x g_1\big)_{|x=0} \ \rho(0),
\end{align*}
so that, with the  boundary condition \eqref{BC:pxg1-0} and with a computation identical to \eqref{eq:diffu_pxg1}, the right-hand side of the above inequality is equal to
\begin{align*}
 &  \mu \int_{\R_+} \left( \partial_x \left(\dfrac{\partial_x \partial_t g_1}{\bv} \right) \right)^2 \rho 
- C \|\rho\|_{W^{2,\infty}}  \int_{\R_+} \big(\partial_x \partial_t g_1\big)^2 \\
 & \quad + 2\dfrac{s^3 (v_+-1)^2}{\mu^2} |\beta'|^2 
 - 4 \dfrac{s^2(v_+-1)^2}{\mu^2} \beta |\beta'|^2.
\end{align*}
Similarly to the $L^\infty(H^1)$ estimate, the last term will be absorbed in the penultimate one when $\|\beta\|_{L^\infty} \leq \delta$ is small enough, while the second integral is controlled thanks to the previous $H^1(H^1)$-estimate.

In the right-hand side of our estimate, we have to control the following integrals
\begin{align*}
I_1 & = \beta \int_{\R_+} \partial_x^2 \partial_t g_1 \; \partial_t \partial_x g_1 \dfrac{\rho}{\bv} \\
I_2 & = \beta' \int_{\R_+} \partial^2_x g_1 \partial_t \partial_x g_1 \dfrac{\rho}{\bv} \\
I_3 & = \mu \beta' \int_{\R_+} \partial^2_x \left(\dfrac{\partial_x \bv}{\bv^2} g\right) \partial_t \partial_x g_1 \dfrac{\rho}{\bv} \\
I_4 & = \mu \beta \int_{\R_+} \partial^2_x \left(\dfrac{\partial_x \bv}{\bv^2} \partial_t g\right) \partial_t \partial_x g_1 \dfrac{\rho}{\bv} \\
I_5 & = \mu \int_{\R_+}  \cA \partial^2_x \partial_t \left(\ln \left(1 + \frac{g}{\bv}\right) - \frac{g}{\bv}\right)\p_x\left(\partial_t \partial_x g_1 \dfrac{\rho}{\bv}\right).
\end{align*}

\noindent
$\bullet$ The first integral is controlled via an integration by parts:
\begin{align*}
|I_1|
&  \leq \dfrac{|\beta|}{2} \int_{\R_+} (\partial_t \partial_x g_1)^2 \partial_x \left(\frac{\rho}{\bv}\right) 
+ \beta \left(\partial_t \partial_x g_1\right)^2_{|x=0} \\
& \leq C|\beta| \|\partial_t \partial_x g_1\|_{L^2(\R_+)}^2 
+ C |\beta'|^2 \beta.
\end{align*}

\medskip
\noindent
$\bullet$ For $I_2, I_3, I_4$, the Cauchy-Schwarz inequality yields
\begin{align*}
|I_2| + |I_3| + |I_4|
& \leq  \|\partial_x \partial_t g_1\|_{L^2(\R_+)}^2 
+ C|\beta'|^2 \Big(\|\partial^2_x g_1\|_{L^\infty([0,\bar T],L^2(\R_+))}^2+ \|g\|_{L^\infty([0,\bar T],H^2(\R_+))}^2 \Big) \\
& \quad
+ \|\beta\|_{L^\infty(0,\bar T)}^2 \|\partial_t g(t)\|_{H^2(\R_+)}^2.
\end{align*}

\medskip
\noindent
$\bullet$ For the nonlinear term $I_5$, we have similarly to~\eqref{eq:nonlin-1}:
\begin{align*} 
    & \left| \mu  \int_0^\infty  \cA \partial^2_x \partial_t \left(\ln \left(1 + \frac{g}{\bv}\right) - \frac{g}{\bv}\right)\p_x\left(\partial_t \partial_x g_1 \dfrac{\rho}{\bv}\right)\right| \nonumber \\
    & \leq C \left(\left\| \p_x \frac{\p_x \p_t g_1}{\bv} \right\|_{L^2(\R_+)} + \|\p_t \p_x g_1\|_{L^2(\R_+)}\right) \\
    & \quad \times \Bigg(  \|g\|_{L^\infty_{t,x}}\|\p_t \partial_x g\|_{H^2(\R_+)} + \|\p_t g\|_{W^{1,\infty}_{t,x}}\|g\|_{H^2(\R_+)} +  \|g\|_{W^{1,\infty}_{t,x}}(\|\p_t g\|_{H^2(\R_+)} + \|g\|_{H^3(\R_+)})
    \Bigg). \nonumber 
\end{align*}
To control $\|\partial_t \partial_x g\|_{H^2(\R_+)}$, we write in the case $\tw^0\equiv u_+$:
\[
\mu \p_t \p_x^3 \ln\left(1 + \frac{g}{\bv}\right)=\mu \partial_t \partial^3_x \ln \frac{v}{\bv} 
= \partial_t \partial^2_x (u -\bu),
\]
so that $\|\partial_t \partial_x g\|_{H^2(\R_+)} \leq C( \|\partial_t h\|_{H^2(\R_+)} + \|\p_t g\|_{H^2(\R_+)})$, where we recall that $h = u -\bu$.

\medskip
Gathering all terms and integrating in time, we get, using \eqref{est:g-global}, \eqref{est:h-global} and \eqref{borne-g1-L2H1}
\begin{multline}
    \| \partial_t \p_x g_1\|_{L^\infty([0,\bar T], L^2(\R_+))}^2 + \left\|\p_x \frac{\p_x \partial_t g_1}{\bv}\right\|^2_{L^2([0,\bar T]\times \R_+)}+ \|\beta'\|_{L^2([0,\bar T])}^2\\
    \leq  C\cE_0 + C(\cE_0 + \|\beta\|_{H^1(0,\bar T)}^2)^2 + C(\cE_0^{1/2} + \|\beta\|_{H^1(0,\bar T)})\|\beta'\|_{L^2(0,\bar T)}^2.
    \end{multline}
Observing that the last term in the right-hand side can be absorbed in the left-hand side provided $\cE_0^{1/2} + \|\beta\|_{H^1(0,\bar T)}$ is small enough,
we obtain eventually
\begin{align}\label{est:b'}
&\|\partial_t \partial_x g_1\|_{L^\infty([0,\bar T], L^2)}^2 + \left\|\partial_x \left(\dfrac{\partial_t \partial_x g_1}{\bv} \right) \right\|_{L^2([0, \bar T] \times \R_+)}^2 
+ \|\beta'\|_{L^2([0,\bar T])}^2 
\leq C \Bigg[\cE_0 + (\cE_0 +\|\beta\|_{H^1(0,\bar T)}^2 )^2 \Bigg]. 
\end{align}

$\rhd${\it Conclusion:} Combining ~\eqref{est:b} and~\eqref{est:b'}, we deduce that for $\cE_0 \leq c\delta^2$ with $c > 0$ small enough (depending only on the parameters $\mu, v_+,u_\pm$), 
\[
\|\beta\|_{H^1([0, \bar T])}^2 \leq C \left( \cE_0 + (\cE_0 + \|\beta\|_{H^1(0,\bar T)}^2)^2\right) \leq C c_0 \delta^2 + C \delta^4 \leq \frac{\delta^2}{4}.
\]
This completes the proof of Proposition \ref{prop:bootstrap} in the case $\tw^0\equiv u_+$. Note that in this case $\p_x^k \tw^0=0$ for $k=1,2,3$. Hence the additional conditions on $\tw^0$ are automatically satisfied.

\subsection{Case when \texorpdfstring{$\tw^0$}{w0} is arbitrary}
\label{ssec:w0-qcq}

In the case when $\tw^0$ is not constant, the estimates of the previous section involve new contributions coming both from  the equations, which contain additional terms linked to $\partial_x \tw^0$, and  from the boundary $x=0$, since $(\partial_x g)_{|x=0}$ (and thus $(g_1)_{|x=0}$) is not equal to $0$ anymore. More precisely,
\be\label{eq:g1-w0}
\ba
\p_t g_1 + \cA \p_x g_1 
= \beta \cA \p_x g + \mu \cA \p_x^2\left(\ln\left(1+ \frac{g}{\bv}\right) - \frac{g}{\bv}\right) + \cA \partial_x \tw^0(x+ \tx(t)),\\
g_{1|t=0}=\cA\left(\tv^0-\bv\right),\ea
\ee
with at the boundary $x=0$
\begin{equation}\label{eq:g1-x0}
(g_1)_{|x=0} 
= -s\underset{=0}{\underbrace{g_{|x=0}}} - \mu \left(\partial_x\left(\dfrac{g}{\bv}\right)\right)_{|x=0} 
= - \mu (\partial_x (v -\bv))_{|x=0}
= \tw^0(\tx) -u_+
\end{equation}
since $u(t,x) - \mu \p_x \ln v(t,x)=\tw^0(x+\tx(t))$ for all $t,x$.

Nevertheless, we will show that these new contributions are all controlled provided $\|\tw^0 - u_+\|_{L^\infty(\R_+)}$ and $\|(1+\sqrt{x})\partial^k_x \tw^0\|_{L^2(\R_+)}$, $k=1,2,3$, are small.
We recall that using the definition of $g_1$ together with the estimates~\eqref{est:h-global} and with the equation on $u$, we have the regularities~\eqref{reg:g1-1}-\eqref{reg:g1-3}:
\[
g_1\in L^2((0,\bar T), H^3(\R_+)),\quad \p_t g_1\in L^2((0,\bar T), H^2(\R_+)), \quad \p_t^2 g_1\in L^2((0,\bar T)\times \R_+).
\]
As a consequence,
\[
\p_x^k g_{1|x=0},\ \p_t \p_x^l g_{1|x=0}\in L^2(0,\bar T)
\]
for $0\leq k\leq 3$ and $0\leq l \leq 1$.

\paragraph{Estimates on the traces}
Let us detail the different traces associated to $g_1$ that will be involved in the estimates and how they are controlled in $L^2([0, \bar T])$

\begin{lem}[Trace estimates]\label{lem:traces}
Assume that $\|\tw^0-u_+\|_{W^{1,\infty}(\R_+)}\leq 1$.
Then

\begin{itemize}
    \item Estimates on $g_{1|x=0}$:
    \be\label{est:bc-g1}
    \|(g_1)_{|x=0}\|_{L^2(0,\bar T)}^2 + \|(\partial_t g_1)_{|x=0}\|_{L^2(0,\bar T)}^2 
 \leq C \|(1+\sqrt{x})\partial_x \tw^0\|_{L^2(\R_+)}^2
 \leq C \cE_0;
 \ee
    
    \item Estimates on $\p_x g_{1|x=0}$:
    \be\label{g1-R1}
   \p_x g_{1|x=0}= \beta \dfrac{s(v_+ - 1 )}{\mu} + R_1,
    \ee
    where
     \[
\|R_1\|_{H^1(0,\bar T)}\leq  C\left(\|\beta\|_{H^1(0,\bar T)}\|\tw^0- u_+\|_{W^{1,\infty}(\R_+)} + \|\tw^0-u_+\|_{W^{1,4}(\R_+)}^2 + \|\p_x \tw^0\|_{H^1(\R_+)}\right).
\]

\item Estimates on the traces of the nonlinear term:
\[
 \p_x^k \left(\ln \left(1+ \frac{g}{\bv}\right) - \frac{g}{\bv} \right)_{|x=0}=T_k, \quad k=1,2,3,
\]
with
\be\label{est:traces-NL}
\ba
T_1 = \left[\partial_x \left(\dfrac{g}{\bv} \right)\dfrac{g}{\bv}\right]_{|x=0} = 0, \\
\|T_2\|_{H^1(0,\bar T)}\leq C \|\tw^0-u_+\|_{W^{1,4}(\R_+)}^2, \\
\|T_3\|_{H^1(0,\bar T)}\leq C \left(\|\beta\|_{H^1(0, \bar T)}\|\tw^0-u_+\|_{W^{1,\infty}(\R_+)} + \|\tw^0-u_+\|_{W^{2,4}(\R_+)}^2\right).
\ea
\ee

 \item Estimates on $\p_x^2 g_{1|x=0}$:
 \begin{equation}\label{bc:px2g1}
\mu \left[\partial_x \left(\dfrac{\partial_x g_1}{\bv}\right) \right]_{|x=0}
= -(s+\beta)(\partial_x g_1)_{|x=0} - R_2 ,
\end{equation}
where
\begin{equation}\label{est:R2}
\|R_2\|_{H^1(0,\bar T)}\leq C\left( \|\beta\|_{H^1(0, \bar T)}\|\tw^0-u_+\|_{W^{1,\infty}(\R_+)} + \| \tw^0-u_+\|_{W^{2,4}(\R_+)}^2 + \|\p_x \tw^0\|_{H^2(\R_+)} \right).
 \end{equation}

\end{itemize}

\end{lem}

\bigskip

\begin{proof}
Throughout the proof, we recall that for all $t\in [0, \bar T]$, $\tx'$ satisfies \eqref{bound:tx'}, so that
\[
\frac{st}{2}\leq \tx(t)\leq \frac{3st}{2}.
\]

$\bullet$ First, the estimate on $g_1$ follows directly from \eqref{eq:g1-x0}, writing
\[
(\tw^0(\tx)-u_+)^2=-2\int_{\tx}^\infty (\tw^0(x)-u_+) \p_x \tw^0(x)\:dx 
= - 2\int_0^\infty \mathbf 1_{\{x>\tx\}}(\tw^0(x)-u_+)\p_x \tw^0(x)\:dx
\]
and using the same kind of trick as in the proof of Lemma \ref{lem:x+y}.

For the time derivative, we write
\[
\partial_t g_{1|x=0}=\tx' \p_x \tw^0(\tx)
\]
and integrate in time.

$\bullet$ Estimates on $\p_x g_{1|x=0}$: First, we have by~\eqref{eq:g1-x0}
\begin{align}\label{eq:d2ln-x0}
\partial^2_x \left(\ln \left(1+ \frac{g}{\bv}\right) -\frac{g}{\bv} \right)_{|x=0}
& = - \left(\partial^2_x\left(\frac{g}{\bv}\right)\dfrac{g/\bv}{1+ g/\bv} \right)_{|x=0}
+ \left(\partial_x \left(\frac{g}{\bv}\right)\right)^2_{|x=0} \left(\dfrac{1}{(1+ g/\bv)^2}\right)_{|x=0} \nonumber \\
& = \dfrac{1}{\mu^2} (\tw^0(\tx)-u_+)^2=:T_2
\end{align}
Note that the  estimate of \eqref{est:traces-NL} with $k=2$ follows.

Next, taking the trace of the equation satisfied by $g$~\eqref{eq:g0} (with the additional $\partial_x \tw^0$ in the right-hand side), we infer that
\begin{align}\label{eq:dxg1-0}
(\partial_x g_1)_{|x=0}
& = \beta (\partial_x v)_{|x=0} 
+ \mu \partial^2_x \left(\ln \left(1+ \frac{g}{\bv}\right) -\frac{g}{\bv} \right)_{|x=0} + \partial_x \tw^0(\tx)  \nonumber \\
& = \dfrac{u_- - \tw^0(\tx)}{\mu} \beta +\mu T_2+ \partial_x \tw^0(\tx)\\
& = \beta \dfrac{s(v_+ - 1 )}{\mu} + R_1 \nonumber,
\end{align}
with
\[
R_1 := \beta \dfrac{u_+ - \tw^0(\tx)}{\mu}  + \mu T_2 + \partial_x \tw^0(\tx).
\]
We then bound $R_1$ in $H^1(0,T)$. We have
\begin{align*}
\|R_1\|_{H^1(0,\bar T)}
& \leq  C\Big(\|\beta\|_{H^1(0, \bar T)}\|\tw^0- u_+\|_{L^{\infty}(\R_+)} + \|\beta\|_{L^2(0, \bar T)} \|\p_x\tw^0\|_{L^\infty(\R_+)} \\
& \qquad + \|T_2\|_{H^1(0, \bar T)} + \|\p_x \tw^0\|_{H^1(\R_+)}\Big).
\end{align*}

$\bullet$ Estimate of $T_3$: we have
\begin{align*}
\partial^3_x \left(\ln \left(1+ \frac{g}{\bv}\right) - \frac{g}{\bv} \right) 
& = - \partial^3_x \left(\frac{g}{\bv}\right) \dfrac{\frac{g}{\bv}}{1+ \frac{g}{\bv}}
- \dfrac{3}{(1+\frac{g}{\bv})^2} \partial_x \left(\frac{g}{\bv}\right) \partial^2_x \left(\frac{g}{\bv}\right)\\
& \quad +  \dfrac{2}{(1+ \frac{g}{\bv})^3}\left(\partial_x \left(\frac{g}{\bv}\right)\right)^3,
\end{align*}
and
\begin{align*}
\dfrac{\mu}{1+\frac{g}{\bv}} \partial^2_x \left(\frac{g}{\bv}\right) - \dfrac{\mu}{(1+\frac{g}{\bv})^2} \left(\partial_x\left(\frac{g}{\bv}\right)\right)^2 
& = \mu \partial_x \left( \dfrac{\partial_x\left(\frac{g}{\bv}\right) }{1+\frac{g}{\bv}} \right)\\
& = \mu \partial^2_x \ln\left(1 + \frac{g}{\bv}\right) \\
& = \partial_x (u-\bu) -\partial_x \tw^0(x +\tx).
\end{align*}
By~\eqref{y'-s}, we have
\begin{align}
(\partial_x(u -\bu))_{|x=0}
& = \dfrac{s}{\mu} (\tw^0(\tx) - u_+) - \frac{\beta}{\mu} (u_- -\tw^0(\tx)) \nonumber \\
& = - \beta s (v_+ -1) + \dfrac{\beta + s}{\mu} (\tw^0(\tx) -u_+), 
\end{align}
and therefore
\begin{align*}
\left[\partial^2_x \left(\frac{g}{\bv}\right)\right]_{|x=0}
& = \left(\partial_x\left(\frac{g}{\bv}\right)\right)^2_{|x=0} + \dfrac{1}{\mu}\Big[(\partial_x(u -\bu))_{|x=0} - \partial_x \tw^0(\tx)\Big] \\
& = \dfrac{(\tw^0(\tx) - u_+)^2}{\mu^2} 
-\beta \dfrac{s (v_+-1)}{\mu} + \dfrac{\beta +s}{\mu^2} (\tw^0(\tx) - u_+)- \dfrac{\partial_x\tw^0(\tx)}{\mu}.
\end{align*}
It follows that
\begin{align*}
T_3
& = - 3 \left[\partial^2_x \left(\frac{g}{\bv}\right)\right]_{|x=0} \left[\partial_x \left(\frac{g}{\bv}\right)\right]_{|x=0} 
+ 2\left[\partial_x \left(\frac{g}{\bv}\right)\right]^3_{|x=0} \\
& = -3\beta\dfrac{s(v_+-1)}{\mu^2} (\tw^0(\tx) - u_+) + 3 (\beta +s) \dfrac{(\tw^0(\tx)-u_+)^2}{\mu^3} \\
& \quad - \dfrac{3}{\mu^2} \partial_x \tw^0(\tx) (\tw^0(\tx)-u_+) + \dfrac{(\tw^0(\tx) - u_+)^3}{\mu^3}.
\end{align*}
Note that $(\tw^0(\tx)-u_+)^3$ can be considered as negligible compared to $(\tw^0(\tx) - u_+)^2$ since we assume that $\|\tw^0 - u_+\|_{W^{1, \infty}(\R_+)} \leq 1$. 
Hence, taking the $H^1$ norm of the right-hand side, we obtain
\[
\|T_3\|_{H^1(0,\bar T)}\leq C \left(\|\beta\|_{H^1(0, \bar T)}\|\tw^0-u_+\|_{W^{1,\infty}(\R_+)} + \|\tw^0-u_+\|_{W^{2,4}(\R_+)}^2\right).
\]

$\bullet$ Traces of $\partial^2_x g_1$:

We have
\[
\mu \left(\partial_x\dfrac{\partial_x g_1}{\bv}\right)_{|x=0}
= -s (\partial_x g_1)_{|x=0} - (\cA \partial_x g_1)_{|x=0}
\]
where, thanks to~\eqref{eq:g1-w0}-\eqref{eq:g1-x0} and \eqref{comm-A-dx}:
\begin{align*}
(\cA \partial_x g_1)_{|x=0}
& =
 \beta (\cA \partial_x g)_{|x=0} + \mu \left[\cA \partial^2_x \left(\ln \left(1+ \frac{g}{\bv}\right) - \frac{g}{\bv} \right) \right]_{|x=0} \\
& \quad  + (\cA \partial_x \tw^0(x+\tx))_{|x=0} - (\partial_t g_1)_{|x=0} \\
& = \beta (\partial_x g_1)_{|x=0} + \mu \beta \left[\partial_x \left(\dfrac{\partial_x \bv}{\bv^2} g\right)\right]_{|x=0} + \mu \left[\cA \partial^2_x \left(\ln \left(1+ \frac{g}{\bv}\right) - \frac{g}{\bv} \right) \right]_{|x=0} \\
& \quad  + s (v_+-2) \partial_x \tw^0(\tx) - \mu \partial^2_x \tw^0(\tx) - \tx'(t) \partial_x \tw^0(\tx).
\end{align*}
Hence, we deduce that
\begin{equation}\label{bc:p2xg1-w0}
\mu \left(\partial_x\dfrac{\partial_x g_1}{\bv}\right)_{|x=0}
= -(s+\beta) (\partial_x g_1)_{|x=0} - R_2, 
\end{equation}
with, compared to~\eqref{bc:p2xg1-0}, a remainder:
\begin{align*}
R_2 
& = \mu \beta \left[\partial_x \left(\dfrac{\partial_x \bv}{\bv^2} g\right)\right]_{|x=0} + \mu \left[\cA \partial^2_x \left(\ln \left(1+ \frac{g}{\bv}\right) - \frac{g}{\bv} \right) \right]_{|x=0} \\
& \quad + s (v_+-2) \partial_x \tw^0(\tx) - \mu \partial^2_x \tw^0(\tx) - \tx'(t) \partial_x \tw^0(\tx).
\end{align*}

Let us now estimate $R_2$. We have by definition of $\cA$
\begin{align*}
& \mu \left[\cA \partial^2_x \left(\ln \left(1+ \frac{g}{\bv}\right) - \frac{g}{\bv} \right) \right]_{|x=0} \\
& = -s \mu \left[\partial^2_x \left(\ln \left(1+ \frac{g}{\bv}\right) - \frac{g}{\bv} \right) \right]_{|x=0}
- \mu^2 \left[\partial_x \left( \frac{1}{\bv} \partial^2_x \left(\ln \left(1+ \frac{g}{\bv}\right) - \frac{g}{\bv} \right) \right)\right]_{|x=0}\\
& = \big(\mu(\partial_x \bv)_{|0}-s\big) \mu \left[\partial^2_x \left(\ln \left(1+ \frac{g}{\bv}\right) - \frac{g}{\bv} \right) \right]_{|x=0}
- \mu^2 \left[\partial^3_x \left(\ln \left(1+ \frac{g}{\bv}\right) - \frac{g}{\bv}  \right)\right]_{|x=0}\\
& = (v_+-2) s \mu T_2
- \mu^2 T_3.
\end{align*}
Replacing in the expression of $R_2$ and using \eqref{eq:g1-x0}, we obtain
\begin{eqnarray*}
R_2&=&\frac{s(v_+-1)}{\mu} \beta (\tw^0(\tx)-u_+) + (v_+-2) s \mu T_2
- \mu^2 T_3\\
&&+ s (v_+-2) \partial_x \tw^0(\tx) - \mu \partial^2_x \tw^0(\tx) - \tx'(t) \partial_x \tw^0(\tx).
\end{eqnarray*}
Thus
\begin{eqnarray*}
\|R_2\|_{H^1(0,\bar T)}
&\leq & C \left(\|\beta\|_{H^1}\|\tw^0-u_+\|_{W^{1,\infty}(\R_+)} + \|T_2\|_{H^1(0, \bar T)} + \|T_3\|_{H^1(0, \bar T)} + \|\p_x \tw^0\|_{H^2(\R_+)}\right)\\
&\leq & C\left( \|\beta\|_{H^1(0, \bar T)}\|\tw^0-u_+\|_{W^{1,\infty}(\R_+)} + \| \tw^0-u_+\|_{W^{2,4}(\R_+)}^2 + \|\p_x \tw^0\|_{H^2(\R_+)} \right) .
\end{eqnarray*}

\end{proof}

\bigskip

\paragraph{Estimates on $g_1$}
We now turn towards the energy estimates for $g_1$.
As announced before, we focus on the additional terms coming either from the boundary terms or from the non-zero source term involving $\p_x \tw^0$.

$\rhd${\it $L^{\infty}(L^2)\cap L^2(H^1)$ estimate:}
Multiplying~\eqref{eq:g1-w0} by $g_1$ and proceeding as in the previous subsection, we have now
\begin{align}\label{estim:g1-bis-0}
& \frac{1}{2}\frac{d}{dt}\int_0^\infty g_1^2 + \frac{\mu}{v_+}\int_0^\infty (\p_x g_1)^2 \nonumber \\
& \leq   C(\cE_0 + \|\beta\|_{H^1(0, \bar T)}^2)\left(|\beta|^2 + \left\| g(t) \right\|_{H^2(\R_+)}^2\right)\\
& \quad + \left|\int_{\R_+} \cA \partial_x \tw^0(x+ \tx) g_1 \right|
+ \dfrac{s+ \beta}{2} (g_1)_{|x=0}^2 + \mu \big|(g_1)_{|x=0}(\partial_x g_1)_{|x=0} \big| \nonumber \\
& \quad +\mu^2 \left|\left(\partial^2_x\left(\ln\left(1+ \frac{g}{\bv}\right) -\frac{g}{\bv} \right)\right)_{|x=0} (g_1)_{|x=0} \right|
+ s\mu \underset{=0}{\underbrace{\left| \partial_x \left(\ln \left(1+ \frac{g}{\bv}\right) -\frac{g}{\bv} \right)_{|x=0}\right|}} |(\partial_x g_1)_{|x=0}| \nonumber
\end{align}
where the two last lines contain the additional contributions compared to the case $\tw^0 \equiv u_+$ (cf.~\eqref{eq:energ-g1-0}). 

We treat the integral involving $\partial_x \tw^0$ as follows:
\begin{align*}
\int_{\R_+} \cA \partial_x \tw^0(x+ \tx) g_1
& = s \int_{\R_+} (\tw^0(x+\tx) - u_+) \partial_x g_1 
+ \mu \int_{\R_+} \dfrac{\partial_x \tw^0}{\bv} \partial_x g_1 \\
& \quad + s(\tw^0(\tx) - u_+) (g_1)_{|x=0} + \mu \partial_x \tw^0(\tx) (g_1)_{|x=0} 
\end{align*}
so that
\begin{align*} \left|\int_{\R_+} \cA \partial_x \tw^0(x+ \tx) g_1\right|
& \leq \dfrac{\mu}{8v_+} \int_{\R_+} (\partial_x g_1)^2 
+  s |\tw^0(\tx) - u_+| |(g_1)_{|x=0}| + \mu |\partial_x \tw^0(\tx)| |(g_1)_{|x=0}| \\
& \quad + C \Big[\|\tw^0(x+\tx) - u_+\|_{L^2(\R_+)}^2 + \|\partial_x \tw^0(\cdot+\tx)\|_{L^2(\R_+)}^2 \Big] ,
\end{align*}
where the first integral can be absorbed in the left-hand side of~\eqref{estim:g1-bis-0}.

We then come back to~\eqref{estim:g1-bis-0} and integrate in time. The traces of $g_1$, $\p_x g_1$ and of the nonlinear term are estimated thanks to Lemma \ref{lem:traces}.
Note in particular that by the Cauchy-Schwarz inequality and the assumption $\|\tw^0 - u_+\|_{W^{1,\infty}(\R_+)} \leq 1$
\begin{eqnarray*}
\int_0^T \left| (g_1)_{|x=0}(\partial_x g_1)_{|x=0}\right|
&\leq& C\|(1+\sqrt{x})\p_x \tw^0\|_{L^2(\R_+)}\left(\|\beta\|_{H^1(0, \bar T)} + \cE_0^{1/2}\right) \\
&\leq& C \|\beta\|_{H^1(0, \bar T)}^4 + C\|(1+\sqrt{x})\p_x \tw^0\|_{L^2(\R_+)}^{4/3} + C \cE_0.
\end{eqnarray*}

Gathering all terms, we deduce that
\begin{align}\label{borne-g1-L2H1-bis}
& \|g_1\|_{L^\infty([0,\bar T], L^2(\R_+)}^2 + \frac{\mu}{v_+}\|\p_x g_1\|_{L^2([0,\bar T], L^2(\R_+)}^2 \nonumber \\
& \leq C \left(\cE_0+ \|(1+\sqrt{x})\p_x \tw^0\|_{L^2(\R_+)}^{4/3} +  (\cE_0 + \|\beta\|_{H^1(0, \bar T)}^2)^2\right).
\end{align}

\bigskip

$\rhd${\it $L^{\infty}(H^1)\cap L^2(H^2)$ estimate:}
Differentiating with respect to $x$ Eq.~\eqref{eq:g1-w0}, we obtain~\eqref{eq:dxg1} with the additional contribution $\partial_x \cA \partial_x \tw^0$ in the right-hand side. 
Following the same steps as in the constant case $\tw^0 \equiv u_+$, we use Lemma~\ref{lem:coerc-A} to handle the diffusion term:
\begin{eqnarray*}
\int_0^\infty \left( \p_x \cA \p_x g_1\right) \;\p_x g_1 \; \frac{\rho}{\bv } &\geq & \mu \int_0^\infty \left(\p_x \left(\frac{\p_x g_1}{\bv}\right)\right)^2 \rho - C \|\rho\|_{W^{2,\infty}}\int_0^\infty (\p_x g_1)^2\\
&&+ \left( \p_x g_{1|x=0}\right)^2 \left( \frac{s}{2} \rho(0)- \frac{\mu}{2}\rho'(0)\right) 
+ \mu \left(\partial_x\dfrac{\partial_x g_1}{\bv}\right)_{|0} (\partial_x g_1)_{|x=0} \rho(0)
\end{eqnarray*}
Choosing once again $\rho(0)=2$, 
$\rho'(0)=-4s/\mu$ and using \eqref{g1-R1}, we have
\begin{eqnarray*}
&&\left( \p_x g_{1|x=0}\right)^2 \left( \frac{s}{2} \rho(0)- \frac{\mu}{2}\rho'(0)\right) 
+ \mu \left(\partial_x\dfrac{\partial_x g_1}{\bv}\right)_{|0} (\partial_x g_1)_{|x=0} \rho(0)\\
&\geq &\left(\frac{s(v_+-1)}{\mu}\right)^2(s+\beta) \beta^2 -C\left(R_1^2 + |R_1|\; |\beta| + |R_2|\; |\beta| + |R_2|\; |R_1|\right).
\end{eqnarray*}
We then control the time integral of the second term in the right-hand side thanks to Lemma \ref{lem:traces}.

Following the same steps as in the constant case $\tw^0 \equiv u_+$, we can check that we have the following other additional contributions:
\begin{itemize}
    \item the additional integral involving $\partial_x \cA \partial_x \tw^0$ is treated by integration by parts with boundary terms of the type $\partial^k_x \tw^0_{|x=0} (\partial_x g_1)_{|x=0}$, $k=1,2$. The integral is easily controlled by use of the Cauchy-Schwarz inequality;
    \item from the nonlinear term, we get the following boundary term:
    \[
    \left[\cA \partial^2_x \left(\ln \left(1+ \frac{g}{\bv}\right) - \frac{g}{\bv} \right) \right]_{|x=0} (\partial_x g_1)_{|x=0},
    \]
\end{itemize}
which we also treat using Lemma \ref{lem:traces}, noticing that
\begin{eqnarray}
\left[\cA \partial^2_x \left(\ln \left(1+ \frac{g}{\bv}\right) - \frac{g}{\bv} \right) \right]_{|x=0}&=&-s T_2 -\mu\p_x\left( \frac{1}{\bv}
\partial^2_x \left(\ln \left(1+ \frac{g}{\bv}\right) - \frac{g}{\bv} \right)\right)_{|x=0}\nonumber\\
&=&s(v_+-2)T_2-\mu T_3 .\label{trace-A-dx2-NL}
\end{eqnarray}

We obtain
\begin{eqnarray*}
 &&  \|\p_x g_1\|_{L^\infty([0,\bar T], L^2)}^2 + \left\|\p_x \frac{\p_x g_1}{\bv}\right\|^2_{L^2([0,\bar T]\times \R_+)}+ \|\beta\|_{L^2([0,\bar T])}^2\\
  &  \leq &  C\left(\cE_0 + \left\| (1+\sqrt{x}) \p_x \tw^0\right\|_{L^2(\R_+)}^{4/3} + \left\|  \p_x \tw^0\right\|_{H^2(\R_+)}^{4/3}+(\cE_0 + \|\beta\|_{H^1(0,\bar T)}^2)^2 \right).
    \end{eqnarray*}

\bigskip

$\rhd${\it $W^{1,\infty}(L^2)\cap H^1(H^1)$ estimate:}
Differentiating with respect to time Equation~\eqref{eq:g1-w0}, we get Equation~\eqref{eq:dt-g1} with an additional term in the right-hand side, namely $\tx'(t)\cA \partial^2_x \tw^0(x+\tx(t))$.
Multiplying the equation by $\partial_t g_1$ and integrating in space, we have, similarly to \eqref{estim:g1-bis-0}
\begin{align}
& \frac{1}{2}\frac{d}{dt}\int_0^\infty (\p_t g_1)^2 + \frac{\mu}{v_+}\int_0^\infty (\p_t \p_x g_1)^2 \nonumber \\
& \leq   C(\cE_0 + \|\beta\|_{H^1(0, \bar T)}^2)\left(|\beta'|^2 + |\beta|^2+ \left\| g \right\|_{H^2(\R_+)}^2 + \| \p_t g\|_{H^2(\R_+)}^2\right)\label{estim:dt-g1-bis-0}\\
& \quad + \left|\int_{\R_+}  \tx'(t)\int_{\R_+}\cA \partial^2_x \tw^0(x+\tx(t)) \partial_t g_1 \right|\nonumber\\
&\quad + \dfrac{s+ \beta}{2} (\p_t g_1)_{|x=0}^2 + \mu \big|(\p_t g_1)_{|x=0}(\p_t \partial_x g_1)_{|x=0} \big|  +\mu^2 \left|\p_t T_2 (\p_t g_1)_{|x=0} \right|
 \nonumber\\
 &\quad + |\beta'| \; |g_{1|x=0}| \; |\p_t g_{1|x=0}|\nonumber.
\end{align}
We first consider the integral term involving $\cA \partial^2_x \tw^0(x+\tx(t))$:
\begin{align*}
& \tx'(t)\int_{\R_+}\cA \partial^2_x \tw^0(x+\tx(t)) \partial_t g_1 \\ 
& = - \tx'(t) \int_{\R_+} \left( \cA \partial_x \tw^0(x+ \tx(t)) + \mu \left(\frac{\partial_x \bv}{\bv^2}\partial_x \tw^0(x+\tx(t))\right) \right) \partial_x \partial_t g_1   \\
& \quad -\tx'(t) \left[ \cA \partial_x \tw^0(x+ \tx(t)) + \mu \left(\frac{\partial_x \bv}{\bv^2}\partial_x \tw^0(x+\tx(t))\right)\right]_{|x=0}(\partial_t g_1)_{|x=0}.
\end{align*}
The integral is then controlled by a Cauchy-Schwarz inequality, using the trace estimates from Lemma \ref{lem:traces} together with Lemma \ref{lem:x+y}
\begin{align*}
& \int_0^{\bar T}\int_{\R_+} \left|\tx'(t)  \left( \cA \partial_x \tw^0(x+ \tx(t)) + \mu \left(\frac{\partial_x \bv}{\bv^2}\partial_x \tw^0(x+\tx(t))\right) \right) \partial_x \partial_t g_1 \right| \\
& \leq \frac{\mu}{4v_+} \|\partial_t \partial_x g_1\|_{L^2((0,\bar T)\times \R_+)}^2 
+ C \Big(\|\sqrt{x} \partial_x \tw^0\|_{L^2(\R_+)}^2 + \|\sqrt{x} \partial^2_x \tw^0\|_{L^2(\R_+)}^2 \Big)\\
&\quad + C\|\p_x \tw^0\|_{H^1} \|(1+\sqrt{x})\p_x \tw^0\|_{L^2}.
\end{align*}
To conclude as in the constant case, we have to control the different boundary terms in \eqref{estim:dt-g1-bis-0}. Using Lemma \ref{lem:traces},
their $L^2(0,\bar T)$ norm is bounded by
\begin{align*}
    &C\Big(\| g_{1|x=0}\|_{H^1(0,\bar T)}^2 + \|\p_t T_2\|_{L^2(0,\bar T)}^2 + \|\p_t g_{1|x=0}\|_{L^2(0,\bar T)} \|\p_t \p_x g_{1|x=0}\|_{L^2(0,\bar T)} \\
    &\qquad+ \|\beta'\|_{L^2(0,\bar T)} \|g_{1|x=0}\|_{H^1(0,\bar T)}^2\Big)\\
    \leq& C\left( \|(1+\sqrt{x}) \p_x \tw^0\|_{L^2}^2 + \|\tw^0-u_+\|_{W^{1,4}}^4\right)\\
    &\quad +C \|(1+\sqrt{x}) \p_x \tw^0\|_{L^2} \left(\|\beta\|_{H^1(0,\bar T)} + \|\tw^0-u_+\|_{W^{1,4}}^2 + \|\p_x \tw^0\|_{L^2}\right)\\
    &\leq C \left(\|\beta\|_{H^1(0,\bar T)}^4 + \|(1+\sqrt{x}) \p_x \tw^0\|_{L^2}^{4/3} + \cE_0\right).
\end{align*}
Note that we used here the smallness of $\cE_0$. We obtain eventually
\begin{eqnarray*}
&&\|\p_t g_1\|_{L^\infty([0,\bar T], L^2}^2 + \|\p_x\p_t g_1\|_{L^2([0,\bar T]\times \R_+)}^2 \\
&\leq & C \left( \cE_0 + (\cE_0 + \|\beta\|_{H^1(0,\bar T)}^2)^2 +  \|(1+\sqrt{x}) \p_x \tw^0\|_{L^2}^{4/3}\right).
\end{eqnarray*}

\bigskip

$\rhd${\it $W^{1,\infty}(H^1)\cap H^1(H^2)$ estimate:}
We differentiate the equation~\eqref{eq:dxg1} satisfied by $\partial_x g_1$ with respect to time:
\begin{align}\label{eq:dtdx-g1-w0}
& \partial_t (\partial_t \partial_x g_1) + \partial_x \cA \partial_t \partial_x g_1 \nonumber \\
& = \beta \partial^2_x \partial_t g_1 + \beta'\partial^2_x g_1
+ \mu \beta' \partial^2_x \left(\dfrac{\partial_x \bv}{\bv^2} g\right)
+ \mu \beta \partial^2_x \left(\dfrac{\partial_x \bv}{\bv^2} \partial_t g\right) \\
& \quad + \mu \partial_x \cA \partial_x^2 \partial_t \left(\ln \left(1+ \dfrac{g}{\bv}\right) - \dfrac{g}{\bv} \right)
+ \tx'(t) \partial_x \cA \partial^2_x \tw^0(x + \tx(t))
\nonumber
\end{align}
As in the constant case $\tw^0 \equiv u_+$, we aim at obtaining a control of $\|\partial_t \partial_x g_1\|_{L^\infty(L^2)} + \|\beta'\|_{L^2}$.
For that purpose, we want to use $(\partial_t \partial_x g_1) \rho/\bv$ as a test function in the weak formulation associated with~\eqref{eq:dtdx-g1-w0}. This weak formulation is similar to \eqref{weak-dtdxg1}, with extra source and boundary terms. Using \eqref{trace-A-dx2-NL} and Lemma \ref{lem:traces}, we find that
for any test function $\psi\in L^\infty((0,\bar T), L^2(\R_+))\cap L^2((0,\bar T), H^1(\R_+))$, for almost every $t\in (0,\bar T)$,
\begin{eqnarray*}
&&\langle \p_t^2 \p_x g_1(t), \psi(t)\rangle_{H^{-1}, H^1}
-\int_0^\infty \cA(\p_t\p_x g_1)\; \p_x \psi \\&&-\left(\beta\p_t\p_x g_{1|x=0} + \beta' \p_xg_{1|x=0} + \p_t R_2\right)\psi_{|x=0}\\
&=&-\int_0^\infty \mu \cA \p_x^2 \p_t \left(\ln \left(1+ \dfrac{g}{\bv}\right) - \dfrac{g}{\bv} \right) \p_x \psi(t,x)\:dx \\
&& +\int_0^\infty \left[\beta \p_x^2\p_t g_1 + \beta'\partial^2_x g_1
+ \mu \beta' \partial^2_x \left(\dfrac{\partial_x \bv}{\bv^2} g\right)
+ \mu \beta \partial^2_x \left(\dfrac{\partial_x \bv}{\bv^2} \partial_t g\right) \right] \psi\\
&&+ \tx'(t) \int_0^\infty\partial_x \cA \partial^2_x \tw^0(x + \tx(t)) \psi\\
&&-\left[s\mu (v_+-2)\p_t T_2 - \mu \p_t T_3\right] \psi_{|x=0}.
\end{eqnarray*}

We now take $\psi=(\partial_t \partial_x g_1) \rho/\bv$, with $\rho\in \mathcal C^2_b(\R_+)$ such that $\rho(0)=2,$ $\rho'(0)=-4s/\mu$ as before.

The diffusion term is treated as previously:
\begin{align*}
-\int_0^\infty \cA(\p_t\p_x g_1)\; \p_x \left(\p_x \partial_t g_1 \dfrac{\rho}{\bv}\right)
& \geq \mu \int_{\R_+} \left( \partial_x \left(\dfrac{\partial_x \partial_t g_1}{\bv} \right) \right)^2 \rho 
 - C \|\rho\|_{W^{2,\infty}}  \int_{\R_+} \big(\partial_x \partial_t g_1\big)^2 \\
& \quad + 3s \left( \partial_t \p_x g_{1|x=0}\right)^2 .
\end{align*}

Recalling \eqref{g1-R1}, we have
\begin{align}
(\partial_x \partial_t g_1)_{|x=0}
& = \beta' \dfrac{s (v_+ -1)}{\mu} + \partial_t R_1.
\end{align}
Gathering all the boundary terms of the left-hand side of the weak formulation, we obtain, using once again Lemma \ref{lem:traces}
\begin{eqnarray*}
&&-2\left(\beta\p_t\p_x g_{1|x=0} + \beta' \p_xg_{1|x=0} + \p_t R_2\right)(\partial_t \partial_x g_1)_{|x=0}  + 3s \left( \partial_t \p_x g_{1|x=0}\right)^2 \\
&\geq & C^{-1}  \left( \partial_t \p_x g_{1|x=0}\right)^2 - C \left((\beta')^2 (\p_xg_{1|x=0})^2 + (\p_t R_2)^2\right)\\
&\geq & C^{-1}  (\beta')^2 - C \left( (\p_t R_1)^2 + |\beta'|^2 (\|\beta\|_{H^1(0,\bar T)}^2 + \|R_1\|_{H^1(0,\bar T)}^2) + (\p_t R_2)^2\right),
\end{eqnarray*}
for some constant $C\geq 1$ depending only on the parameters of the problem.
Note that
\begin{eqnarray*}
&&\int_0^{\bar T} \left( (\p_t R_1)^2 + |\beta'|^2 (\|\beta\|_{H^1(0,\bar T)}^2 + \|R_1\|_{H^1(0,\bar T)}^2) + (\p_t R_2)^2\right)\\
&\leq & C\left(\cE_0 + (\cE_0+ \|\beta\|_{H^1(0,\bar T)}^2)^2 \right).
\end{eqnarray*}

The reader may check that the other additional boundary terms can be controlled as well. Eventually, we are led to
\begin{eqnarray*}
&&\| \p_t \p_x g_1\|_{L^\infty([0,\bar T], L^2(\R_+)}^2 + \int_0^{\bar T} \int_0^\infty \left( \p_x \left(\frac{\p_t \p_x g_1}{\bv}\right) \right)^2
+ \|\beta'\|_{L^2(0, \bar T)}^2\\
&\leq & C\left(\cE_0 + \| (1+\sqrt{x}) \p_x^3 \tw^0\|_{L^2(\R_+)}^2 + (\cE_0+ \|\beta\|_{H^1(0,\bar T)}^2)^2 \right)\\
&&+ C  \|\p_t\p_x g_1\|_{L^2((0,\bar T)\times \R_+)}^2\\
&\leq & C \left( \cE_0+ \| (1+\sqrt{x}) \p_x^3 \tw^0\|_{L^2}^2 + (\cE_0 + \|\beta\|_{H^1(0,\bar T)}^2)^2 +  \|(1+\sqrt{x}) \p_x \tw^0\|_{L^2(\R_+)}^{4/3}\right).
\end{eqnarray*}

 \subsubsection*{Conclusion}
 
 We infer that
 \[
 \|\beta\|_{H^1(0,\bar T)}^2 \leq C \left( \cE_0+ (\cE_0 + \|\beta\|_{H^1(0,\bar T)}^2)^2 + \sum_{k=1}^3 \|(1+\sqrt{x}) \p_x^k \tw^0\|_{L^2(\R_+)}^{4/3}\right).
 \]
 Assume that
 \[
  \cE_0\leq c_0 \delta^2,\quad \|(1+\sqrt{x}) \p_x^k \tw^0\|_{L^2(\R_+)}\leq c_0\delta^{3/2}\quad \text{for }k=1,2,3.
 \]
 Then
 \[
 \|\beta\|_{H^1(0,\bar T)}^2 \leq C c_0 \delta^2 + C \delta^4.
 \]
 Once again, choosing  the constant $c_0$ small enough,  we obtain the desired result. This completes the proof of \cref{prop:bootstrap}.

\subsection{Long time behavior}
\label{ssec:stability}
\begin{cor}
Assume that the hypotheses of Theorem \ref{thm:main-glob} are satisfied.

Let $(\tx, v_s,u_s)$ be the unique global smooth solution to~\eqref{EDO-tx}-\eqref{eq:vs}-\eqref{eq:NS-shifted-1-u} constructed in Proposition~\ref{prop:fixed-pt}. Then 
\begin{equation}
|\tx'(t) - s| + \sup_{x \in \R_+} \Big| (v_s,u_s)(t,x) - (\bv, \bu)(x) \Big| \longrightarrow 0 \quad \text{as}~ t \to +\infty.
\end{equation}
Moreover, setting $p_s(t,x) = p_s(t) = -\mu [\partial_x u_s]_{|x=0}$, we also ensure that
\begin{equation*}
\big|(p_s - p_-)(t)\big| \longrightarrow 0 \quad \text{as}~ t \to +\infty,
\end{equation*}
where $p_- = -\mu [\partial_x \bu]_{|x=0} = s^2(v_+-1)$ has been defined in Lemma~\ref{lem:TW}.
\end{cor}

\medskip
\begin{proof}
The result easily follows from the controls of the solution $(\tx,v_s,u_s)$ and its time derivative. 
For instance, $v_s -\bv$ is controlled in $L^2 \cap L^\infty(\R_+; H^1(\R_+))$ and $\partial_t v_s \in L^2(\R_+; H^1(\R_+))$. 
Therefore,
\[
\|(v_s -\bv)(t)\|_{H^1(\R_+)} \longrightarrow 0 \quad \text{as}~ t\to +\infty,
\]
and 
\[
\sup_{x\in \R_+}\big| (v_s -\bv)(t,x) \big| \leq C \|(v_s -\bv)(t) \|_{L^2(\R_+)}^{1/2} \|\partial_x (v_s -\bv)(t) \|_{L^2(\R_+)}^{1/2} \longrightarrow 0 \quad \text{as}~ t\to +\infty.
\]
The long time behavior of $u_s-\bu$ and $\tx' -s$ can be derived with the same arguments. 
Finally, the long time behavior of $p_s$ is obtained through the control of $\partial_x(u_s -\bu)$ in $L^2 \cap L^\infty(\R_+; H^1(\R_+))$ and of $\partial_t \partial_x(u_s -\bu) \in L^2(\R_+; H^1(\R_+))$.

\bigskip

\end{proof}

\section*{Acknowledgement}
This project has received funding from the European Research Council (ERC) under the European Union's Horizon 2020 research and innovation program Grant agreement No 637653, project BLOC ``Mathematical Study of Boundary Layers in Oceanic Motion’’. This work was supported by the SingFlows and CRISIS projects, grants ANR-18-CE40-0027 and ANR-20-CE40-0020-01 of the French National Research Agency (ANR).
A.-L. D. acknowledges the support of the Institut Universitaire de France.
This material is based upon work supported by the National Science Foundation under Grant No. DMS-1928930 while the authors participated in a program hosted by the Mathematical Sciences Research Institute in Berkeley, California, during the Spring 2021 semester.

\newpage
\appendix

\section{Linear parabolic equations with Dirichlet boundary conditions}

In this section, we consider equations of the type
\be
\label{eq:lin-Dirichlet}
\ba
\p_t u +\p_x( b(t,x)  u )+ c(t,x) u -\p_{x}(a(t,x) \p_x u)=f(t,x),\quad t>0, \ x>0\\
u_{|x=0}=0,\quad 
u_{|t=0}=u^0.
\ea
\ee
We denote $\Omega_T:=(0,T)\times \R_+$. We will always assume that there exists $\alpha>0$ such that
\be
\label{hyp:minimales-coeffs}
a,b,c\in L^\infty([0,T]\times \R_+),\quad \inf_{[0,T]\times \R_+} a\geq \alpha>0.
\ee
Note that under such assumptions, it is classical that for any $f\in L^2([0,T], H^{-1}(\R_+))$, for any $u^0\in L^2(\R_+)$, equation \eqref{eq:lin-Dirichlet} has a unique weak solution $u\in L^\infty([0,T], L^2(\R_+))\cap L^2([0,T], H^1_0(\R_+))$.
We will now derive regularity estimates for $u$. These estimates are quite standard, and can be found in numerous textbooks on partial differential equations. 
Since they are used repeatedly in this paper, we gathered them in the following Proposition, for which we provide a proof for the reader's convenience.

\begin{prop}\label{prop:est-gal}Let $T>0$.
Assume that $a\in H^1([0,T];L^2(\R_+)) \cap L^2([0,T]; H^2(\R_+))$, $b\in W^{1,\infty}(\Omega_T)$, $c\in W^{1,\infty}([0,T], L^\infty(\R_+))$.\\
Let $u^0\in H^1(\R_+)$ such that $u^0(0)=0$, and let $f\in L^2(\Omega_T)$.
Let $u\in L^\infty([0,T], L^2(\R_+))\cap L^2([0,T], H^1_0(\R_+))$ be the unique weak solution of \eqref{eq:lin-Dirichlet}.

Then $u \in L^\infty([0,T];H^1(\R_+)) \cap H^1([0,T]; L^2(\R_+)) \cap L^2([0,T];H^2(\R_+))$ and we have the classical energy estimate
\be\label{est:energy-gal-1}
 A_0:=\| u\|_{L^\infty([0,T], L^2)}^2 + \| \p_x u\|_{L^2(\Omega_T)}^2 
\leq  C (\|f\|_{L^2(\Omega_T)}^2 + \| u^0\|_{L^2(\R_+)}^2)e^{CT}
\ee
for some constant $C=C(\alpha, \|b\|_\infty, \|c\|_\infty)$,
as well as the higher regularity estimates
\begin{align}\label{est:energy-gal-u_x}
A_1&:= \|\partial_x u\|_{L^\infty L^2}^2 + \|\partial_t u\|_{L^2 L^2}^2 + \|\partial^2_x u\|_{L^2 L^2}^2 \nonumber \\
& \leq C \Bigg[
\|a\|_{L^\infty}\|\partial_x u^0\|_{L^2}^2 + \big( \|\partial_x b\|_{L^2 L^\infty}^2 + \|c\|_{L^2L^\infty}^2\big) \|u\|_{L^\infty L^2}^2 + \|f\|_{L^2 L^2}^2\\ 
& \qquad + \big(\|\partial_t a\|_{L^\infty L^2}+ \|\partial_t a\|_{L^\infty L^2}^2 + \|\partial_x a\|_{L^\infty L^4}^2 + \|\partial_x a\|_{L^\infty L^4}^4 + \|b\|_{L^\infty}^2 \big) \|\partial_x u\|_{L^2 L^2}^2 
\Bigg]  \nonumber 
\end{align}
and
\begin{align}\label{est:energy-gal-u_xx}
& \|\p_t u\|_{L^\infty([0,T], L^2)}^2 + \| \p_x\p_t u\|_{L^2L^2}^2 \nonumber \\
&\leq C\Big[ \|u^0\|_{H^2(\R_+)}^2 \left(1 + \| b\|_{L^\infty(H^1)}^2 + \|c\|_\infty^2 + \| a\|_{L^\infty(H^1)}^2\right)+ A_0\|\partial_t b\|_{L^2 L^\infty}^2\\ 
&\qquad+ A_1\big(\|\partial_t c\|_{L^2(\Omega_T)}^2+ \|c\|_{L^\infty} + \|\partial_t a\|_{L^\infty L^2} + \|b\|_\infty+1\big)  + \|\partial_t f\|_{L^2(\Omega_T)}^2 \Big]
\end{align}

\end{prop}

\begin{proof}
The weak solution $u$ of \eqref{eq:lin-Dirichlet} can be defined as the limit as $R\to \infty$ of the solution $u_R$ of the linear equation in a bounded domain
\be\label{eq:uR}
\ba
\p_t u_R +\p_x( b(t,x)  u_R )+ c(t,x) u_R -\p_{x}(a(t,x) \p_x u_R)=f(t,x),\quad t>0, \ x\in (0,R)\\
u_{R|x=0}=u_{R|x=R}=0,\quad 
u_{|t=0}=u^0_R,
\ea
\ee
where $u^0_R=u^0 \chi_R$, where $\chi_R(x)=1$ if $x\leq R-2$ and $\chi_R(x)=0$ if $x\geq R-1$, with $\|\chi_R\|_{W^{2,\infty}}\leq C$.
In turn, the function $u_R$ is defined as the limit of the Galerkin approximation
\[
u_{R,n}(t,x)=\sum_{k=1}^n d^{R,n}_k(t) w_{k,R}(x),
\]
where $w_{k,R}$ are eigenfunctions of the Laplacian with Dirichlet boundary conditions in $(0,R)$, normalized in $L^2$, namely $w_{k,R}(x)= (2/R)^{1/2} \sin (k\pi x/R)$. More precisely, for all $k\in \{1,\cdots,n\} $, for a.e. $t\in [0,T]$
\begin{multline}
\label{Galerkin}
\int_0^R\p_t u_{R,n}(t)w_{k,R} - \int_0^R b(t)u_{R,n}(t) \p_x w_{k,R} + \int_0^R c(t)u_{R,n}(t)  w_{k,R} + \int_0^R a(t) \p_x u_{R,n}(t) \p_x w_{k,R} \\=\int_0^R f(t) w_{k,R}.
\end{multline}
The initial data $u_{R,n|t=0}$ is the orthogonal projection in $L^2(0,R)$ of $u_R^0$ onto $\mathrm{Span} (w_{1,R},\cdots, w_{n,R})$.

{\it Energy estimate~\eqref{est:energy-gal-1}.}
The classical energy estimate is obtained by multiplying~\eqref{Galerkin} by $d^{R,n}_k(t)$ and summing for $k\in \{1,\cdots, n\}$. 
Using Cauchy-Schwarz inequalities, we infer that 
\[
\frac{d}{dt}\| u_{R,n}(t)\|_{L^2(0,R)}^2 + \alpha \| \p_x  u_{R,n}(t)\|_{L^2(0,R)}^2 \leq \|f(t)\|_{L^2(0,R)}^2 + \left(1 + \frac{\|b\|_\infty^2}{\alpha} + \| c\|_\infty\right) \| u_{R,n}(t)\|_{L^2(0,R)}^2.
\]
Integrating in time and recalling that $\| u_{R,n}(t=0)\|_{L^2(0,R)} \leq \| u_R^0\|_{L^2([0,R])}\leq \| u^0\|_{L^2(\R_+)}$, we deduce that $u_{R,n}$ satisfies \eqref{est:energy-gal-1} for all $R,n$. Passing to the limit as $n\to \infty$, we deduce that $u_R$ also satisfies \eqref{est:energy-gal-1}.

\medskip
{\it Estimate~\eqref{est:energy-gal-u_x}.}
We first multiply \eqref{Galerkin} by $(d^{R,n}_k)'(t)$ and sum for $k\in \{1,\cdots, n\}$. After integrations by parts in space and time, we get
\begin{align*}
& \int_0^R \frac{a(t)}{2} |\partial_x u_{R,n}(t)|^2 +  \int_0^t \int_0^R |\partial_t u_{R,n}|^2 \nonumber \\
& = \int_0^R \frac{a(0)}{2} |\partial_x u_{R,n}(0)|^2
- \int_0^t \int_0^R \partial_x(b u_{R,n}) \partial_t u_{R,n} 
- \int_0^t \int_0^R c u_{R,n} \partial_t u_{R,n} \\
& \quad + \int_0^t \int_0^R f \ \partial_t u_{R,n}
+ \int_0^t \int_0^R \frac{\partial_t a}{2} |\partial_x u_{R,n}|^2 \nonumber
\end{align*}
and therefore
\begin{align}\label{eq:du1}
& \int_0^R \frac{a(t)}{2} |\partial_x u_{R,n}(t)|^2 +  \dfrac{1}{2} \int_0^t \int_0^R |\partial_t u_{R,n}|^2 \nonumber \\
& \leq  \int_0^R \frac{a(0)}{2} |\partial_x u_{R,n}(0)|^2 
+ \int_0^t \int_0^R \frac{|\partial_t a|}{2} |\partial_x u_{R,n}|^2 \\
& \quad + C \Big[ 
\|b\|_{L^\infty(\Omega_T)}^2 \|\partial_x u_{R,n}\|_{L^2(\Omega_T)}^2 
+(\|\partial_x b\|_{L^2 L^\infty}^2 + \|c\|_{L^2 L^\infty}^2 ) \|u_{R,n}\|_{L^\infty L^2}^2  + \|f\|_{L^2(\Omega_T)}^2 \Big] \nonumber
\end{align}
The crucial point here is to notice that thanks to the compatibility condition $u_R^0(0)=u_R^0(R)=0$, we have
\be\label{est:pythagore}
\|  \p_x u_{R,n}(t=0)\|_{L^2(0,R)}\leq \| \p_x u_R^0\|_{L^2(0,R)}\leq C \| u^0\|_{H^1(\R_+)}. 
\ee
Indeed, 
\[
\p_x u_{R,n}(t=0)=\sum_{k=0}^n d^{R,n}_k(0) \p_x w_{k,R}.
\]
Computing $d^{R,n}_k(0)$ and using the compatibility conditions on $u_R^0$, we infer that
\[
d^{R,n}_k(0)= \sqrt{\frac{2}{R}}\int_0^R u_R^0(x)\sin \left( \frac{k\pi x}{R}\right)dx=\sqrt{\frac{2}{R}}\frac{R}{k\pi}\int_0^R\p_x  u_R^0(x)\cos \left( \frac{k\pi x}{R}\right)dx.
\]
Setting $v_{k,R}:=\sqrt{\frac{2}{R}}\cos(k\pi x/R)$ for $k\geq 1$, we notice that $(v_{k,R})_{k\geq 1}$ is an orthonormal family in $L^2(0,R)$, and that
\[
\p_x u_{R,n}(t=0)=\sum_{k=0}^n \left( \p_x u_R^0, v_{k,R}\right) v_{k,R},
\]
where $(\cdot ,\cdot )$ is the usual $L^2$ scalar product. Inequality \eqref{est:pythagore} then follows from the Bessel inequality.

Next, we multiply \eqref{Galerkin} by $\lambda^R_k d^{R,n}_k(t)$ where $\lambda^R_k$ is the $k$-th eigenvalue of the Laplacian. 
Summing for $k\in \{1,\cdots, n\}$, we get after integrations by parts:
\begin{align*}
- \int_0^t \int_0^R \partial_x(a\partial_x u_{R,n}) \partial^2_x u_{R,n} 
& = - \int_0^t \int_0^R \partial_x(bu_{R,n}) \partial^2_x u_{R,n}
- \int_0^t\int_0^Rcu_{R,n} \partial^2_x u_{R,n} \\
& \quad + \int_0^t\int_0^R f \ \partial^2_x u_{R,n} 
- \int_0^t\int_0^R \partial_t u_{R,n} \partial^2_x u_{R,n}. \nonumber 
\end{align*}
Using the Cauchy-Schwarz inequality, we deduce that 
\begin{align}\label{eq:du2}
& \int_0^t \int_0^R a |\partial^2_x u_{R,n}|^2 \nonumber\\
& \leq {C} \Big[\|\partial_t u_{R,n}\|_{L^2(\Omega_T)}^2 
+ \int_0^t\int_0^R \left|\frac{\partial_x a}{\sqrt{a}}\right|^2 |\partial_x u_{R,n}|^2  
+ \|b\|_{L^\infty(\Omega_T)}^2 \|\partial_x u_{R,n}\|_{L^2(\Omega_T)}^2 \\
& \qquad +(\|\partial_x b\|_{L^2 L^\infty}^2 + \|c\|_{L^2 L^\infty}^2 ) \|u_{R,n}\|_{L^\infty L^2}^2 
+ \|f\|_{L^2(\Omega_T)}^2
\Big] .\nonumber
\end{align}
Combining \eqref{eq:du1} and \eqref{eq:du2}, we obtain
\begin{align*}
& \int_0^R \frac{a(t)}{2} |\partial_x u_{R,n}(t)|^2 + \frac{1}{4} \int_0^t \int_0^R |\partial_t u_{R,n}|^2
+ \frac{1}{4{C}}\int_0^t \int_0^R a |\partial^2_x u_{R,n}|^2  \nonumber \\
& \leq  \int_0^R \frac{a(0)}{2} |\partial_x u_{R,n}(0)|^2 
+ \int_0^t \int_0^R \frac{|\partial_t a|}{2} |\partial_x u_{R,n}|^2 
+ \frac{1}{4{C}}\int_0^t\int_0^R \left|\frac{\partial_x a}{\sqrt{a}}\right|^2 |\partial_x u_{R,n}|^2 \\
& \quad + C \Big[ 
\|b\|_{L^\infty(\Omega_T)}^2 \|\partial_x u_{R,n}\|_{L^2(\Omega_T)}^2 
+(\|\partial_x b\|_{L^2 L^\infty}^2 + \|c\|_{L^2 L^\infty}^2 ) \|u_{R,n}\|_{L^\infty L^2}^2  + \|f\|_{L^2(\Omega_T)}^2 \Big].
\end{align*} 
Moreover,
\begin{align*}
\int_0^R |\partial_t a| |\partial_x u_{R,n}|^2 
& \leq \|\partial_t a \|_{L^2(0,R)} \|\partial_x u_{R,n}\|_{L^2(0,R)} \|\partial_x u_{R,n}\|_{L^\infty(0,R)} \\
& \leq C \|\partial_t a \|_{L^2(0,R)} \|\partial_x u_{R,n}\|_{L^2(0,R)} \|\partial_x u_{R,n}\|_{H^1(0,R)}
\end{align*}
so that
\begin{align*}
\int_0^t \int_0^R |\partial_t a| |\partial_x u_{R,n}|^2 
\leq \eta \|\partial^2_x u_{R,n}\|_{L^2(\Omega_T)}^{2} + C_\eta \left(\|\partial_t a\|_{L^\infty L^2}+ \|\partial_t a\|_{L^\infty L^2}^2\right) \|\partial_x u_{R,n}\|_{L^2(\Omega_T)}^2.
\end{align*}
In the same manner,
\begin{align*}
\int_0^t\int_0^R \left|\frac{\partial_x a}{\sqrt{a}}\right|^2 |\partial_x u_{R,n}|^2 
\leq \eta \|\partial^2_x u_{R,n}\|_{L^2(0,R)}^{2} + C_\eta\left( \|\partial_x a\|_{L^\infty L^4}^2 +\|\partial_x a\|_{L^\infty L^4}^4\right)  \|\partial_x u_{R,n}\|_{L^2(\Omega_T)}^2.
\end{align*}
As a consequence, choosing $\eta$ small enough, and setting $J(\xi)=\xi + \xi^2$ for $\xi\in \R$,
\begin{align*}
& \alpha \int_0^R |\partial_x u_{R,n}(t)|^2 + \int_0^t \int_0^R |\partial_t u_{R,n}|^2
+ \alpha \int_0^t \int_0^R |\partial^2_x u_{R,n}|^2  \nonumber \\
& \leq C \Bigg[ \|a\|_{L^\infty} \|\partial_x u_{R,n}(0)\|_{L^2(0,R)}^2 
+ \left(J\left(\|\partial_t a\|_{L^\infty L^2}\right) + J\left(\|\partial_x a\|_{L^\infty L^4}^2\right)\right) \|\partial_x u_{R,n}\|_{L^2(\Omega_T)}^2 \\
& \quad + \|b\|_{L^\infty(\Omega_T)}^2 \|\partial_x u_{R,n}\|_{L^2(\Omega_T)}^2 
+(\|\partial_x b\|_{L^2 L^\infty}^2 + \|c\|_{L^2 L^\infty}^2 ) \|u_{R,n}\|_{L^\infty L^2}^2  + \|f\|_{L^2(\Omega_T)}^2 \Bigg]
\end{align*} 
and eventually, passing to the limit $n\to +\infty$
\begin{align*}
& \| \p_x u_{R}\|_{L^\infty((0,T), L^2((0,R))}^2 + \| \p_t u_{R}\|_{L^2((0,T)\times (0,R))}^2 + \|\p_x^2 u_{R}\|_{L^2((0,T)\times (0,R))}^2 \\
& \leq C \Bigg[
\|a\|_{L^\infty}\|\partial_x u^0_R\|_{L^2}^2 +  \left(J\left(\|\partial_t a\|_{L^\infty L^2}\right) + J\left(\|\partial_x a\|_{L^\infty L^4}^2\right)\right) \|\partial_x u_R\|_{L^2 L^2}^2 
+ \|b\|_{L^\infty}^2 \|\partial_x u_R\|_{L^2 L^2}^2 \\
& \qquad + \big( \|\partial_x b\|_{L^2 L^\infty}^2 + \|c\|_{L^2L^\infty}^2\big) \|u_R\|_{L^\infty L^2}^2 + \|f\|_{L^2 L^2}^2
\Bigg]. \nonumber
%
\end{align*}

\medskip
{\it Estimate~\eqref{est:energy-gal-u_xx}.}
In a last step, we differentiate \eqref{Galerkin} with respect to $t$.
We find that $u_{R,n}':=\p_t u_{R,n}$ also satisfies a weak formulation similar to \eqref{Galerkin}, namely
\begin{multline}\label{Galerkin-bis}
   \int_0^R\p_t u_{R,n}'(t)w_{k,R} - \int_0^R b(t)u_{R,n}'(t) \p_x w_{k,R} + \int_0^R c(t)u_{R,n}'(t) w_{k,R} + \int_0^R a(t) \p_x u_{R,n}'(t) \p_x w_{k,R} \\
   =\int_0^R \tilde f_{R,n}(t) w_{k,R} + \int_0^R g_{R,n}(t) \p_x w_{k,R} 
 \end{multline}
 where
\[\ba
\tilde f_{R,n}:= \p_t f + \p_t c  u_{R,n} ,\\
g_{R,n}:=- \p_t b u_{R,n} + \p_t a \p_x u_{R,n}\ea
\]
and with the initial data
\[
u_{R,n}'(0)=\mathbb P_n\left[ -\p_x (b(t=0) u_{R}^0) - c(t=0) u_R^0 + \p_x (a(t=0) \p_x u_R^0)\right],
\]
where $\mathbb P_n$ is the orthogonal projection in $L^2(0,R)$ onto $\mathrm{Span}(w_{1,R}, \cdots, w_{n,R})$. 
Using \eqref{est:energy-gal-u_x}, we obtain 
\begin{align*}
\| \tilde f_{R,n}\|_{L^2((0,T)\times (0,R))} 
& \leq C \left(\|\p_t f\|_{L^2(\Omega_T)} + \|u_{R,n}\|_{L^\infty([0,T], H^1(0,R))} \|\p_t c\|_{L^2(\Omega_T)}\right),\\
\|g_{R,n}\|_{L^2((0,T)\times (0,R))} 
& \leq C\Big[ \|u_{R,n}\|_{L^\infty([0,T], L^2(0,R))} \|\p_t b\|_{L^2([0,T];L^\infty(\R_+))} \\
& \qquad + \|\p_t a\|_{L^\infty([0,T], L^2(\R_+))}\|\p_x u_{R,n}\|_{L^2([0,T], H^1)}\Big],
\end{align*}
so that
\begin{align*}
& \|u'_{R,n}\|_{L^\infty L^2}^2 + \|\partial_x u'_{R,n}\|_{L^2L^2}^2 \\
& \leq C\Big[ \|u^0\|_{H^2(\R_+)}^2 \left(1 + \| b\|_{L^\infty(H^1)}^2 + \|c\|_\infty^2 + \| a\|_{L^\infty(H^1)}^2\right)+ A_0\|\partial_t b\|_{L^2 L^\infty}^2\\ 
&\qquad+ A_1\big(\|\partial_t c\|_{L^2(\Omega_T)}^2+ \|c\|_{L^\infty} + \|\partial_t a\|_{L^\infty L^2} + \|b\|_\infty+1\big)  + \|\partial_t f\|_{L^2(\Omega_T)}^2 \Big]
\end{align*}
Passing to the limit as $n\to \infty$, we deduce that $u_R$ satisfies \eqref{est:energy-gal-u_xx}.

\bigskip

We then extend $u_R$ to $(0,T)\times (0, +\infty)$ in such a way that the extension satisfies the same bounds as $u_R$. 
The family thus obtained is compact in $L^2_\text{loc}([0,T]\times\R_+)$, and we can extract a subsequence converging weakly in $w^*-L^\infty([0,T], H^2(\R_+)$, and whose time-derivative converges weakly in $w^*-L^\infty([0,T], L^2(\R_+)$ and in $w-L^2([0,T], H^1(\R_+))$.
Passing to the limit in \eqref{eq:uR}, we find that the limit is a solution of \eqref{eq:lin-Dirichlet}, satisfying the estimates~\eqref{est:energy-gal-1}, \eqref{est:energy-gal-u_x} and~\eqref{est:energy-gal-u_xx}.
\end{proof}

\section{\texorpdfstring{$L^2$}{L2} estimates for some nonlinear parabolic equations}
In this section, we consider diffusion equations of the form
\be\label{eq:generique-g}
\ba
\p_t g -\tilde y'(t)\p_x g-\mu \p_{xx} \ln \left(1+\frac{g}{\bar g}\right)= G\quad t\in (0,T), \ x\in (0,R),\\
g_{|t=0}=g_0\in H^1_0((0,R)),\\
g_{|x=0}=0,\\
g_{|x=R}=0\quad \text{if }R<\infty, \quad \lim_{x\to \infty} g(t,x)=0\quad \text{if }R=+\infty.
\ea
\ee
We prove the following
\begin{lem}
Let $G\in L^2((0,T)\times (0,R))$, $g_0\in H^1_0((0,R))$. Assume that $\bar g \in  W^{1,\infty}\cap H^1((0,T)\times (0,R))$ and that $\bar g_{|x=0}=1$. Assume furthermore that $M^{-1}\leq \ty'(t)\leq M$ for a.e. $t\in (0,T)$, for some constant $M\geq 1$.

Let $g\in L^\infty([0,T], H^1(0,R))$ be a solution of \eqref{eq:generique-g} such that $\p_t g\in L^2((0,T)\times (0,R))$. 
Assume furthermore that $1\leq g + \bar g \leq \bar C$, $1\leq \bar g \leq \bar C$ a.e. for some constant $\bar C>1$. 

Then $g/\bar g\in  L^\infty([0,T], H^1(0,R))\cap L^2((0,T), H^2((0,R)))$, and there exists a constant $C$ depending only on $\mu, M$ and $\|\bar g\|_{W^{1,\infty}([0,T]\times [0,R])}$ such that
\begin{eqnarray*}
&&\|g\|_{L^\infty([0,T], H^1(0,R))} + \| \p_t g\|_{ L^2((0,T)\times (0,R))} +  \|\p_x g \|_{ L^2((0,T)\times (0,R))}   \\
&\leq & C  \left( \|g_0\|_{H^1} + \| G\|_{L^2((0,T)\times (0,R))} +\|g\|_{L^2((0,T)\times (0,R))}\right),
\end{eqnarray*}
and, if $R=+\infty$,
\begin{eqnarray*}\left\| \p_x^2 \left(\frac{g}{\bar g}\right)\right\|_{L^2((0,T)\times \R_+)}
&\leq & C \left( \|g_0\|_{H^1} + \| G\|_{L^2((0,T)\times \R_+)} +\|g\|_{L^2((0,T)\times \R_+)}\right)\\
&+& C \inf(1, T^{1/2})\left( \|g_0\|_{H^1} + \| G\|_{L^2((0,T)\times \R_+)} +\|g\|_{L^2((0,T)\times \R_+)}\right)^2.
\end{eqnarray*}
Additionally, for all $R>0$, $g$ satisfies  the exponential growth estimate
\begin{align*}
    &\|g\|_{L^\infty([0,T], H^1(0,R))} + \| \p_t g\|_{ L^2((0,T)\times (0,R))} +  \|\p_x g \|_{ L^2((0,T)\times (0,R))}\\
    &\leq C  \left( \|g_0\|_{H^1} + \| G\|_{L^2((0,T)\times (0,R))}\right) \exp((1+ \|\p_x \bar g\|_\infty^2) T).
\end{align*}

\label{lem:generique-g}
\end{lem}
\begin{proof}The proof is quite classical. The only subtlety lies in the derivation of the $L^\infty(H^1)$ estimate.  We only sketch the main steps in order to highlight the dependency on $\bar g$. Note that our purpose here is merely to derive energy estimates once the regularity of the solution is known, rather than proving extra regularity for the solution.

$\bullet$ We start with the $L^\infty(L^2)$ estimate. Multiplying \eqref{eq:generique-g} by $g$ and integrating by parts, we obtain
\[
\frac{1}{2}\frac{d}{dt}\int_0^R|g|^2 + \mu \int_0^R \frac{\p_x(\frac{g}{\bar g}) }{1 + \frac{g}{\bar g}}\p_x g = \int_0^R G g.
\]
Recalling that $\bar g\geq 1,\ g+ \bar g\geq 1$, we have
\[
\int_0^R \frac{\p_x(\frac{g}{\bar g}) }{1 + \frac{g}{\bar g}}\p_x g\geq \frac{1}{2}\int_0^R\frac{(\p_x g)^2}{g+ \bar g} - \frac{1}{2} \|\p_x \bar g\|_\infty^2 \int_0^R g^2,
\]
and therefore
\begin{eqnarray}\label{est:g-generique-L^2}
&&\|g\|_{L^\infty([0,T], L^2(0,R))} + \|\p_x g\|_{L^2((0,T)\times (0,R))}\\
&\leq& C \left(\|g_0\|_{L^2(0,R)} +  \| G\|_{L^2((0,T)\times (0,R))} + (1+ \|\p_x \bar g\|_\infty) \|g\|_{L^2((0,T)\times (0,R))}\right).\nonumber
\end{eqnarray}
Note that we also obtain the exponential estimate
\begin{eqnarray*}
&&\|g\|_{L^\infty([0,T], L^2(0,R))} + \|\p_x g\|_{L^2((0,T)\times (0,R))}\\
&\leq& C \left(\|g_0\|_{L^2(0,R)} +  \| G\|_{L^2((0,T)\times (0,R))}\right)\exp(  (1+ \|\p_x \bar g\|_\infty^2)T).\nonumber
\end{eqnarray*}

$\bullet$ Let us now tackle the $L^\infty(H^1)$ estimate. We multiply \eqref{eq:generique-g} by $\p_t \ln (1+\frac{g}{\bar g})$ and examine each term separately. We have
\begin{eqnarray*}
\int_0^R \p_t g \p_t \ln \left( 1 + \frac{g}{\bar g}\right) &=& \int_0^R \frac{\left(\p_t \frac{g}{\bar g}\right)^2 }{1 + \frac{g}{\bar g}}\bar g - \int_0^R \frac{\p_t \bar g}{\bar g} g \frac{\p_t \frac{g}{\bar g}}{1 + \frac{g}{\bar g}}\\
&\geq & \frac{1}{2} \int_0^R\frac{\left(\p_t \frac{g}{\bar g}\right)^2 }{g+ \bar g}\bar g^2 -\frac{1}{2}\|\p_t \bar g\|_\infty^2 \int_0^R g^2.
\end{eqnarray*}
As for the diffusion term, we recognize a time derivative
\begin{eqnarray*}
        &&-\mu\int_0^R \p_{xx}\ln\left( 1 + \frac{g}{\bar g}\right)\p_t \ln \left( 1 + \frac{g}{\bar g}\right)\\
        &=&\mu \int_0^R \p_x\ln\left( 1 + \frac{g}{\bar g}\right) \p_t\p_x \ln\left( 1 + \frac{g}{\bar g}\right)\\
        &=&\frac{\mu}{2}\frac{d}{dt}\int_0^R \left(\p_x \ln\left( 1 + \frac{g}{\bar g}\right)\right)^2.
\end{eqnarray*}
Transferring the transport term into the right-hand side and using a Cauchy-Schwarz inequality, we obtain
\begin{eqnarray*}
&&\left| \int_0^R \left(G-\tilde y'(t)\p_x g\right) \p_t \ln \left( 1 + \frac{g}{\bar g}\right) \right|\\
&\leq & \frac{1}{4}\int_0^R\frac{\left(\p_t \frac{g}{\bar g}\right)^2 }{g+ \bar g}\bar g^2  + C \int_0^R (G^2 + (\p_x g)^2).
\end{eqnarray*}

Gathering all the terms and recalling that $\bar g \geq 1$, $g+\bar g \leq \bar C$, we obtain
\begin{eqnarray*}
&&\left\| \p_t \frac{g}{\bar g}\right\|_{L^2((0,T)\times (0,R))} + \left\| \p_x \frac{g}{\bar g}\right\|_{L^\infty((0,T), L^2 (0,R))}\\
&\leq & C \Big(\left\| \p_x \frac{g_0}{\bar g(t=0)}\right\|_{L^2 (0,R)} + \| G\|_{L^2((0,T)\times (0,R))} + \| \p_x g\|_{L^2((0,T)\times (0,R))} \\
&&\qquad+ \|\p_t \bar g\|_\infty \|g\|_{L^2((0,T)\times (0,R))}\Big).
\end{eqnarray*}
Recalling \eqref{est:g-generique-L^2}, we obtain the estimate announced in the Lemma for $\|\p_x g\|_{L^\infty(H^1)}$ and $\|\p_t g\|_{L^2(L^2)}$. 

$\bullet$ Let us now derive the estimate on the second derivative in the case $R=+\infty$. Let us set $f=\p_x \ln (1+ \frac{g}{\bar g})$. According to the previous estimates, we know that $f\in L^\infty((0,T), L^2(\R_+)$. Furthermore, using equation \eqref{eq:generique-g}, we have
\[
\mu \p_x f = \p_t g - \tilde y'(t) \p_x g - G\in L^2((0,T)\times \R_+).
\]
Hence, using the previous estimates,
\begin{eqnarray*}
   \|\p_x f\|_{L^2((0,T)\times \R_+)}&\leq& C \left(\|\p_t g\|_{L^2((0,T)\times \R_+)} + \|\p_x g\|_{L^2((0,T)\times \R_+)} + \|G\|_{L^2((0,T)\times \R_+)}\right) \\
   &\leq & C \left(\|g_0\|_{H^1(\R_+)} + \|g\|_{L^2((0,T)\times \R_+)} + \|G\|_{L^2((0,T)\times \R_+)}\right).
\end{eqnarray*}
In particular, the Gagliardo-Nirenberg-Sobolev inequality entails that
\begin{eqnarray*}
\| f\|_{L^4((0,T)\times\R_+)}&\leq& C \left(\int_0^T \|f(t)\|_{L^2(\R_+)}^3 \|\p_x f(t)\|_{L^2(\R_+)}\right)^{1/4} \\&\leq& C\|f\|_{L^\infty((0,T), L^2(\R_+))}^{1/2} \| f\|_{L^2((0,T)\times \R_+)}^{1/4}\| \p_x f\|_{L^2((0,T)\times \R_+)}^{1/4}\\
&\leq & C \left(\|g_0\|_{H^1(\R_+)} + \|g\|_{L^2((0,T)\times \R_+)} + \|G\|_{L^2((0,T)\times \R_+)}\right).
\end{eqnarray*}

Now
\[
\p_x f=\p_{xx}\ln\left(1 + \frac{g}{\bar g}\right) = \p_x \left(\frac{\p_x \frac{g}{\bar g}}{1 + \frac{g}{\bar g}}\right)= \frac{\p_{xx} \frac{g}{\bar g}}{1 + \frac{g}{\bar g}} - f^2.
\]
Thus
\begin{eqnarray*}
&&\|\p_{xx} g\|_{L^2((0,T)\times \R_+)}\\&\leq & C \left( \|\p_x f\|_{L^2((0,T)\times \R_+)} + \|f\|_{L^4((0,T)\times \R_+)}^2\right)\\
&\leq &  C  \left(\|g_0\|_{H^1(\R_+)} + \|g\|_{L^2((0,T)\times \R_+)} + \|G\|_{L^2((0,T)\times \R_+)}\right)\\
&&+ C \inf(1, T^{1/2})\left(\|g_0\|_{H^1(\R_+)} + \|g\|_{L^2((0,T)\times \R_+)} + \|G\|_{L^2((0,T)\times \R_+)}\right)^2.
\end{eqnarray*}

\end{proof}

\section{Two technical Lemmas}
\begin{lem}\label{lem:x+y}
Let $F\in L^2(\R_+)$ such that $\sqrt{x} F\in L^2(\R_+)$, and let $M\geq 1$.
Let $\ty\in L^\infty([0,T])$ such that $\ty(t)\geq t/M$ a.e.
Then for all $T>0$, $R>0$
\[
\int_0^T \int_0^R F^2(x+\ty(t))dx\:dt\leq M\int_0^\infty z F^2(z)\:dz.
\]
\end{lem}
\begin{proof}
We write
\begin{eqnarray*}
\int_0^T \int_0^R F^2(x+\ty(t))dx\:dt&=&\int_0^T\int_{\ty(t)}^{R+\ty(t)} F^2(z)\:dz\:dt\\
&\leq & \int_0^T \int_0^\infty \mathbf 1_{z>\ty(t)} F^2(z)\:dz\: dt.
\end{eqnarray*}
Using the assumption on $\ty$, we deduce
\begin{eqnarray*}
\int_0^T \int_0^R F^2(x+\ty(t))dx\:dt
&\leq & \int_0^T \int_0^\infty \mathbf 1_{z>t/M} F^2(z)\:dz\:dt\\
&\leq & M\int_0^\infty z F^2(z)\:dz.
\end{eqnarray*}
\end{proof}

\begin{lem}
Let $\ty_1, \ty_2\in L^\infty_\text{loc}([0,T])$ and such that $\ty_i'\in L^2\cap L^\infty(0,T)$. 
Assume that $\ty_1(0)=\ty_2(0)=0$ and that $M^{-1}\leq \ty_i'\leq M$ for some constant $M\geq 1$ and for $i=1,2$. Let $\tw^0\in W^{1,\infty}(\R_+)$ such that $\sqrt{x}\p_x \tw^0 , \sqrt{x}\p_x^2 \tw^0\in L^2(\R_+)$.

Then there exists a  constant $C$ depending only on $M$ such that
\begin{align*}
& \|\tw^0 (\ty_1)- \tw^0(\ty_2)\|_{L^\infty(0,T)}\\&\leq C\left(\| \sqrt{x} \p_x \tw^0\|_{L^2(\R_+)} +\| \sqrt{x} \p_x^2 \tw^0\|_{L^2(\R_+)} + \|\p_x \tw^0\|_{L^\infty(\R_+)}\right) \|\ty_1'-\ty_2'\|_{L^2(0,T)}
    \end{align*}
and
\[
\|\tw^0 (\ty_1)- \tw^0(\ty_2)\|_{L^2(0,T)}\leq C \|\ty_1'-\ty_2'\|_{L^2(0,T)}\| \sqrt{z} \p_x \tw^0(z)\|_{L^2(\R_+)}.
\]

\label{lem:Taylor}
\end{lem}
\begin{proof}
First, using the Sobolev embedding $H^1(0,T)\subset L^\infty(0,T)$, we have
\[
\|\tw^0 (\ty_1)- \tw^0(\ty_2)\|_{L^\infty(0,T)} \leq C \left( \|\tw^0 (\ty_1)- \tw^0(\ty_2)\|_{L^2(0,T)} + \|\p_t (\tw^0 (\ty_1)- \tw^0(\ty_2))\|_{L^2(0,T)}\right).
\]
Let us start with the first term in the right-hand side. Using a Taylor formula,
\[
\tw^0 (\ty_1(t))- \tw^0(\ty_2(t))=(\ty_1(t)-\ty_2(t))\int_0^1 \p_x \tw^0\left(\tau \ty_1(t) + (1-\tau) \ty_2(t)\right)\:d\tau.
\]
For a.e. $t\in [0,T]$, the Cauchy-Schwarz inequality implies
\[
|\ty_1(t)-\ty_2(t)|=\left| \int_0^t (\ty_1'-\ty_2')\right|\leq t^{1/2}\|\ty_1'-\ty_2'\|_{L^2(0,T)}.
\]
Hence, setting $z=\tau \ty_1(t) + (1-\tau) \ty_2(t)$ and observing that $z/M\leq t\leq Mz$,
\begin{eqnarray*}
\|\tw^0 (\ty_1)- \tw^0(\ty_2)\|_{L^2(0,T)}^2&\leq &\|\ty_1'-\ty_2'\|_{L^2(0,T)}^2 \int_0^T\int_0^1 t \left(\p_x \tw^0\left(\tau \ty_1(t) + (1-\tau) \ty_2(t)\right)\right)^2\:d\tau\:dt\\
&\leq &M^2 \|\ty_1'-\ty_2'\|_{L^2(0,T)}^2 \int_0^{MT}\int_0^1 z (\p_x \tw^0(z))^2 \:d\tau\:dz\\
&\leq & M^2\|\ty_1'-\ty_2'\|_{L^2(0,T)}^2\| \sqrt{z} \p_x \tw^0(z)\|_{L^2(\R_+)}^2.
\end{eqnarray*}
The second term is treated in a similar fashion.

\end{proof}

\section{Proof of Lemma \ref{lem:coerc-A}}
\label{app:lem:coerc-A}

We recall that 
\[
\cA=-s\mathrm{Id } - \mu \p_x \left(\frac{\cdot}{\bv}\right).
\]
Let $\varphi \in H^2(\R_+)$ be arbitrary. 
\begin{itemize}
    \item Integrating by parts and recalling that $\bv(0)=1$,
    \begin{eqnarray*}
    \int_0^\infty (\cA \p_x \varphi) \varphi&=& -    \int_0^\infty s\p_x \varphi\;  \varphi - \mu \int_0^\infty\p_x \left(\frac{\p_x \varphi}{\bv}\right)\varphi\\
    &=&\frac{s}{2} (\varphi(0))^2 + \mu \int_0^\infty \frac{(\p_x \varphi)^2}{\bv} + \mu \varphi(0) \varphi'(0).
        \end{eqnarray*}
    
    \item In a similar fashion, for any $\rho\in \mathcal C_b^2(\R_+)$,
    \begin{eqnarray*}
    \int_0^\infty (\p_x\cA  \varphi) \frac{\varphi}{\bv} \rho &=& -s\int_0^\infty \p_x \varphi \; \frac{\varphi}{\bv} \rho - \mu \int_0^\infty \left(\p_x^2\left(\frac{\varphi}{\bv}\right)\right) \left(\frac{\varphi}{\bv}\right) \rho\\
    &=& \frac{s}{2} \varphi(0)^2 \rho(0) + \frac{s}{2} \int_0^\infty \varphi^2 \p_x \left(\frac{\rho}{\bv}\right) + \mu \int_0^\infty \left(\p_x\left(\frac{\varphi}{\bv}\right)\right)^2 \rho\\
    && + \mu \int_0^\infty \left(\p_x\left(\frac{\varphi}{\bv}\right)\right)\left(\frac{\varphi}{\bv}\right) \p_x \rho + \mu \p_x \left(\frac{\varphi}{\bv}\right)_{|x=0} \varphi(0)\rho(0)\\
    &=&  \mu \int_0^\infty \left(\p_x\left(\frac{\varphi}{\bv}\right)\right)^2 \rho + \frac{s}{2} \int_0^\infty \varphi^2 \p_x \left(\frac{\rho}{\bv}\right)\\
    &&- \frac{\mu}{2}\int_0^\infty \left(\frac{\varphi}{\bv}\right)^2 \rho''\\
    && +\frac{s}{2} \varphi(0)^2 \rho(0) - \frac{\mu}{2}\varphi(0)^2 \rho'(0)  + \mu \p_x \left(\frac{\varphi}{\bv}\right)_{|x=0} \varphi(0)\rho(0).
    \end{eqnarray*}
\end{itemize}
The result follows.
\bibliography{bibli-DP}
\end{document}